\documentclass[reqno]{amsart}

\input xy
\xyoption{all}
\usepackage{accents}
\usepackage{epsfig}
\usepackage{color}
\usepackage{amsthm}
\usepackage{amssymb}
\usepackage{amsmath}
\usepackage{amscd}
\usepackage{amsopn}
\usepackage{graphicx}
\usepackage{centernot}

\usepackage{xspace}

\usepackage{url}
\usepackage{enumitem, hyperref}\hypersetup{colorlinks}

\usepackage{color} 

\definecolor{darkred}{rgb}{1,0,0} 
\definecolor{darkgreen}{rgb}{0,0.6,0}
\definecolor{darkblue}{rgb}{0,0,.8}

\hypersetup{colorlinks,
linkcolor=darkblue,
filecolor=darkgreen,
urlcolor=darkred,
citecolor=darkgreen}

\makeatletter
\def\reflb#1#2{\begingroup
    #2%
    \def\@currentlabel{#2}%
    \phantomsection\label{#1}\endgroup
}
\makeatother

\numberwithin{equation}{section}
\newtheorem {Theorem}{Theorem}
\numberwithin{Theorem}{section}

\newtheorem {Lemma}[Theorem]    {Lemma}

\newtheorem {Proposition}[Theorem]{Proposition}
\newtheorem {Corollary}[Theorem]{Corollary}
\theoremstyle{definition}
\newtheorem{Definition}[Theorem]{Definition}
\theoremstyle{remark}
\newtheorem{Remark}[Theorem]{Remark}
\newtheorem{Example}[Theorem]{Example}

\def    \eps    {\epsilon}

\newcommand{\CA}{{\mathcal A}}

\newcommand{\CI}{{\mathcal I}}

\newcommand{\CM}{{\mathcal M}}
\newcommand{\CN}{{\mathcal N}}
\newcommand{\CU}{{\mathcal U}}

\newcommand{\CS}{{\mathcal S}}

\newcommand{\supp}{\operatorname{supp}}
\newcommand{\gap}{\operatorname{gap}}
\newcommand{\sgn}{\operatorname{sgn}}
\newcommand{\sign}{\operatorname{sign}}
\newcommand{\lcm}{\operatorname{lcm}}

\newcommand{\const}{{\mathit const}}

\newcommand{\ff}{{\mathfrak f}}

\newcommand{\tu}{\tilde{u}}

\newcommand{\tF}{\tilde{F}}

\newcommand{\hmu}{\hat{\mu}}

\newcommand{\tPhi}{\tilde{\Phi}}
\newcommand{\tH}{\tilde{H}}
\newcommand{\tQ}{\tilde{Q}}

\newcommand{\Bb}{{\mathcal B}}
\newcommand{\CB}{{\mathcal B}}
\newcommand{\Cc}{{\mathcal C}}

\newcommand{\PP}{{\mathcal P}}

\def    \C      {{\mathbb C}}
\def    \R      {{\mathbb R}}

\def    \Z      {{\mathbb Z}}
\def    \N      {{\mathbb N}}
\def    \Q      {{\mathbb Q}}

\def    \T      {{\mathbb T}}
\def    \CP     {{\mathbb C}{\mathbb P}}

\def    \12    {{\frac{1}{2}}}

\def    \p      {\partial}
\def    \bp      {\bar{\partial}}

\def    \ev  {\operatorname{\mathit{ev}}}

\def    \im     {\operatorname{im}}

\def    \SH     {\operatorname{SH}}
\def    \Sp     {\operatorname{Sp}}
\def    \TSp     {\widetilde{\operatorname{Sp}}}
\def    \TSpn    {\widetilde{\operatorname{Sp}}{}^*}
\def    \U     {\operatorname{U}}

\def    \HF     {\operatorname{HF}}
\def    \HC     {\operatorname{HC}}

\def    \CCM     {\operatorname{CM}}

\def    \Gr     {\operatorname{Gr}}
\def    \H     {\operatorname{H}}

\def    \HM     {\operatorname{HM}}

\def    \CF     {\operatorname{CF}}

\def    \PD     {\operatorname{PD}}

\def    \barf   {\bar{f}}

\def    \MURS  {\operatorname{\mu_{\scriptscriptstyle{RS}}}}

\def    \MUM  {\operatorname{\mu_{\scriptscriptstyle{M}}}}
\def    \mum  {\operatorname{\mu_{\scriptscriptstyle{M}}}}
\def    \s  {\operatorname{c}}
\def    \hs  {\hat{\operatorname{c}}}

\begin{document}


\setlength{\smallskipamount}{6pt}
\setlength{\medskipamount}{10pt}
\setlength{\bigskipamount}{16pt}





\title[Lusternik--Schnirelmann Theory and Closed Reeb
Orbits]{Lusternik--Schnirelmann Theory and Closed Reeb Orbits}

\author[Viktor Ginzburg]{Viktor L. Ginzburg}
\author[Ba\c sak G\"urel]{Ba\c sak Z. G\"urel}

\address{BG: Department of Mathematics, University of Central Florida,
  Orlando, FL 32816, USA} \email{basak.gurel@ucf.edu}

\address{VG: Department of Mathematics, UC Santa Cruz, Santa Cruz, CA
  95064, USA} \email{ginzburg@ucsc.edu}



\date{\today} 

\thanks{The work is partially supported by NSF CAREER award
  DMS-1454342 (BG) and NSF grants DMS-1414685 (BG) and DMS-1308501
  (VG)}

\bigskip

\begin{abstract}
  We develop a variant of Lusternik--Schnirelmann theory for the shift
  operator in equivariant Floer and symplectic homology. Our key
  result is that the spectral invariants are strictly decreasing under
  the action of the shift operator when periodic orbits are
  isolated. As an application, we prove new multiplicity results for
  simple closed Reeb orbits on the standard contact sphere, the unit
  cotangent bundle to the sphere and some other contact manifolds. We
  also show that the lower Conley--Zehnder index enjoys a certain
  recurrence property and revisit and reprove from a different
  perspective a variant of the common jump theorem of Long and
  Zhu. This is the second, combinatorial ingredient in the proof of
  the multiplicity~results.
\end{abstract}

\maketitle

\tableofcontents

\section{Introduction and main results}
\label{sec:intro+results}

\subsection{Introduction}
\label{sec:intro} 
In this paper we develop a variant of Lusternik--Schnirelmann theory
for the shift operator in equivariant Floer and symplectic homology
and show that the spectral invariants are strictly decreasing under
the action of the shift operator when periodic orbits are isolated. As
an application of this theory, we obtain new multiplicity results for
simple (i.e., un-iterated) closed Reeb orbits on, e.g., $S^{2n-1}$ and
$ST^*S^n$. We also establish, as a second, combinatorial ingredient in
the proof of these multiplicity results, a recurrence property of the
lower Conley--Zehnder index and use it to reprove from a different
perspective a variant of the common jump theorem from \cite{LZ}; see
also \cite{DLW}.

\subsubsection{Lusternik--Schnirelmann theory}
\label{sec:LS-intro-0}
This theory, \cite{LS}, hinges on the general principle that the
minimax critical values associated with homology classes decrease
under cohomology operations. Moreover, they decrease strictly when the
critical sets are sufficiently small, e.g., when the critical points
are isolated. This principle, which we sometimes call the
Lusternik--Schnirelmann inequality, applies to a broad class of
functionals and some (perhaps, many) cohomology operations. In fact,
the authors are not aware of any instance where the principle would
convincingly break down.  An alternative way to think about the
principle is that although the local (co)homology of an isolated
critical point can be arbitrarily large, it does not support
non-trivial cohomology operations such as the cup product or Massey
products. The Lusternik--Schnirelmann inequality is usually applied to
a chain of homology classes to establish a lower bound on the number
of critical values, and hence critical points, of a functional. Such
lower bounds do not require any non-degeneracy conditions, but are
weaker than those coming from Morse theory.

Here, referring the reader to, e.g., \cite[Sec.\ 6]{GG:gaps} and
\cite{Vi97} for a more thorough treatment of the question and
references, we only illustrate the point by a few simple examples.

The most elementary variant of the theory is the
Lusternik--Schnirelmann theorem for smooth functions, giving a lower
bound for the number of critical values of a function with isolated
critical points via the cup-length of its domain.  Ultimately, it is
based on the fact that the minimax value associated with the intersection
product $v\cap w$ of two homology classes is strictly smaller than the
minimax values for both $v$ and $w$, provided of course that the
critical points are isolated and neither $v$ nor $w$ is the
fundamental class.  Note that this variant of the theory entirely
bypasses the notion of the Lusternik--Schnirelmann category.  The
classical Lusternik--Schnirelmann theorem on the existence of three
simple geodesics on $S^2$ is another incarnation of the same principle
applied now to the length functional,~\cite{Gr}.

In Floer theory the minimax critical values are usually referred to as
spectral invariants, and Lusternik--Schnirelmann theory is at the very
heart of the proof of the known cases of the Arnold conjecture for
degenerate Hamiltonians; see, e.g., \cite{Fl1, Fl2, Ho, How:thesis,
  LO, Sc}. Here the functional is the action functional and the
(co)homology operation is the pair-of-pants product or the action of
the quantum homology on the Floer homology. (See \cite[Sec.\
6.2]{GG:gaps} for a detailed discussion of the Hamiltonian Arnold
conjecture in this context and further references.)

When the underlying space is equipped with an $S^1$-action, the
pairing with the generator of $\H^2(\CP^\infty)$ gives rise to an
operator $D$ of degree $-2$ on the $S^1$-equivariant homology. We call
$D$ the shift operator. When the functional is $S^1$-invariant, one
can expect a variant of Lusternik--Schnirelmann theory to hold for
$D$; cf.\ \cite{Gi:sl}. For instance, for a smooth $S^1$-invariant
function, the critical values associated with equivariant homology
classes over $\Q$ are strictly decreasing under the action of $D$ when
the $S^1$-action is locally free; see Section \ref{sec:LS-Morse}. One
can think of this fact as simply the Lusternik--Schnirelmann theorem
for the quotient orbifold.

The shift operator $D$ in the equivariant symplectic homology is the
connecting map in the Gysin exact sequence relating the equivariant
and non-equivariant symplectic homology. It is defined in
\cite{BO:Gysin} via counting Floer trajectories matching marked points
on the orbits. One of the main results of the present paper (Theorem
\ref{thm:LS1-intro}) is the Lusternik--Schnirelmann theorem for the
equivariant symplectic homology. This result is an easy consequence of
the Lusternik--Schnirelmann inequality for the equivariant Floer
homology of autonomous Hamiltonians (Theorem
\ref{thm:LS-Floer-intro}). One point of independent interest that
arises in the proof is that while the original definition of $D$ is
well adapted for the construction of the Gysin exact sequence, it is
not perfectly suited for the proof of the Lusternik--Schnirelmann
inequality. Hence we give a different definition of $D$ via the
intersection index with the ``hyperplane-section'' cycle, and prove
the equivalence of the two definitions in Section
\ref{sec:LS-Floer-D}. A local counterpart of these results is the fact
that $D\equiv 0$ on the level of the local symplectic or Floer
homology of an isolated closed orbit (non-constant in the Hamiltonian
case); see Proposition \ref{prop:D-local} and Corollary
\ref{cor:D-local}.

Applying the Lusternik--Schnirelmann theorem to Reeb flows on the
standard contact $S^{2n-1}$ and $ST^*S^n$, we obtain sequences of
closed Reeb orbits with strictly decreasing actions, satisfying
certain index constraints; see Theorems \ref{thm:LS2-intro} and
\ref{thm:STSn-intro}. These theorems are central to the proofs of our
multiplicity results (Theorems \ref{thm:mult-intro} and
\ref{thm:mult4-intro}) discussed in more detail below.

Lusternik--Schnirelmann theory for the shift operator in equivariant
symplectic homology has some closely related predecessors which have
also been used to obtain lower bounds on the number of simple periodic
orbits. For instance, in the framework of the classical calculus of
variations, we have the Lusternik--Schnirelmann theorem for the shift
operator in a suitably defined equivariant homology for convex
Hamiltonians on $\R^{2n}$, based on the Fadell--Rabinowitz index,
\cite{FR}; see, e.g., \cite[Sec.\ V.3]{Ek} and also \cite[Chap.\
15]{Lo}. Also, there is a similar result for the energy functional on
the space of closed loops yielding lower bounds for the number of
closed geodesics on, say, $S^n$; see \cite{BTZ:JDG83}. More recently,
the Lusternik--Schnirelmann inequality for $D$ in ECH was used to
prove the existence of at least two simple periodic Reeb orbits on
every closed contact three-manifold, \cite{CGH}. In this case, the
nature of the inequality is particularly transparent because the shift
operator is defined by counting holomorphic curves ``passing'' through
a fixed point. Placing this point away from the relevant periodic
orbits, one obtains a lower bound on the energy of the holomorphic
curve by the monotonicity lemma.

\subsubsection{Multiplicity results for closed Reeb orbits and index theory}
\label{sec:mult-intro-0}
As has been mentioned above, our main applications of
Lusternik--Schnirelmann theory are to multiplicity results for simple
closed Reeb orbits (Theorems \ref{thm:mult-intro} and
\ref{thm:mult4-intro}). Arguably, in this class of questions beyond
dimension three, the most interesting manifolds are the standard
contact $S^{2n-1}$ and $ST^*S^n$, i.e., the boundaries of star-shaped
domains in $\R^{2n}$ and $T^*S^n$. For these manifolds, on the one
hand, we have strong benchmark multiplicity results for simple closed
characteristics on convex hypersurfaces and closed geodesics of
Riemannian or Finsler metrics, which we will discuss shortly, and, on
the other, we have precise general conjectures on the minimal number
of simple periodic orbits.

Namely, hypothetically, any Reeb flow on the standard contact
$S^{2n-1}$ has at least $r_{\min}=n$ simple closed Reeb orbits. For
$ST^*S^n$, we have $r_{\min}=n$ when $n$ is even and $r_{\min}=n+1$
when $n$ is odd. In both cases, the lower bounds, which we will refer
to as the \emph{multiplicity conjectures}, are sharp. For $S^{2n-1}$,
the example is an irrational ellipsoid and for $ST^*S^n$ the
Katok--Ziller Finsler metrics have exactly $r_{\min}$ simple closed
geodesics, \cite{Ka,Zi}. A closely related fact is that for both
manifolds the positive equivariant symplectic homology is rather
small. In particular, in contrast with most other cotangent bundles
and some other contact manifolds, there is no homological growth which
would imply the existence of infinitely many simple periodic orbits;
see, e.g., \cite{CH,GM,HM,McL}. Nor do we have in this case the
contact Conley conjecture type phenomena; see \cite{GGM,GG:CC-survey}.

The existence of at least two simple closed Reeb orbits on $S^3$ and
$ST^*S^2$ is a very particular case of the theorem from \cite{CGH}
mentioned above. This result was also proved, specifically for $S^3$
and $ST^*S^2$, in \cite{GHHM,GGo,LL} by different methods; see also
\cite{BL} for the case of Finsler metrics on $S^2$.

When $n\geq 3$, the multiplicity conjectures are completely open. In
general, it is not even known if there are at least two simple closed
Reeb orbits on $S^{2n-1}$ or $ST^*S^n$, i.e., $r_{\min}\geq 2$. (It is
easy to show that this is true in the non-degenerate case, \cite[Rmk.\
3.3]{Gu:pr}.)

However, once certain convexity or index conditions are imposed on the
contact form, the question becomes much more tractable.  Of course,
convexity is not invariant under symplectomorphisms, and in the
context of symplectic topology it is usually replaced by dynamical
convexity introduced in \cite{HWZ:convex} for $\R^4$. A contact form
on $S^{2n-1}$ is said to be dynamically convex if every closed Reeb
orbit has Conley--Zehnder index $\mu\geq n+1$ or, without
non-degeneracy assumptions, $\mu_-\geq n+1$ where $\mu_-$ is the lower
Conley--Zehnder index; see Section \ref{sec:index}.  This lower bound
on $\mu_-$, which follows from geometrical convexity, appears to be a
suitable replacement for convexity of hypersurfaces in $\R^{2n}$. For
$ST^*S^n$, the requirement is that $\mu_-\geq n-1$, which is a
consequence of certain curvature pinching conditions if the contact
form comes from a Riemannian or Finsler metric on $S^n$; see, e.g.,
\cite{HP,Ra:pinching,Wa12,DLW}.

The multiplicity conjectures for non-degenerate contact forms on
$S^{2n-1}$ and $ST^*S^n$ are proved in
\cite{AM,DLLW,DLW,GK,Lo,LZ,Wa,WHL} under the above or even much weaker
index (or convexity) requirements. The proofs are usually based on a
Morse theoretic methods using the size of the relevant homology to
obtain lower bounds on the number of simple orbits. We review these
results in detail in the next section and, in fact, reprove some of
them in Section \ref{sec:pf-main}.

However, our main interest in this paper is the general case of
possibly degenerate contact forms, and this is where
Lusternik--Schnirelmann theory becomes absolutely essential. We prove
that when $\mu_-\geq n+1$ for $S^{2n-1}$ or $\mu_-\geq n-1$ for
$ST^*S^n$, the Reeb flow has at least $r=\lceil n/2\rceil +1$ simple
periodic orbits; see Theorems \ref{thm:mult-intro} and
\ref{thm:mult4-intro}. Leaving a comparison with previous results to
the next section, we only mention here that this lower bound for
$S^{2n-1}$ is a generalization of the lower bounds established in
\cite{LZ,Wa16} for convex hypersurfaces in $\R^{2n}$. We also obtain
some additional information about the orbits when the Reeb flow has
only finitely many simple closed orbits. In particular, we show that
then the orbits satisfy certain action--index resonance relations,
generalizing the relations from \cite{Gu:pr}, and that at least one of
the orbits is elliptic; see Section \ref{sec:pf-main}.

The key difficulty intrinsic in this class of problems is that neither
Lusternik--Schnirelmann nor Morse theory can distinguish simple from
iterated orbits. In other words, while both theories do provide lower
bounds on the number of periodic orbits in a given action or index
range, these orbits need not be geometrically distinct and can in
general be iterations of just one simple orbit. To circumvent this
difficulty, we use what is essentially a combinatorial argument
closely related to the common jump theorems proved in \cite{DLW,LZ}.
(See Section \ref{sec:IR} where we establish a certain index
recurrence property (Theorems \ref{sec:IRT} and \ref{thm:IRT2}) and
prove a variant of the common jump theorem.)  There is a considerable
overlap between Section \ref{sec:IR} and the results from those
papers.  Our treatment of the problem is, however, self-contained and
quite different from \cite{DLW,Lo,LZ}, and the proofs are relatively
straightforward.

\subsection{Results}
\label{sec:results}
Now we are ready to state precisely some of the key results of the
paper starting with Lusternik--Schnirelmann theory for equivariant
symplectic and Floer homology; see Sections \ref{sec:LS} and
\ref{sec:LS2} for more details.

Let $(M^{2n-1},\alpha)$ be a closed contact manifold and let $W$ be an
exact strong symplectic filling of $M$ with $c_1(TW)|_{\pi_2(W)}=0$;
see Section \ref{sec:SH-conditions}. Abusing notation, let us denote
the positive equivariant symplectic homology of $W$ over $\Q$ by
$\SH_*^{G,+}(W;\Q)$, where $G=S^1$, graded by the Conley--Zehnder
index. Let $D\colon \SH_*^{G,+}(W)\to \SH_{*-2}^{G,+}(W)$ be the shift
operator introduced in \cite{BO:Gysin}. There is a spectral invariant
$\s_w(\alpha)\in (0,\,\infty)$, a point in the action spectrum
$\CS(\alpha)$, associated to any non-zero element
$w\in \SH_*^{G,+}(W)$. (By definition, $\s_0(\alpha)=-\infty$.)

Denote by $\PP(\alpha)$ the collection of contractible in $W$ closed Reeb
orbits (not necessarily simple) of $\alpha$. An orbit
$x\in\PP(\alpha)$ is said to be \emph{isolated} if it is isolated in
the extended phase space $M\times (0,\,\infty)$, i.e., there exists a
tubular neighborhood $U$ of the image of $x$ in $M$ such that no other
closed Reeb orbit with period sufficiently close to the period of $x$
intersects $U$.

One of our main results is

\begin{Theorem}[Lusternik--Schnirelmann inequality; symplectic
  homology]
\label{thm:LS1-intro}
Assume that all orbits in $\PP(\alpha)$ are isolated.  Then, for any
non-zero element $w\in \SH_{*}^{G,+}(W)$, we have
$$
\s_w(\alpha)>\s_{D(w)}(\alpha).
$$
\end{Theorem}

This theorem, proved in Section \ref{sec:LS-SH-result}, readily
follows from a similar inequality for the equivariant Floer
homology. Let $H$ be an autonomous Hamiltonian on a symplectic
manifold $V$ which is symplectically aspherical and either closed or
the symplectic completion of a compact domain $W$ with contact type
boundary. (In the latter case, $H$ is also required to be admissible
at infinity.) We denote the filtered $G=S^1$-equivariant Floer
homology of $H$ over $\Q$ for an interval $I$ by
$\HF_{*}^{G,I}(H;\Q)$. In this case we also have a shift operator
$D\colon \HF_{*}^{G,I}(H;\Q)\to \HF_{*-2}^{G,I}(H;\Q)$ (see Section
\ref{sec:LS-Floer}), and to every non-zero class
$w\in \HF_{*}^{G,I}(H;\Q)$ we can associate a spectral invariant
$\s_w(H)\in \CS(H)\cap I$, where $\CS(H)$ is the action spectrum of
$H$.

\begin{Theorem}[Lusternik--Schnirelmann inequality; Floer homology]
\label{thm:LS-Floer-intro}
Assume that all contractible one-periodic orbits of $H$ with action in
$I$ are non-constant and isolated. Then for any non-zero class
$w\in \HF_{*}^{G,I}(H;\Q)$, we have
$$
\s_w(H)>\s_{D(w)}(H).
$$
\end{Theorem}

This theorem is proved in Section \ref{sec:LS-Floer-pf}. Note also
that in both Theorems \ref{thm:LS1-intro} and
\ref{thm:LS-Floer-intro}, the non-strict inequalities hold without any
additional assumptions on the periodic orbits and readily follow from
the definitions.

As has been mentioned in the introduction, to prove Theorem
\ref{thm:LS-Floer-intro} we give an alternative definition of $D$ for
which the proof of the Lusternik--Schnirelmann inequality is more
natural and show in Section \ref{sec:LS-Floer-D} that this definition
is equivalent to the one from \cite{BO:Gysin}. Another consequence of
the proofs of Theorems \ref{thm:LS1-intro} and
\ref{thm:LS-Floer-intro} is that the operator $D$ vanishes in the
local symplectic or Floer homology over $\Q$ of an isolated
non-constant periodic orbit; see Sections \ref{sec:local-MB} and
\ref{sec:LS-SH-result}.

Next, let us apply these results to the equivariant symplectic
homology of $S^{2n-1}$ and $ST^*S^n$.  Let $\alpha$ be a contact form
on $M=S^{2n-1}$ supporting the standard contact structure. Then
$(M,d\alpha)$ can be symplectically embedded as a hypersurface in
$\R^{2n}$ bounding a star-shaped domain $W$. As is well known, there
exists a sequence of elements $w_k\in \SH_{n+2k-1}^{G,+}(W)$,
$k\in\N$, such that $Dw_{k+1}=w_k$. Set $\s_k:=\s_{w_k}$ and denote
by $\hmu(y)$ the mean index of $y\in\PP(\alpha)$ and, when $y$ is
isolated, by $\SH_*^G(y;\Q)$ the local equivariant symplectic homology
of $y$; see Section \ref{sec:LS-SH-result}.  Applying Theorem
\ref{thm:LS1-intro} to the classes $w_k$, we obtain
 
\begin{Theorem}
\label{thm:LS2-intro}
Assume that all closed Reeb orbits of $\alpha$ are isolated.  Then
$$
\s_1(\alpha)<\s_2(\alpha)<\s_3(\alpha)<\cdots.
$$
As a consequence, there exists an injection
$$
\psi\colon \N\to \PP(\alpha),\quad k\mapsto y_k
$$
called a \emph{carrier map} such that 
$$
\CA_{\alpha}(y_1)<\CA_{\alpha}(y_2)<\CA_{\alpha}(y_3)<\cdots
$$
and $\SH_*^G(y_k;\Q)\neq 0$ in degrees $*=n+2k-1$. In particular,
$$
|\hmu(y_k)-(n+2k-1)|\leq n-1.
$$
\end{Theorem}

This theorem is a symplectic homology analog of the
Lusternik--Schnirelmann inequalities for Clarke's dual action
functional; see \cite{Ek,Lo} and references therein. Its slightly more
general version is proved in Section \ref{sec:displ} as Corollary
\ref{cor:sphere}.

A similar result holds for $ST^*S^n$ but the chain of inequalities has
length $n$. Namely, let $\alpha$ be a contact form supporting the
standard contact structure on $M=ST^*S^n$. We can treat $(M,d\alpha)$
as the boundary of a fiber-wise star-shaped domain $W$ in $T^*S^n$. By
Proposition~\ref{prop:D-STS}, for every $j\in\N$ there exist $n$
non-zero elements $w_i\in \SH_{2i+(2j-1)(n-1)}^{G,+}(W;\Q)$,
$i=0,\,\ldots,\,n-1$, such that $Dw_{i+1}=w_i$.

\begin{Theorem}
\label{thm:STSn-intro}
Assume that all orbits in $\PP(\alpha)$ are isolated. Then, for every
$j\in \N$, there exist $n$ periodic orbits $y_0,\ldots,y_{n-1}$ of the
Reeb flow of $\alpha$ such that $\s_{w_i}(\alpha)=\CA_\alpha(y_i)$ and
$\SH_*^G(y_i;\Q)\neq 0$ in degrees $*=(2j-1)(n-1)+2i$. In particular,
$$
\CA_{\alpha}(y_0)< \CA_{\alpha}(y_1)<\cdots< \CA_{\alpha}(y_{n-1}),
$$
and
$$
\big|\hmu(y_i)-\big((2j-1)(n-1)+2i\big)\big|\leq n-1.
$$
\end{Theorem}
This result is proved, in a slightly more general form, in Section
\ref{sec:T^*S^n-1} as Corollary~\ref{cor:STSn}.  The non-strict
inequalities in both of the theorems hold without any assumptions on
the orbits. Furthermore, these theorems and the multiplicity results
below readily extend with straightforward modifications to several
other classes of contact manifolds and Liouville domains. Among these
are, for instance, displaceable (e.g., subcritical) Liouville domains;
see Section \ref{sec:displ} and Remark \ref{rmk:displ-mult}.

Let us now turn to the multiplicity results for simple closed Reeb
orbits. Although our main focus is on the general setting where the
contact form can be degenerate and Lusternik--Schnirelmann theory is
essential, we also consider for the sake of completeness the
non-degenerate case where stronger results can usually be obtained
by other methods.

We start with lower bounds on the number of simple closed Reeb orbits
on $S^{2n-1}$.

\begin{Theorem}
\label{thm:mult-intro}
Let $\alpha$ be a dynamically convex contact form on $S^{2n-1}$
supporting the standard contact structure.  Then the Reeb flow of
$\alpha$ has at least $r=\lceil n/2\rceil +1$ simple closed
orbits. When $\alpha$ is non-degenerate, $r=n$. Moreover, assume in
addition that there are only finitely many geometrically distinct
closed orbits. Then these $r$ simple orbits $x$ can be chosen so that
the ratios $\CA(x)/\hmu(x)$ are the same for all of them (the
resonance relations) and, when $\alpha$ is non-degenerate, all $r$
simple orbits are even.
\end{Theorem}

This theorem is an immediate consequence of Theorem \ref{thm:mult2}
(multiplicity) and Theorem \ref{thm:RR} (resonance relations). The
part of the theorem on the resonance relations generalizes the
relations for perfect Reeb flows from \cite[Thm.\ 1.2]{Gu:pr}.

Except for the resonance relations, the non-degenerate case of the
theorem is not new and included only for the sake of completeness and
because the proof requires no extra work. For strictly convex
hypersurfaces in $\R^{2n}$ the result goes back to \cite{LZ} and
\cite{Lo}. (Here convexity is understood in a very strong sense: the
(outward) second fundamental form is negative definite at every
point.)  For dynamically convex hypersurfaces, the non-degenerate case
of the theorem is proved in \cite{AM,GK} under the slightly less
restrictive index condition that $\mu\geq n-1$. More recently,
in \cite{DLLW}, this index requirement is relaxed even further.

The degenerate case of the theorem is new. For convex hypersurfaces, a
similar result is proved in \cite{LZ} with the same lower bound when
$n$ is even and the lower bound $\lfloor n/2\rfloor +1$ when $n$ is
odd, and the case of a convex hypersurface in $\R^{6}$ is treated in
\cite{WHL}. Recently, again for convex hypersurfaces, the lower bound
exactly as in Theorem \ref{thm:mult-intro} has been established in
\cite{Wa16} for all $n$. The existence of at least two distinct orbits
in the dynamically convex case is proved in \cite{AM}.

To the best of our understanding, Theorem \ref{thm:mult-intro}
incorporates essentially all, but one, results to date on the
multiplicity of simple closed Reeb orbits on convex or dynamically
convex hypersurfaces in $\R^{2n}$ without additional assumptions such
as non-degeneracy, pinching or symmetry. The exception is a theorem
from \cite{Wa} asserting the existence of four closed characteristics
on a convex hypersurface in $\R^{8}$. This is the first dimension
where the lower bound from Theorem \ref{thm:mult-intro}, which for
$n=4$ is three, is below the multiplicity conjecture lower bound $n$.

Next, let us turn to contact forms $\alpha$ on $M=ST^*S^n$ supporting
the standard contact structure. The pair $(M,\alpha)$ is then the
boundary of a fiber-wise star-shaped domain in $T^*S^n$.  In this
setting, the right analog of the dynamical convexity condition is the
requirement that $\mu_-\geq n-1$ for all closed Reeb orbits of
$\alpha$. For the unit sphere bundle of a Riemannian or Finsler
metric, this requirement is satisfied, for instance, when the metric
meets certain curvature pinching conditions; see, e.g.,
\cite{AM:elliptic, DLW, HP, Ra:pinching, Wa12}. Along the lines of
Theorem \ref{thm:mult-intro}, we have the following result.

\begin{Theorem}
\label{thm:mult4-intro}
Let $\alpha$ be a contact form on $M=ST^*S^n$ supporting the standard
contact structure.  Assume that $\mu_-(y)\geq n-1$ for every closed
Reeb orbit $y$ on $M$.  Then $M$ carries at least $r$ simple closed
orbits, where $r=\lfloor n/2\rfloor-1$.  When $\alpha$ is
non-degenerate, we can take $r=n$ if $n$ is even and $r=n+1$ if $n$ is
odd.
\end{Theorem}

This theorem is restated as Theorem \ref{thm:mult4} and proved in
Section \ref{sec:T^*S^n-2}. Again, the non-degenerate case of the
theorem is not new and included only for the sake of completeness. A
much more general result is established in \cite{AM}. (However, the
argument in Section \ref{sec:T^*S^n-2} is self-contained.)

There are three main ingredients to the proofs of Theorems
\ref{thm:mult-intro} and \ref{thm:mult4-intro}. The first and the
major one is, of course, Lusternik--Schnirelmann theory discussed
extensively above. The second one, having considerable overlap with
the results in \cite{DLW,LZ,Lo}, is the combinatorial index analysis
from Section \ref{sec:IR}. These are sufficient to prove Theorem
\ref{thm:mult4-intro} and generalize the results of \cite{LZ} to
dynamically convex hypersurfaces, but not to refine those results.
The refinement in Theorem \ref{thm:mult-intro} comes from an
application of one of the key theorems from \cite{GHHM}. Namely, it
turns out that in every second dimension (odd $n$) a dynamically
convex Reeb flow on $S^{2n-1}$ with exactly $\lfloor n/2\rfloor +1$
geometrically distinct closed orbits must have a simple closed orbit
of a particular type, the so-called symplectically degenerate maximum
(SDM).  A Reeb flow with a simple SDM orbit necessarily has infinitely
many periodic orbits, \cite{GHHM}. Hence, for an odd $n$, there must
be at least $\lfloor n/2\rfloor +2=\lceil n/2\rceil +1$ closed Reeb
orbits.

There are several ways, completely within the scope of our methods, to
generalize Theorems \ref{thm:mult-intro} and \ref{thm:mult4-intro} and
other results of this type, which, although of interest, are either
not considered here at all or only discussed very briefly. For
instance, one can generalize Theorem \ref{thm:mult-intro} by making
the lower bound depend on the ``degree of non-degeneracy''. This is
the approach taken in \cite{LZ,Lo}. Secondly, the dynamical convexity
condition $\mu_-\geq n+1$ in Theorem \ref{thm:mult-intro} and the
condition $\mu_-\geq n-1$ in Theorem \ref{thm:mult4-intro} can also be
quantified and replaced by the condition $\mu_-\geq q$ yielding, in
general, weaker lower bounds on the number of simple closed
orbits. (See Theorems \ref{thm:mult3} and \ref{thm:mult5}. For
$S^{2n-1}$ and $q=n-1$, we recover \cite[Thm.\ 1.4]{GK}. In the
non-degenerate case, conceptually stronger results are now available;
see \cite{DLLW,DLW}.)

In the setting of Theorems \ref{thm:mult-intro} and
\ref{thm:mult4-intro}, one can draw some conclusions on the existence
of simple closed orbits of specific type (e.g., elliptic) when the
number of simple closed orbits is finite; cf.\ \cite{AM,GK,LZ,Lo}. We
mainly leave this question aside; see, however, Remark
\ref{rmk:elliptic}.  Generalizations of these theorems to some other
contact manifolds are noted in Remark \ref{rmk:displ-mult}.  Finally,
we touch upon the Ekeland--Lasry theorem from the perspective of
Lusternik--Schnirelmann theory for the shift operator; see Section
\ref{sec:generalizations} and, in particular, Corollary \ref{prop:EL}.

\subsection{Organization of the paper}
\label{sec:organization}
The rest of the paper is organized as follows.  In Section
\ref{sec:conventions} we specify in detail our conventions and
notation.

In Section \ref{sec:LS} we outline the construction of the equivariant
Floer homology and prove the Lusternik--Schnirelmann inequality for
the shift operator. This section carries the most technical and
conceptual load in the paper. We start by discussing the shift
operator and the Lusternik--Schnirelmann inequality in the equivariant
homology for $S^1$-actions in Section \ref{sec:LS-Morse}. Formally
speaking, the results and constructions from this section are never
used in the rest of the paper. However, on the conceptual level the
difference between Morse theory and Floer theory is purely technical
in this setting and we often refer to Section \ref{sec:LS-Morse} when
dealing with its Floer theoretic counterpart. In Section
\ref{sec:LS-Floer} we define equivariant Floer homology and the shift
operator for autonomous Hamiltonians and state some of our main
results.  Section \ref{sec:local-MB} is a digression on local
equivariant Floer homology and other localization constructions. The
main objective of Section \ref{sec:LS-Floer-pf} is the proof of the
Lusternik--Schnirelmann inequality in Floer homology.  This is the key
technical part of the paper. Finally, in Section
\ref{sec:quotient-Floer}, we show that the equivariant Floer homology
for Hamiltonians with non-constant periodic orbits can be interpreted
as the homology of a complex generated by periodic orbits by proving
that the underlying Morse--Bott spectral sequence collapses in the
$E^2$-term. This construction comes handy in the proofs of
multiplicity results in the non-degenerate case.

The main objective of Section \ref{sec:LS2} is to define the shift
operator in the equivariant symplectic homology and prove the
Lusternik--Schnirelmann inequalities. We do this in Section
\ref{sec:LS-sympl} where we also briefly discuss local equivariant
symplectic homology and its properties. Examples and applications to
$S^{2n-1}$, the unit cotangent bundle $ST^*S^n$ and some other contact
manifolds are worked out in Section~\ref{sec:LS-SH-examples}.

Sections \ref{sec:LS-Floer} and \ref{sec:LS-sympl} comprise an
introduction to equivariant Floer and symplectic homology, following
mainly \cite{BO:12,BO:Gysin} with some modifications.

In Section \ref{sec:index} we introduce some basic concepts from the
Conley--Zehnder index theory and (re)prove several auxiliary
results. For instance, in Section \ref{sec:index:DC} we give a simple
proof of the fact that convexity implies dynamical convexity. In
Section \ref{sec:IR} we establish the index recurrence theorem and a
variant of the common jump theorem, \cite{DLW,LZ}, which are both
essential for deriving multiplicity results from the
Lusternik--Schnirelmann inequalities. Sections \ref{sec:index} and
\ref{sec:IR} can be read independently of the rest of the paper, and
Section \ref{sec:index} is intended to be a reasonably self-contained,
although concise, introduction to index theory.

Finally, in Section \ref{sec:pf-main}, we prove the multiplicity
theorems for simple closed Reeb orbits and the action-index resonance
relations. We also discuss other related results including the
Ekeland--Lasry theorem.

\subsection{Conventions and notation}
\label{sec:conventions}
In this section we spell out the conventions and notation used
throughout the paper.

We usually assume that the underlying symplectic manifold
$(V^{2n},\omega)$ is symplectically aspherical, i.e.,
$[\omega]|_{\pi_2(V)}=0=c_1(TV)|_{\pi_2(V)}$, and either compact or
$V$ is the \emph{symplectic completion} of a Liouville domain
$(W,\omega=d\alpha)$. To be more specific, in the latter case, we have
$$
V=\widehat{W}=W\cup_{M} \big(M\times [1,\infty)\big)
$$ 
with the symplectic form $\omega$ extended to the cylindrical part as
$d(r\alpha)$, where $\alpha$ is a contact primitive of $\omega$ on
$M=\p W$ and $r$ is the coordinate on $[1,\infty)$.  (See Section
\ref{sec:SH-conditions} for more details.) As noted in Remarks
\ref{rmk:general-V} and~\ref{rmk:general-W}, these conditions can be
modified or relaxed.

We denote by $\PP(\alpha)$ the set of contractible in $W$ periodic
orbits of the Reeb flow on $(M^{2n-1},\alpha)$ and by $\CS(\alpha)$
the \emph{period or action spectrum}, i.e., the collection of their
periods or, equivalently, \emph{contact actions}.

The circle $S^1=\R/\Z$ plays several different roles throughout the
paper. We denote it by $G$ when we want to emphasize the role of the
group structure on~$S^1$.

The Hamiltonians $H$ on $V$ are always required to be one-periodic in
time, i.e., $H\colon S^1\times V\to \R$. In fact, most of the time the
Hamiltonians we are actually interested in are \emph{autonomous},
i.e., independent of time. The (time-dependent) Hamiltonian vector
field $X_H$ of $H$ is given by the Hamilton equation
$i_{X_H}\omega=-dH$. For instance, on the cylindrical part
$M\times [1,\infty)$ with $\omega=d(r\alpha)$ the Hamiltonian vector
field of $H=r$ is the Reeb vector field of $\alpha$.

We focus on contractible one-periodic orbits (or $k$-periodic with
$k\in\N$) of $H$. Such orbits can be identified with the critical
points of the \emph{action functional} $\CA_H\colon \Lambda\to\R$ on
the space $\Lambda$ of contractible loops $x$ in $V$ given by
\begin{equation}
\label{eq:action-0}
\CA_{H}(x) =\CA(x)-\int_{S^1} H(t,x(t))\,dt,
\end{equation}
where $\CA(x)$ is the symplectic area bounded by $x\in \Lambda$, i.e.,
the integral of $\omega$ over a disk with boundary $x$. (Modifications
needed to work with non-contractible orbits are discussed in Remarks
\ref{rmk:general-V} and \ref{rmk:general-W}.) When $V=\widehat{W}$, we
always require $H$ to be \emph{admissible}, i.e., to have the form
$H=\kappa r+c$, where $\kappa\not\in\CS(\alpha)$, outside a compact
set. Under this condition the Floer homology of $H$ is defined; see,
e.g., \cite{Vi:GAFA}. Note, however, that the homology depends on
$\kappa$.

The \emph{action spectrum} of $H$, i.e., the collection of action
values for all contractible one-periodic orbits of $H$, will be
denoted by $\CS(H)$.  When $H$ is autonomous, a one-periodic orbit $y$
is said to be a \emph{reparametrization} of $x$ if $y(t)=x(t+\theta)$
for some $\theta\in G=S^1$.  Two one-periodic orbits are said to be
\emph{geometrically distinct} if one of them is not a
reparametrization of the other. We denote by $\PP(H)$ the collection
of all geometrically distinct contractible one-periodic orbits of $H$
and by $\PP(H,I)$ the collection of such orbits with action in an
interval $I$.

With our sign conventions in the definitions of $X_H$ and $\CA_H$, the
Hamiltonian actions converge to the contact actions in the
construction of the symplectic homology; see Section
\ref{sec:SH-conditions}. In particular, $\CS(H)\to\CS(\alpha)$.

We normalize the \emph{Conley--Zehnder index}, denoted throughout the
paper by $\mu$, by requiring the flow for $t\in [0,\,1]$ of a small
positive definite quadratic Hamiltonian $Q$ on $\R^{2n}$ to have index
$n$. More generally, when $Q$ is small and non-degenerate, the flow has
index equal to $(\sgn Q)/2$, where $\sgn Q$ is the signature of $Q$;
see Section \ref{sec:index-prop}. In other words, the Conley--Zehnder
index of a non-degenerate critical point of a $C^2$-small autonomous
Hamiltonian $H$ on $V^{2n}$ is equal to $n-\MUM$, where $\MUM=\MUM(H)$
is the Morse index of $H$. The Floer homology and symplectic homology
are graded by the Conley--Zehnder index.

These action and index conventions differ by sign for both the action
and the index from the conventions in most of our recent papers (see,
e.g., \cite{GG:wm,GG:gap,GG:hyperbolic}) and to our taste are rather
awkward to use when $V$ is closed, although we still have
$\HF_*(H)\cong \H_{*+n}(V)$ in this case. (The reason is that $\CA_H$
restricted to the space $V\subset \Lambda$ of constant loops is
$-H$. Thus, for a $C^2$-small autonomous Hamiltonian $H$, the Floer
complex of $H$ graded by the Conley--Zehnder index is the Morse
complex of $-H$ graded by $n-\MUM(H)=\MUM(-H)-n$; see \cite{SZ}.)
However, we find these conventions convenient when $V$ is open, which
is the main focus of this paper.

The differential in the Floer or Morse complex is defined by counting
the downward Floer or Morse trajectories. As a consequence, a monotone
increasing homotopy of Hamiltonians induces a continuation map
preserving the action filtration in homology. (Clearly, $H\geq K$ on
$S^1\times V$ if and only if $\CA_H\leq \CA_K$ on $\Lambda$.) In many
instances in this paper, the coefficient ring in Floer or symplectic
homology has to have zero characteristic and then we take it to be
$\Q$. When the coefficient ring is immaterial, it is omitted from the
notation.

Our choice of signs in \eqref{eq:action-0} also has effect on the
Floer equation. The Floer equation is the $L^2$-anti-gradient flow
equation for $\CA_H$ on $\Lambda$ with respect to a metric
$\left<\cdot\,,\cdot\right>$ compatible with $\omega$:
$$
\p_s u =-\nabla_{L^2}\CA_H(u),
$$
where $u\colon \R\times S^1\to V$ and $s$ is the coordinate on
$\R$. Explicitly, this equation has the form
\begin{equation}
\label{eq:Floer-0}
\p_s u + J\p_t u=\nabla H,
\end{equation}
where $t$ is the coordinate on $S^1$. Here the almost complex
structure $J$ is defined by the condition
$\left<\cdot\,,\cdot\right>=\omega(J\cdot\,,\cdot)$ making $J$ act on
the first argument in $\omega$ rather than the second one, which is
more common, to ensure that the left hand side of the Floer equation
is still the Cauchy--Riemann operator $\bar\p_J$. (In other words,
$J=-J_0$, where $J_0$ is defined by acting on the second argument in
$\omega$, and $X_H=-J\nabla H$.) Note, however, that now the right
hand side of \eqref{eq:Floer-0} is $\nabla H$ with positive sign.

\subsection*{Acknowledgments}
The authors are grateful to Alberto Abbondandolo, Fr\'ederic
Bourgeois, River Chiang, Jean Gutt, Umberto Hryniewicz, Ely Kerman,
Leonardo Macarini, Marco Mazzucchelli, Yaron Ostrover, Jeongmin Shon,
and Otto van Koert for useful discussions and remarks. This project greatly
benefited from the Workshop on Conservative Dynamics and Symplectic
Geometry held at IMPA, Rio de Janeiro in August 2015. The authors
would like to thank the organizers (Henrique Bursztyn, Leonardo
Macarini, and Marcelo Viana) of the workshop.  A part of this work was
carried out while the first author was visiting National Cheng Kung
University, Taiwan, and he would like to thank NCKU for its warm
hospitality and support.

\section{Shift operator: Morse
  and Floer homology}
\label{sec:LS}

Our goal in this and the next sections is to develop
Lusternik--Schnirelmann theory for the shift operator in equivariant
Floer and symplectic homology. We do this with some redundancy in
several steps starting with Morse theory and then moving on to Floer
theory and, in Section \ref{sec:LS2}, to symplectic homology.

Recall that we denote the circle $S^1=\R/\Z$ by $G$ when it is treated
as a group rather than a manifold. Furthermore, unless the coefficient
ring is specified, a homology group can be taken with arbitrary
coefficients.  However, the choice of coefficients is essential in our
strict Lusternik--Schnirelmann inequality results: Theorems
\ref{thm:LS-Morse}, \ref{thm:LS-Floer}, \ref{thm:LS-SH}, and Corollary
\ref{cor:sphere}. In this case, the coefficient field must have zero
characteristic and we take to be $\Q$.

\subsection{Shift operator in equivariant Morse theory}
\label{sec:LS-Morse}

To set the stage for studying the shift operator in equivariant Floer
and symplectic homology, let us start with the toy model of Morse
theory and ordinary equivariant homology.

\subsubsection{Lusternik--Schnirelmann theory for the shift operator in
  equivariant homology}
\label{sec:LS-Morse-1}
Recall that for a $G=S^1$-principal bundle $\pi\colon P\to B$ we have
the Gysin exact sequence
\begin{equation}
\label{eq:Gysin}
\ldots \to
\H_*(P) \stackrel{\pi_*}{\to}
\H_*(B) \stackrel{D}{\to}
\H_{*-2}(B) \stackrel{\pi^!}{\to}
\H_{*-1}(P)
\to \ldots.
\end{equation}
Here the operator $\pi^!$ sends a class $[C]$ in $B$, where $C$ is a
cycle, to the class $[\pi^{-1}(C)]$ in $P$ and the operator $D$, the
one we are interested in, is given by the paring with $-c_1(\pi)$, the
negative first Chern class of the principle $G$-bundle $\pi$. In other
words, on the level of cycles, $D$ is the intersection product with
the Poincar\'e dual of $-c_1(\pi)$. (The negative sign here is not
particularly important; its role is to make this definition of $D$
equivalent to its standard Morse theoretic counterpart; see Section
\ref{sec:LS-Morse-2}.)

Next, consider a closed manifold $Y$ with an action of the circle
$G$. The equivariant homology of $Y$ is by definition the homology
$\H_*^G(Y)=\H_*(Y\times_G EG)$ where $Y\times_G EG:=(Y\times EG)/G$
and the product is equipped with the diagonal circle action; see,
e.g., \cite[Appen. C]{GGK} for more details.  We can take the limit of
odd-dimensional spheres $S^{2m+1}$ with the Hopf action as $EG$ and
$BG=\CP^\infty$. Since $EG$ is contractible, the Gysin exact sequence
\eqref{eq:Gysin} for the principal bundle
$\pi\colon Y\times EG\to Y\times_G EG$ takes the form
\begin{equation}
\label{eq:Gysin2}
\ldots \to
\H_*(Y) \stackrel{\pi_*}{\to}
\H_*^G(Y) \stackrel{D}{\to}
\H_{*-2}^G(Y) \stackrel{\pi^!}{\to}
\H_{*-1}(Y)
\to \ldots.
\end{equation}
As is easy to see, $c_1(\pi)$ is the pull-back of
the first Chern class of the universal bundle $EG\to BG=\CP^{\infty}$,
the generator of $\H^2(\CP^\infty)$, under the natural projection
$Y\times_G EG\to BG$, and $D$ is the pairing with $-c_1(\pi)$. We call
$D$ the \emph{shift operator in equivariant
  homology}. Alternatively, one can replace here the universal bundle
by the Hopf bundle $S^{2m+1}\to \CP^m$ and then pass to the limit as
$m\to\infty$.

To every non-zero $w\in \H_*^G(Y)$ and a smooth $G$-invariant function
$f\colon Y\to \R$ one can associate a \emph{spectral invariant} or a
\emph{critical value selector} $\s_w(f)$. This is done exactly as in
the non-equivariant case. For instance, we can set
$$
\s_w(f)=\inf\{a\in\R\mid w\in \im(i^a_*)\},
$$
where $i^a\colon \{f\leq a\}\to Y$ is the natural
inclusion. Alternatively, $\s_w(f)$ can be defined using the standard
minimax construction for cycles representing $w$. It is not hard to
show that $\s_w(f)$ is a critical value and $\s_w$ has all the
expected properties of critical value selectors; see, e.g.,
\cite[Sect. 6.1]{GG:gaps}. Finally, let us set $\s_0(f)=-\infty$.

\begin{Theorem}
\label{thm:LS-Morse}
Assume that the $G=S^1$-action on $Y$ is locally fee, i.e., the action
has no fixed points, and that the critical sets of $f$ are isolated
$G$-orbits and $w\neq 0$ in $\H^G_*(Y;\Q)$. Then
\begin{equation}
\label{eq:LS-Morse}
\s_w(f)>\s_{D(w)}(f).
\end{equation}
\end{Theorem}

This is the Lusternik--Schnirelmann inequality for the shift operator
in the equivariant homology.

\begin{Remark}
  The conditions that the action is locally free and that the critical
  sets of $f$ are isolated $G$-orbits are essential; as is the
  assumption that the homology is taken over $\Q$. Also note that here
  only the strict inequality requires a proof. For the non-strict
  inequality $\s_w(f)\geq \s_{D(w)}(f)$ is a general feature of
  critical value selectors, which holds without any restrictions on
  $f$ or on the coefficient ring; see \cite[Sect. 6.1]{GG:gaps}.
\end{Remark}

\begin{proof}
  First, note that since the sequence of graded vector spaces
  $\H_*(Y\times_G S^{2m+1})$ converging to $\H_*^G(Y)$ stabilizes in 
  every degree $*$, we can replace $EG$ in the construction by a
  sphere $S^{2m+1}$ of sufficiently high dimension. Let us pull-back
  $f$ to $P=Y\times S^{2m+1}$ and then push forward the resulting
  function to the smooth manifold $B=P/G$. We denote the push-forward
  by $\barf$. Clearly, $\barf$ is a smooth function on $B$ and the
  critical manifolds of $\barf$ are lens spaces. (A critical circle
  $S$ of $f$ gives rise to the critical manifold
  $C=S\times_G S^{2m+1}$ of $\barf$ diffeomorphic to the lens space
  $S^{2m+1}/\Gamma$, where $\Gamma$ is the stabilizer of the orbit $S$
  of $G$.)

  To prove \eqref{eq:LS-Morse}, it is sufficient to show that
\begin{equation}
\label{eq:LS-Morse2}
\s_w(\barf)>\s_{D(w)}(\barf)
\end{equation}
for any $w\in \H_*(B;\Q)$.  

\begin{Lemma}
\label{lemma:LS-Morse}
Let $\barf$ be a smooth function on a closed manifold $B$ such
that
$$
\s_w(\barf)=\s_{w\cap \PD(v)}(\barf)
$$
for some non-zero $w\in \H_*(B)$ and $v\in \H^{*>0}(B)$, where $\PD$
stands for the Poincar\'e duality map and $\cap$ is the intersection
product. Then the restriction of $v$ to the critical set $K$ of
$\barf$ with critical value $\s_w(\barf)$ is non-zero in
$\H^*(K)$. (If $K$ is not smooth, $\H^*(K)$ is the Alexander--Spanier
cohomology of $K$.)
\end{Lemma}

This lemma is central to Lusternik--Schnirelmann theory. For instance,
as a consequence of the lemma, we have
$\s_w(\barf)>\s_{w\cap \PD(v)}(\barf)$ when the critical points of
$\barf$ are isolated, and hence in this case the number of critical
values is strictly greater than the cup-length of $B$. The lemma is
not new (see, e.g.,~\cite[Thm.\ 1.1]{Vi97} or \cite[Rmk.\
6.8]{GG:gaps}) and we omit its proof.

In the setting of Theorem \ref{thm:LS-Morse}, we have
$D(w)=-w\cap \PD(c_1(\pi))$, where $c_1(\pi)\in \H^2(B;\Q)$ is the
first Chern class of the principal $G$-bundle $\pi\colon P\to B$. As
has been mentioned above, the critical manifolds of $\barf$ are the
lens spaces $C$, and hence $K$ is a disjoint union of several such
lens spaces. Since $\H^2(C;\Q)=0$, we have $\H^2(K;\Q)=0$, and
therefore $c_1(\pi)|_K=0$. Thus the strict inequality in
\eqref{eq:LS-Morse2} follows from the lemma.
\end{proof}

\begin{Remark}
  The condition that the manifold $Y$ is closed can be
  relaxed. Namely, it is sufficient to assume that $Y$ is a manifold
  with boundary and the function $f\colon Y\to I=[a,\,b]$ is locally
  constant on $\p Y$ and $f(\p Y)\subset \{a,b\}$. In other words,
  without any assumption on $Y$, the results from this section hold
  for any proper function $f$ on $Y$ and the filtered Morse homology
  of $f$ with a range of values $I$, i.e., for
  $\H_*(f^{-1}(I),\{f=a\};\Q)$, such that all critical sets of $f$
  within $I$ are isolated $G$-orbits with finite stabilizer.
\end{Remark}

A Floer theoretic analog of Lemma \ref{lemma:LS-Morse} for closed
symplectic manifolds has been proved in \cite{How,How:thesis} by using
the pair-of-pants product. With the definition of the product from
\cite{AS}, the argument readily extends to admissible Hamiltonians on
the symplectic completion of a Liouville domain. Alternatively -- and
this would be a simpler approach -- one can employ the action of
(quantum) homology on the filtered Floer homology; see, e.g.,
\cite{LO}, and also \cite[Rmk.\ 12.3.3]{MS} and \cite[Sect.\
2.3]{GG:hyperbolic} and \cite{Li,Vi:Floer}.  

One difficulty in extending Theorem \ref{thm:LS-Morse} to the Floer
theory setting lies in that the definition of the operator $D$ in
Floer homology, given in \cite{BO:Gysin}, is based on the standard
Morse theoretic description of $D$ which is different from the one
used above. Thus, before turning to the Floer theoretic setting, let
us discuss the two definitions of the operator $D$ in Morse theory in
more detail. The results from the next section are nowhere directly
used in the paper and serve only as an illustration and motivation for
the Floer theoretic arguments in Section \ref{sec:LS-Floer}.

\subsubsection{Two definitions of the shift operator in
  equivariant Morse theory}
\label{sec:LS-Morse-2} 
Let, as in Section \ref{sec:LS-Morse-1}, $\pi\colon P\to B$ be a
principal $G=S^1$-bundle. We assume that $B$ is a closed manifold and
$F\colon B\to \R$ is a Morse function on $B$ with critical points
$x_i$. Then $\tF=F\circ \pi$ is a Morse-Bott function on $P$ with
critical sets $S_i=\pi^{-1}(x_i)$.  We denote by $\MUM(S_i)$ the
transverse Morse index of $\tF$ at $S_i$, i.e., the Morse index of $F$
at $x_i$. Let us recall the description of the Morse (or rather
Morse--Bott) homology $\HM_*(\tF)$ of $\tF$ using broken trajectories.
On each circle $S_i$ we fix a Morse function $g_i$ with exactly one
maximum $S_i^+$ and one minimum $S_i^{-}$. Let us also fix a
$G$-invariant metric on $P$. To work over $\Z$ or an arbitrary ring,
for every $x_i$ we also need to pick a co-orientation of the unstable
manifold of $x_i$. The graded module $\CCM_*(\tF)$, over any ground
ring, is generated by $S^\pm_i$ with grading $\MUM$ determined by
setting $\MUM(S^+_i)=\MUM(S_i)+1$ and $\MUM(S^-_i)=\MUM(S_i)$.

The differential $\p\colon \CCM_*(\tF)\to \CCM_{*-1}(\tF)$ decomposes
as $\p=\p_1+\p_2$.  The term $\p_1$ counts two-stage (or one-stage)
broken Morse trajectories from $S_i$ to $S_j$ with
$\mum(S_j)=\mum(S_i)-1$, connecting minima to minima or maxima to
maxima. For instance, a coefficient $\left<S_i^+,S_j^+\right>$ in
$\p S_i^+$ is the number, with signs, of broken trajectories made of
an anti-gradient trajectory $\eta\colon (-\infty, 0]\to S_i$ of $g_i$
starting at $S_i^+$ and then the anti-gradient trajectory of $\tF$
from $\eta(0)$ to $S_j^+$. The term $\p_2$ counts one-stage
trajectories of $\tF$ connecting minima to maxima with
$\mum(S_j)=\mum(S_i)-2$.

To be more precise, when $\mum(S_j)=\mum(S_i)-2$, we have
$\p_2S_i^+=0$ and
$$
\p_2 S_i^-=\sum_j\left<S_i^-,S_j^+\right>S_j^+,
$$
where $\left<S_i^-,S_j^+\right>$ is the number, with signs, of
anti-gradient trajectories of $\tF$ from $S_i^-$ to $S_j^+$.

Note that, strictly speaking, the decomposition of $\p$ should also
include the term $\p_0$ coming from the standard Morse differential
for $g_i$, but this term is obviously zero since the Morse functions
$g_i$ are perfect.

It is a standard fact that $\p^2=0$ and that the homology of
$(\CCM_*(\tF),\p)$, called the Morse or Morse--Bott homology of $\tF$
is canonically isomorphic to $\H_*(P)$. Note that that the module
$\CCM_*(\tF)$ is completely determined by $\tF$, but the differential
$\p$ depends on the auxiliary data including the functions $g_i$ and
the metric on $P$.

Let $\Cc^\pm$ be the submodules of $\CCM_*(\tF)$ generated by
$S_i^\pm$.  Then $\p_1\colon \Cc^\pm\to\Cc^\pm$, i.e., the
decomposition $\CCM_*(\tF)=\Cc^+\oplus \Cc^-$ is preserved by
$\p_1$. Furthermore, $\p_2(\Cc^+)=0$ and $\p_2\colon \Cc^-\to \Cc^+$.
Hence, $(\Cc^+,\p_1)$ is a subcomplex of $\CCM_*(\tF)$ and the
quotient complex $\CCM_*(\tF)/\Cc^+$ is also isomorphic to
$(\Cc^-,\p_1)$. Essentially by definition, the homology of
$(\Cc^+,\p_1)$ and $(\Cc^-,\p_1)$ is isomorphic to the Morse homology
of $F$, i.e., to $\H_*(B)$, up to the shift of grading. Then the exact
sequence of complexes
$$
0\to(\Cc^+,\p_1)\to (\CCM_*(\tF),\p)\to (\Cc^-,\p_1)\to 0
$$
gives rise to a long exact sequence in homology with connecting map
$\delta$ induced by $\p_2\colon \Cc^-\to \Cc^+$ on the homology of
these complexes. This is the Gysin sequence \eqref{eq:Gysin}, although
now we used a different definition of the connecting map.

\begin{Lemma}
\label{lemma:D=delta}
The two definitions of the shift operator are equivalent: $D=\delta$.
\end{Lemma}

\begin{proof} Let us first translate the definition of $D$ from
  Section \ref{sec:LS-Morse-1} to the language of Morse homology. Fix
  a Morse--Smale metric on $B$ and denote by $\CM(x_i,x_j)$ the moduli
  space of anti-gradient trajectories $u$ of $F$ from $x_i$ to $x_j$.
  This is a smooth manifold of dimension $\mum(x_i)-\mum(x_j)$. We
  orient the moduli spaces $\CM(x_i,x_j)$ using coherent orientations.
  Denote by $\ev\colon \CM(x_i,x_j)\to B$ the evaluation map
  $u\mapsto u(0)$ and let $\Sigma$ be a cycle Poincar\'e dual to
  $-c_1(\pi)$.  (Since this class has degree two, $\Sigma$ can be
  taken to be a co-oriented submanifold of $B$.) For a generic choice
  of $\Sigma$, this cycle is transverse to all evaluation maps $\ev$
  and the number of intersections of $\ev(\CM(x_i,x_j))$ with $\Sigma$
  is finite when $\mum(x_j)=\mum(x_i)-2$. Let
  $\left<x_i,x_j\right>_\Sigma\in \Z$ be the intersection index of
  $\ev$ and $\Sigma$, and
\begin{equation}
\label{eq:D-Morse}
  D_\Sigma x_i=\sum_j\left<x_i,x_j\right>_\Sigma x_j.
\end{equation}
Then $D_\Sigma$ commutes with the Morse differential on $\CCM_*(F)$,
and the operator $D$ induced by $D_\Sigma$ on the Morse homology is
the shift operator from the Gysin sequence \eqref{eq:Gysin}. In what
follows, it will also be convenient to pick $\Sigma$ so that every
un-parametrized trajectory intersects $\Sigma$ at most once; clearly
this is true for a generic choice of $\Sigma$.

Since $\delta$ is induced by $\p_2$, to prove the lemma, it suffices
to show that
\begin{equation}
\label{eq:intersect-M1}
\left<x_i,x_j\right>_\Sigma=\left<S_i^-,S_j^+\right>
\end{equation}
for a suitable lift of the metric to $P$.

Consider a closed two-form $\sigma$ supported in a small tubular
neighborhood $U$ of $\Sigma$ and ``Poincar\'e dual'' to $\Sigma$. (In
particular, $[\sigma]=-c_1(\pi)$.) When $\mum(x_j)=\mum(x_i)-2$, the
pull back $\ev^*\sigma$ has compact support and
\begin{equation}
\label{eq:intersect-M2}
\left<x_i,x_j\right>_\Sigma=\int_{\CM(x_i,x_j)}\ev^*\sigma.
\end{equation}
Note that here we need to take $U$ sufficiently small since for a
single point $x_i$ the evaluation map $\ev$ need not in any sense be a
cycle.)

Therefore, to finish the proof of the lemma it remains to show that
the right hand sides of \eqref{eq:intersect-M1} and
\eqref{eq:intersect-M2} are equal. To this end, fix a connection on
$\pi$ with curvature $-2\pi\sigma$ and consider the standard lift of
the metric from $B$ to $P$ using this connection. (The authors are
aware of and apologize for this clash of $\pi$'s; see, e.g.,
\cite[Appen. A]{GGK} for a discussion of the relevant curvature
conventions.) Then the anti-gradient trajectories of $\tF$ are the
horizontal, i.e., tangent to the connection, lifts of the
anti-gradient trajectories of $F$. Let $\CM$ be a connected component
of $\CM(x_i,x_j)$. When $U$ is sufficiently small, $\ev^*\sigma|_\CM$
is supported in small disjoint disks $\Bb_k$ centered at the inverse
images of the intersections $\ev(\CM)\cap \Sigma$ and the integral of
$\ev^*\sigma$ over $\Bb_k$ is $\pm 1$ with the sign determined by the
sign of the intersection. Let $\widehat{\CM}=\CM/\R$ be the space of
un-parametrized trajectories. Since every un-parametrized trajectory
intersects $\Sigma$ at most once, the images $\widehat{\Bb}_k$ of the
disks $\Bb_k$ in $\widehat{\CM}$ do not overlap if $U$ is small.

For every un-parametrized trajectory $u\in \widehat{\CM}$, denote by
$\tu$ its horizontal lift to $P$ starting at $S_i^-$. We can view
$\widehat{\CM}$ as a one-parameter family of trajectories and, as $u$
varies through this family, the end-point $\tu(\infty)$ will traverse
the circle $S_j$. When $u$ passes through $\widehat{\Bb}_k$, the
end-point will make exactly one revolution in $S_j$ in the direction
given by the sign of the intersection. When $u$ is outside the union
of the intervals $\widehat{\Bb}_k$, the end-point $u(\infty)$ does not
move because the connection is flat outside $U$. Furthermore, for such
a connection chosen generically, this point is different from $S_j^+$.
Thus the number, with signs, of un-parametrized anti-gradient
trajectories of $\tF$ in $\widehat{\CM}$ connecting $S_i^-$ to $S_j^+$
is equal to the integral of $\ev^*\sigma$ over $\CM$. In other words,
this integral is exactly the contribution of $\CM$ to the right hand
side $\left<S_i^-,S_j^+\right>$ of \eqref{eq:intersect-M1}. Taking the
sum over all connected components of all two-dimensional moduli
spaces, we obtain~\eqref{eq:intersect-M1}.
\end{proof}

\begin{Remark}
  \label{rmk:D=delta} 
  The requirements that every trajectory intersects $\Sigma$ only once
  and that the intervals $\widehat{\Bb}_k$ do not overlap are imposed
  only to make the proof geometrically more transparent and are not
  really essential. By tracking the end-point $\tu(\infty)\in S_j$ as
  $u$ varies through $\widehat{\CM}$, it is not hard to show that even
  without these conditions the contribution of $\widehat{\CM}$ to
  $\left<S_i^-,S_j^+\right>$ is equal to the integral of $\ev^*\sigma$
  over $\CM$. It is essential, however, that the critical points of
  $F$ are outside $U$ and hence $\ev^*\sigma$ is compactly
  supported. (The parts of the broken trajectories in the
  compactification of $\widehat{\CM}$ have relative index one and thus
  do not intersect $\Sigma$ by transversality.)
\end{Remark}

The definitions of $D$ and $\delta$, and Lemma \ref{lemma:D=delta}
carry over by continuity to the setting when $F$ is not necessarily
Morse and the homology of $B$ and $P$ is replaced by the filtered
Morse homology.  Furthermore, applying the construction of $\delta$
and the lemma to the principal $G$-bundle
$$
\pi \colon P=Y\times S^{2m+1}\to Y\times_G S^{2m+1}=B
$$
with $F=\bar{f}$ in the notation of the proof of Theorem
\ref{thm:LS-Morse}, we arrive at the Morse theoretic definition of the
shift operator in the equivariant homology $\H_*^G(f)\cong \H_*^G(Y)$
and a Morse theoretic proof of the Gysin sequence \eqref{eq:Gysin2}.

\subsubsection{Equivariant homology as the Morse homology on the
  quotient}
\label{sec:quotient-Morse}
When the $G$-action on $Y$ is locally free, the quotient $Y/G$ is an
orbifold and as is well known $\H_*^G(Y;\Q)\cong \H_*(Y/G;\Q)$; see,
e.g., \cite{GGK}. Furthermore, assume that the critical sets $S$ of a
$G$-invariant function $f$ are isolated Morse--Bott non-degenerate
$G$-orbits. Then the equivariant Morse homology of $f$ over $\Q$ is
isomorphic to $\H_*^G(Y;\Q)$ and can be interpreted as the Morse
homology of the push-forward of $f$ to the orbifold $Y/G$. In other
words, the homology $\H_*^G(Y;\Q)\cong \H_*(Y/G;\Q)$ can be viewed as
the homology of a certain complex generated by the critical sets $S$
of $f$.

To be more specific, consider a critical set $S$. By our assumptions,
$S$ is a $G=S^1$-orbit $G/\Gamma$ where $\Gamma$ is a cyclic
subgroup. Let $\nu$ be the determinant line bundle (i.e., the top
wedge) of the unstable bundle of $S$. Clearly, $G$ acts on $\nu$, and
hence we have a representation $\Gamma\to\Z_2=\{\pm 1\}$ on the fiber
of $\nu$. Let us call $S$ \emph{good} if this representation is
trivial and \emph{bad} otherwise. For instance, $S$ is automatically
good when $|\Gamma|$ is odd. These are proto-good/bad orbits in
equivariant symplectic or contact homology.

\begin{Proposition}
\label{prop:quotient-Morse}
The equivariant homology $\H_*^G(Y;\Q)$ is equal to the homology of a
certain complex generated by the good critical sets $S$ of $f$, graded
by the Morse index, i.e., the dimension of the unstable manifold of
$S$, and filtered by $f$.
\end{Proposition}

In Section \ref{sec:quotient-Floer} we will establish, by a rather
similar argument, a Floer theoretic analog of this result.  The proof
of the proposition hinges on the following general, purely algebraic
lemma that, roughly speaking, asserts that a spectral sequence
starting with $E^{r_0}$ can be reassembled into a single complex.

\begin{Lemma}
\label{lemma:single_complex}
Let $E^{r}$ with $r\geq r_0$ be a spectral sequence of complexes over
a field of zero characteristic converging in a finite number of
steps. Then there exists a differential $\bp$ on $E^{r_0}$ such that
$\H_*(E^{r_0},\bp)=E^{\infty}$. Moreover, fixing an inner product on
$E^{r_0}$ makes the choice of $\bp$ canonical and $E^\infty$ is then
identified with the space of harmonic representatives for
$\H_*(E^{r_0},\bp)$.
\end{Lemma}

\begin{proof}
  Let us fix an inner product on $E^{r_0}$. Then we can identify
  $E^{r_0+1}$ with the orthogonal complement to $\im \p_{r_0}$ in
  $\ker \p_{r_0}$ and, proceeding inductively, $E^{r+1}$ with the
  orthogonal complement to $\im \p_{r}$ in $\ker \p_{r}$. Thus we have
  a nested sequence of bi-graded spaces
$$
E^{r_0}\supset E^{r_0+1}\supset\ldots\supset E^k,
$$
where we have assumed that $E^r$ converges in $k-r_0$ steps, i.e.,
$E^{k}=E^\infty$. The differential $\p_r$ is a map $E^r\to
E^r$ such that $E^{r+1}\subset \ker \p_r$ and, moreover, 
$$
\ker \p_r=\im \p_r\oplus E^{r+1}.
$$ 

Denote by $V_r$ the orthogonal complement to $\ker \p_r$ in
$E^r$. We have the following orthogonal decomposition
\begin{equation}
\label{eq:decomp}
E^{r_0}=\big(V_{r_0}\oplus \im \p_{r_0}\big)\oplus\ldots
\oplus \big(V_{k-1}\oplus \im \p_{k-1}\big)\oplus E^k.
\end{equation}

Let us extend $\p_r$ to $\bp_r\colon E^{r_0}\to E^{r_0}$ by setting
$\bp_r\equiv 0$ on the orthogonal complement to $E^r$. Then 
$$
\bp_r\colon V_r\stackrel{\cong}{\longrightarrow} \im \p_r
$$
is an isomorphism and $\bp_r\equiv 0$ on all terms of the
decomposition \eqref{eq:decomp} other than $V_r$. Let  
$$
\bp=\bp_{r_0}+\ldots+\bp_{k-1}.
$$
As is clear from \eqref{eq:decomp}, $\bp^2=0$
and $\H_*(E^{r_0},\bp)=E^k$. Furthermore, $E^k$ is the space of harmonic
representatives for $\H_*(E^{r_0},\bp)$.
\end{proof}

\begin {proof}[Proof of Proposition \ref{prop:quotient-Morse}]
  Recall that given a Morse--Bott function $F$ on some manifold $P$
  with critical sets $C$, its Morse--Bott homology can be defined by
  fixing a Morse function $g_C$ on every $C$ and then taking a
  complex, called the \emph{Morse--Bott complex}, generated by the
  critical points of $g_C$. The differential on this complex is
  obtained by counting broken anti-gradient trajectories of the
  functions $g_C$ and $F$ as in Section \ref{sec:LS-Morse-2}, although
  now the manifolds $C$ need not be circles. The resulting complex is
  graded by the total Morse index, i.e., by the sum of the Morse--Bott
  index of $F$ at $C$ and the Morse index of a critical point of
  $g_C$.

  Furthermore, the complex is filtered by the value of $F$. To be more
  precise, let us fix a sequence of regular values $a_p$ of $F$ such
  that every interval $(a_{p-1},\,a_p)$ contains exactly one critical
  value. Then the filtration is given by the Morse--Bott complexes of
  $F$ for the intervals $(-\infty,\,a_p)$.  We will call the resulting
  spectral sequence $(E^r_{p,q},\p_r)$ associated with this filtration
  the \emph{Morse--Bott spectral sequence} of $F$. By construction,
  this spectral sequence converges to $E^\infty=\H_*(P;\Q)$ in a finite
  number of steps. (For the range of $p$ is finite.) Its $E^1$-page is
  the direct sum of the homology spaces $\H_*(C;\CN)$ with the local
  coefficient system $\CN$ given by the determinant line bundle of the
  unstable bundle of $C$. In other words,
$$
E^1_{p,q}=\bigoplus_C \H_{p+q}(C;\CN),
$$
where the sum is taken over all critical manifolds of $F$ with
critical value in $(a_{p-1},\,a_p)$. We note that this construction is
different from that in Section \ref{sec:LS-Morse-2} where we used the
index filtration rather than the filtration by the value of $F$; see
Remark \ref{rmk:ind_vs_value}.

In the setting of the proposition, let us consider the Morse--Bott
spectral sequence of the function $F$ induced on
$B=Y\times_G S^{2m+1}$ by $f$. Then the critical sets $C$ are the lens
spaces $S\times_G S^{2m+1}=S^{2m+1}/\Gamma$. The local coefficient
system $\CN$ (over $\Q$) is trivial if and only if the critical set
$S$ is good. Thus we see that the contribution of $S$ to the
$E^1$-page, over $\Q$, is concentrated in degrees $0$ and $2m+1$ if
$S$ is good and only in degree $2m+1$ when $S$ is bad. Let us pass to
the limit as $m\to\infty$. It is not hard to see that the limit
spectral sequence, still denoted by $E^r$, exists for a suitable
``coherent'' choice of the Morse--Bott data for $F$ (the metric and
the functions $g_C$) for all $m\in\N$; cf.\
\cite[Appendix]{GHHM}. Furthermore, the limit sequence converges to
$\H_*^G(Y;\Q)$ in a finite number of steps and has the $E^1$-term
generated by the good critical sets $S$ of $f$. We take an inner
product on $E^1$ for which this collection of critical sets is an
orthonormal basis and apply Lemma \ref{lemma:single_complex} with
$r_0=1$. Then $(E^1,\bp)$ is the required complex, graded by $\mum$
and filtered by $f$. Alternatively, one could apply the lemma to $E^1$
for a finite value of $m$ and then pass to the limit as $m\to\infty$.
\end{proof}

We conclude the section by several remarks elaborating on various
aspects of these constructions.

\begin{Remark}
\label{rmk:diff-M}
The differential $\bp$ on $E^1$ is completely determined by the
Morse--Bott data for $F$ (the metric and the functions $g_C$ for all
$m\in\N$). Furthermore, it is clear from the construction that the
differential is ``natural'', i.e., a monotone decreasing homotopy from
$f_0$ to $f_1$, gives rise to a homomorphism of the filtered
complexes.  Likewise, different choices of the Morse--Bott data result
in isomorphic complexes. Note also that rather than working with the
Morse--Bott complex in the proof of Proposition
\ref{prop:quotient-Morse} we could have taken a non-degenerate
perturbation of $F$, resulting in exactly the same complex $(E^1,\bp)$
for a suitable choice of the perturbation.
\end{Remark}

\begin{Remark}
  The key idea of the constructions from Lemma
  \ref{lemma:single_complex} and the proof of Proposition
  \ref{prop:quotient-Morse} is that the complex $(E^{r_0},\bp)$ can be
  much smaller and more convenient to work with than the original
  filtered complex giving rise to the spectral sequence. It is worth
  keeping in mind that applying the lemma to the $E^0$-term, which is
  isomorphic to the original complex as a vector space, may result in
  a complex $(E^0,\bp)$ different from the original one. For instance,
  consider the filtered Morse complex of a Morse function. This
  complex comes with a preferred basis and hence is canonically
  isomorphic to $E^0$. Yet, the differential $\bp$ on $E^0$ need not
  be equal to the Morse differential. Likewise, in the setting of
  Proposition \ref{prop:quotient-Morse} when the $G$-action on $Y$ is
  free, the complex $(E^1,\bp)$ is not necessarily equal to the Morse
  complex of the function $\bar{f}$ induced by $f$ on the quotient
  $Y/G$. In fact, for a suitable choice of the Morse--Bott data, the
  complex $(E^1,\bp)$ is isomorphic to the complex $(E^0,\bp)$ for
  $\bar{f}$.
\end{Remark}

\begin{Remark}
\label{rmk:ind_vs_value}
In the construction of the Morse--Bott complex it would be more
convenient to use the Morse--Bott index of $C$, rather than the value
of $F$, to produce a filtration as in Section \ref{sec:LS-Morse-2}. In
other words, we would set $p$ to be the Morse--Bott index of $C$ and
$q$ to be the Morse index of a critical point of $g_C$ in $E^0_{p,q}$.
However, this is not always possible, for the differential can
increase the Morse--Bott index. There are, however, two relevant particular
cases when the Morse--Bott index does give rise to a filtration. One
is when $\dim C\leq 1$ for all $C$ as, for instance, when the critical
manifolds are circles. The second one is when $F$ is the pull-back of
a Morse function on the base to a fiber bundle and the metric is a
lift of a metric on the base. Both cases include the pull-back of a
Morse function to a principle $S^1$-bundle considered in Section
\ref{sec:LS-Morse-2}.
\end{Remark}

\subsection{Shift operator in equivariant Floer theory}
\label{sec:LS-Floer}
In this section we extend the constructions from Section
\ref{sec:LS-Morse} to equivariant Floer homology.

\subsubsection{The setting and the main result}
\label{sec:LS-Floer-result}
Let $(V^{2n},\omega)$ be a symplectic manifold. As in Section
\ref{sec:conventions}, we assume that $V$ is symplectically aspherical
and that $V$ is either closed or $V$ is the symplectic completion
$\widehat{W}$ of a compact symplectic manifold $W^{2n}$ with contact
type boundary. The condition that $V$ is symplectically aspherical can
be significantly relaxed but is sufficient for our purpose and
simplifies the exposition; see Remark \ref{rmk:general-V}.

Fix an interval $I=[a,\,b]$ and let $H$ be a Hamiltonian on $V$ such
that the end-points of $I$ are outside the action spectrum $\CS(H)$ of
$H$. (The interval $I$ can be infinite or semi-infinite. For instance,
if $I=\R$, we set $a=-\infty$ and $b=\infty$.) When $V=\widehat{W}$ we
also require $H$ to be admissible at infinity in the standard sense;
see Section \ref{sec:LS-sympl} for the definition. The filtered Floer
homology $\HF_*^I(H)$ is defined by continuity even though the
periodic orbits of $H$ can be degenerate. Namely, we set
$\HF_*^I(H):=\HF_*^I(H')$ where $H'$ is $C^\infty$-small
non-degenerate perturbation of $H$. Since by our assumptions $\CS(H)$
is nowhere dense and the end-points of $I$ are outside $\CS(H)$, the
homology spaces $\HF_*^I(H')$ are canonically isomorphic for all
sufficiently small perturbations $H'$ of~$H$.

Assume in addition that $H$ is autonomous. Then we also have the
equivariant filtered Floer homology $\HF_*^{G,I}(H)$, where $G=S^1$,
defined for contractible one-periodic orbits of $H$ with action in
$I$; see \cite{BO:Gysin,Vi:GAFA}. We will recall in more detail the
definition and some other relevant constructions in Section
\ref{sec:LS-Floer-def}.  As is proved in \cite{BO:Gysin}, the homology
spaces $\HF_*^{G,I}(H)$ fit into the exact sequence
\begin{equation}
\label{eq:Gysin-F}
\ldots \to
\HF_*^I(H) \to
\HF_*^{G,I}(H) \stackrel{D}{\to}
\HF_{*-2}^{G,I}(H) \to
\HF_{*-1}^I(H)
\to \ldots,
\end{equation}
which is a Floer theoretic analog of the Gysin sequence
\eqref{eq:Gysin2}. We emphasize that no non-degeneracy assumption on
$H$ is needed here.

Finally, for a non-zero element $w\in \HF_*^{G,I}(H)$ the
\emph{spectral invariant} or the \emph{action selector} $\s_w(H)$ is
defined in the standard way. Namely, we set
\begin{equation}
\label{eq:s-inv}
\s_w(H)=\inf\big\{b'\in I\setminus \CS(H)\mid w\in \im\big(i^{b'}\big)\big\}
\end{equation}
where $i^{b'}$ is the natural map
$\HF_{*}^{G,I'}(H)\to \HF_{*}^{G,I}(H)$ for
$[a,\,b']=I'\subset I=[a,\,b]$.  We also set $\s_0(H)=a$. As in the
non-equivariant setting, $\s_w(H)$ is monotone and Lipschitz (with
Lipschitz constant equal to one) in $H$.

Abusing terminology, we say that a one-periodic orbit $x$ of an
autonomous Hamiltonian is \emph{isolated} if all one-periodic orbits
$y$ intersecting its sufficiently small neighborhood in $V$ are its
reparametrizations, i.e., $y(t)=x(t+\theta)$ for some $\theta\in G$.
Denote by $\PP(H,I)$ the collection of geometrically distinct
contractible one-periodic orbits of $H$ with action in $I$.  The key
result of this section is

\begin{Theorem}
\label{thm:LS-Floer}
Assume that all one-periodic orbits in $\PP(H,I)$ are isolated and
non-constant. Then, for any non-zero element
$w\in \HF_{*}^{G,I}(H;\Q)$, we have
\begin{equation}
\label{eq:LS-Floer}
\s_w(H)>\s_{D(w)}(H).
\end{equation}
\end{Theorem}
This is Theorem \ref{thm:LS-Floer-intro} from the
introduction. We prove it in Section~\ref{sec:LS-Floer-pf}.

\begin{Remark}
\label{rmk:non-strict}
As in the case of equivariant Morse theory, the non-trivial point of
the theorem is that the inequality is strict. By \eqref{eq:par-Fl},
the non-strict inequality $\s_w(H)\geq \s_{D(w)}(H)$ holds without any
assumptions on $H$ and with any coefficients.
\end{Remark}

We emphasize again that in this theorem we impose no non-degeneracy
requirements on $H$. (When $H$ is non-degenerate, the result easily
follows from the definitions and is not particularly useful.)
However, under the conditions of the theorem, the collection
$\PP(H,I)$ is finite.  Similarly to the finite-dimensional case, the
requirements that all orbits in $\PP(H,I)$ are non-constant and
isolated are essential.

Let $\gap(H)>0$ be the minimal positive action gap in $\CS(H)\cap I$,
i.e.,
\begin{equation}
\label{eq:gap}
\gap(H)=\min \big|\CA_H(x)-\CA_H(y)\big|,
\end{equation}
where the minimum is taken over all pairs of geometrically distinct
orbits $x$ and $y$ in $\PP(H,I)$. Then, since $\s_w(H)$ and
$\s_{D(w)}(H)$ are both elements of $\CS(H)$, as an immediate
consequence of the theorem, we have

\begin{Corollary}
\label{cor:LS-Floer}
Under the conditions of Theorem \ref{eq:LS-Floer},
$$
\s_w(H)\geq \s_{D(w)}(H)+\gap(H)> \s_{D(w)}(H).
$$
\end{Corollary}

\begin{Remark}[Generalizations and variations, I]
\label{rmk:general-V}
In both cases, when $V$ is closed and when $V=\widehat{W}$, the
results from this section carry over to the equivariant Floer homology
for periodic orbits in a fixed, possibly non-trivial, free homotopy
class of loops in $V$.  To be more precise, let us fix such a class
$\ff$ and, if $\ff\neq 1$, also require $V$ to be atoroidal. Under
this assumption, the index and the action of a periodic orbit of $H$
are defined by fixing a reference loop with a reference trivialization
of $TV$ along the loop; see, e.g., \cite{BO:Gysin,Gu:nc}. With this in
mind, one can extend word-for-word the proofs to the non-contractible
setting.

In particular, when $V=\widehat{W}$ is exact and $c_1(TW)=0$, one has
the equivariant Floer homology defined for all free homotopy classes,
naturally filtered by the action and graded by the free homotopy
class. It is clear that the analogs of Theorem \ref{thm:LS-Floer} and
Corollary \ref{cor:LS-Floer} hold in this case.

Furthermore, one can significantly relax our assumptions on $V$. For
instance, when $\ff=1$, it would be sufficient to require $V$ to be
weakly monotone (see \cite{HS}) and rational. The latter condition is
needed to ensure that spectral invariants take values in $\CS(H)$; see
\cite{Us}. Now, however, the Floer homology becomes a module over a
Novikov ring and periodic orbits must be equipped with cappings,
making the geometrical meaning of the results somewhat less
clear. Finally, even the weak monotonicity assumption can be dropped
at the expense of relying on the virtual cycle machinery or one of its
variants such as polyfolds.
\end{Remark}

Before turning to the proof of Theorem \ref{thm:LS-Floer} in Section
\ref{sec:LS-Floer-pf}, we need to recall, following \cite{BO:Gysin},
the construction of the equivariant Floer homology $\HF_{*}^{G}(H)$,
define the shift operator $D$ and also establish in Section
\ref{sec:LS-Floer-D} a Floer theoretic analog of Lemma
\ref{lemma:D=delta}.

\subsubsection{Equivariant Floer homology}
\label{sec:LS-Floer-def}
Throughout this section the symplectic manifold $V$ and the
Hamiltonian $H$ are as in Section \ref{sec:LS-Floer-result}. However,
in contrast with Theorem \ref{thm:LS-Floer}, we do not impose any
additional assumption on one-periodic orbits of $H$.  The action
interval $I$ plays only a superficial role, and hence, for the sake of
brevity, we will suppress $I$ in the notation while keeping in mind
that all the orbits of $H$ considered here are required to be in
$\PP(H,I)$. For instance, we will write $\PP(H)$ and $\HF_{*}^{G}(H)$
for $\PP(H,I)$ and $\HF_{*}^{G,I}(H)$, etc.

Our first goal is to define the equivariant Floer homology
$\HF_*^G(H)$. To this end, consider parametrized Hamiltonians
$\tH\colon S^1\times V\times S^{2m+1}\to \R$ invariant with respect to
the diagonal $G=S^1$-action, i.e., such that
$\tH(t+\theta,z,\theta\cdot\zeta)=\tH(t,z,\zeta)$ where
$(t,z,\zeta)\in S^1\times V\times S^{2m+1}$ and the dot denotes the
Hopf action of $\theta\in G$ on $\zeta$. (When $V=\widehat{W}$, the
Hamiltonians $\tH$ are required to meet a certain admissibility
condition at infinity. For instance, it is sufficient to assume that
at infinity $\tH$ is independent of $(t,\zeta)$ and admissible in the
standard sense; see Section \ref{sec:SH-conditions}.)  We will
sometimes write $\tH_\zeta(t,z)$ for $\tH(t,z,\zeta)$.  Let $\Lambda$
be the space of contractible loops $S^1\to V$. The Hamiltonian $\tH$
gives rise to the action functional
\begin{gather*}
\CA_{\tH} \colon \Lambda\times S^{2m+1}\to\R,\\
\CA_{\tH}(x,\zeta) =\CA(x)-\int_{S^1} \tH(t,x(t),\zeta)\,dt,
\end{gather*}
where $\CA(x)$ is the symplectic area bounded by $x\in \Lambda$. This
functional is just a parametrized version of the standard action
functional \eqref{eq:action-0}.
  
The critical points of $\CA_{\tH}$ are the pairs $(x,\zeta)$
satisfying the condition:
$$
\textrm{the loop $x$ is a one-periodic orbit of $\tH_\zeta$ and }
\int_{S^1} \nabla_\zeta\tH_\zeta(t,x(t))\,dt=0.
$$

One can introduce the notion of transverse non-degeneracy for such
critical points and then show that transversely non-degenerate
parametrized Hamiltonians form a second Baire category set in the
space of all parametrized Hamiltonians (admissible at infinity); see
\cite{BO:Trans}. Here we only point out that due to the $G$-invariance
of $\tH$ the critical points of $\CA_{\tH}$ come in families, even
when $\tH$ is non-degenerate. Each family $S$ is an orbit of $G$,
called a \emph{critical orbit}, and hence the Hessian $d^2\CA_{\tH}$
necessarily has a kernel. Thus this notion of transverse
non-degeneracy, requiring the kernel to be one-dimensional, is more
similar to the \emph{maximal}, in the obvious sense, non-degeneracy of
non-constant periodic orbits of autonomous Hamiltonians rather than
the non-degeneracy of time-dependent Hamiltonians. Under our
assumptions on $V$ and $H$, there are only finitely many critical
orbits.

By the transverse non-degeneracy condition, the equivariant (or rather
the invariant) Floer complex $\CF_*^G(\tH)$ of $\tH$ is generated by
the critical orbits $S$ and the differential $\p^G$ is defined via the
parametrized Floer equation \eqref{eq:par-Fl}.  This
complex is filtered by the action functional $\CA_{\tH}$ and graded by
a variant of the Conley--Zehnder index $\mu(S)$ of $S$; see, e.g.,
\cite{BO:index}. The homology $\HF_*^G(\tH)$ of $(\CF_*^G(\tH),\p^G)$
is by definition the equivariant Floer homology of $\tH$,
\cite{BO:Gysin,Vi:GAFA}.

Before proceeding to the definition of the operator $D$, let us give
some more details on this construction and revisit spectral
invariants. Recall that a point in $V$ is denoted by $z$ and $\zeta$
stands for a point in $S^{2m+1}$. Fix a $(t,\zeta)$-dependent and
$G$-invariant $\omega$-compatible almost complex structure $J$ on $V$
and a $G$-invariant metric on $S^{2m+1}$, and consider pairs of
functions $\tu:=(u,\lambda)$, where $u\colon \R\times S^1\to V$ and
$\lambda\colon \R\to S^{2m+1}$, satisfying the parametrized Floer
equation
\begin{equation}
\label{eq:par-Fl}
\begin{aligned}
\p_s u+J\p_t u - \nabla_z \tH_\lambda &=0,\\
\dot{\lambda} - \int_{S^1} \nabla_\zeta\tH_\lambda(t,x(t))\,dt &=0
\end{aligned}
\end{equation}
and asymptotic in the standard sense to some critical points
$(x,\zeta)\in S$ and $(x',\zeta')\in S'$ of $\CA_{\tH}$ on
$\Lambda\times S^{2m+1}$ as $s\to\pm\infty$. One should think of these
equations as a parametrized version of the standard Floer equation
\eqref{eq:Floer-0}. When certain natural transversality conditions are
met -- and this is the case generically -- the moduli space of
solutions, modulo the $\R\times G$-action, has dimension
$\mu(S)-\mu(S')-1$; see \cite{BO:Trans,BO:Gysin}. (Here the critical
orbits $S$ and $S'$ are fixed, but the points $(x,\zeta)\in S$ and
$(x',\zeta')\in S'$ are allowed to vary.)  Then, as in all other
versions of the oriented Floer or Morse theory,
\begin{equation}
\label{eq:p^G}
\p^G S=\sum \left<S,S'\right> S',
\end{equation}
where $\left<S,S'\right>$ is the
number of points in such a moduli space, taken with signs given by
coherent orientations, when $\mu(S)-\mu(S')=1$.

The original autonomous Hamiltonian $H$ can also be viewed as a
parametrized Hamiltonian independent of $t$ and $\zeta$.  Let us pick
a transversely non-degenerate perturbation $\tH$ sufficiently
$C^\infty$-close to $H$. (When $V=\widehat{W}$, we can furthermore
assume that $\tH\equiv H$ at infinity.)  The equivariant Floer
homology $\HF_*^G(H)$ is by definition the limit of the homology
spaces $\HF_*^G(\tH)$ as $\tH\to H$ and then $m\to\infty$. In the
first of these limits, the graded homology stabilizes due to the
assumption that the end-points of $I$ are outside $\CS(H)$. In the
second limit, the graded spaces stabilize in every fixed
degree. Hence, in what follows, we will assume that $m$ is fixed and
set, abusing notation, $\HF_*^G(H):=\HF_*^G(\tH)$ where $\tH$ is
sufficiently close to $H$. This truncated homology is literally equal
to $\HF_*^G(H)$ for any finite range of degrees when $m$ is large
enough.

Having this definition settled, let us take a second look at the
spectral invariants $\s_w(H)$. For a chain
$\Cc=\sum \lambda_S S\in\CF_*(\tH)$, set
\begin{equation}
\label{eq:action-chain}
\CA_{\tH}(\Cc)=\max_S \big\{\CA_{\tH}(S)\,\big|\,\lambda_S\neq 0\big\}.
\end{equation}
Then, for $w\in \HF_*^G(H)=\HF_*^G(\tH)$, we have 
\begin{equation}
\label{eq:s-inv2}
\s_w(\tH):=\min_{[\Cc]=w} \CA_{\tH}(\Cc)\text{ and } \s_w(H)=\lim_{\tH\to
  H} \s_w(\tH).
\end{equation}
It is easy to see that this ``minimax'' definition of $\s_w(H)$ is
equivalent to \eqref{eq:s-inv}. (The spectral invariant $\s_w(\tH)$
can also be defined similarly to \eqref{eq:s-inv}.)

\begin{Example}
  When $V$ is a closed symplectically aspherical manifold, the global
  (i.e., for $I=\R$) equivariant Floer homology $\HF_*^G(H)$ is not a
  very interesting object. The homology is independent of $H$ and
  hence $\HF_*^G(H)\cong \H_{*+n}(V)\otimes \H_*(\CP^\infty)$. (To
  prove this, it suffices to take $H$ autonomous and $C^2$-small.) The
  shift map $D$ is the pairing with the generator of
  $\H^2(\CP^\infty)$ on the second factor.
\end{Example}

The construction of continuation maps extends to equivariant Floer
homology word-for-word. Namely, let $H_s$, $s\in [0,\,1]$, be a
monotone increasing (i.e., $\p_s H_s\geq 0$) homotopy of autonomous
Hamiltonians on $V$. When $V=\widehat{W}$, we also assume that all
$H_s$ are admissible at infinity. Note that with our conventions
$\CA_{H_s}$ is a monotone decreasing family of functionals on
$\Lambda$. Then the homotopy $H_s$ gives rise to a homomorphism
$\HF_*^{G,I}(H_0)\to \HF_*^{G,I}(H_1)$ as long as the end points of
$I$ are outside $\CS(H_0)\cup\CS(H_1)$. This homomorphism is, in fact,
independent of the homotopy $H_s$ with the initial and terminal
Hamiltonians $H_0\leq H_1$ fixed. Thus, for instance, one can take the
linear homotopy $H_s=(1-s)H_0+sH_1$. When a homotopy is not
increasing, it still gives rise to a map in the filtered Floer
homology with an action shift depending on the homotopy in the same
way as for the standard Floer homology; see, e.g., \cite[Sect.\
3.2.2]{Gi:coiso}.

\subsubsection{The shift operator via the Floer trajectory count}
\label{sec:shift-def1}
Let us next briefly recall the definition of the shift operator $D$
from \cite{BO:Gysin}. That definition is nearly identical to the
definition of the operator $\delta$ in Section \ref{sec:LS-Morse-2}
and we deliberately recycle the argument and notation to emphasize the
similarity. We apologize for the redundancy.

Consider a general non-degenerate parametrized Hamiltonian $\tH$, not
necessarily $C^\infty$-close to $H$, and denote by $S_i$ its critical
$G$-orbits. Now we treat $\tH$ as a Morse-Bott Hamiltonian and want to
define its non-equivariant Morse-Bott Floer complex and homology with
the differential counting broken trajectories.  On each circle $S_i$
we fix a Morse function $g_i$ with exactly one maximum $S_i^+$ and one
minimum $S_i^{-}$. Consider the vector space $\CF_*(\tH)$ generated by
$S^\pm_i$ and graded, with our conventions in mind, by setting
$\mu(S^+_i)=\mu(S_i)+1$ and $\mu(S^-_i)=\mu(S_i)$. The differential
$\p\colon \CF_*(\tH)\to \CF_{*-1}(\tH)$ comprises two terms $\p_1$ and
$\p_2$. The term $\p_1$ counts two-stage (or one-stage) broken
trajectories from $S_i$ to $S_j$ with $\mu(S_j)=\mu(S_i)-1$ connecting
$S_i^-$ to $S_j^-$ or $S_i^+$ to $S_j^+$ and the term $\p_2$ counts
one-stage Floer trajectories connecting $S_i^-$ to $S_j^+$ with
$\mu(S_j)=\mu(S_i)-2$.

To be more specific, when $\mu(S_j)=\mu(S_i)-1$,
we have 
$$
\p_1 S_i^\pm=\sum_j\left<S_i^\pm,S_j^\pm\right>S_j^\pm.
$$
Here $\left<S_i^+,S_j^+\right>$ counts the number of elements, with
signs, in the moduli space of broken trajectories made of an
anti-gradient trajectory $\eta\colon (-\infty, 0]\to S_i$ of $g_i$
starting at $S_i^+$ and then the Floer trajectory, i.e., a solution
$\tu=(u,\lambda)$ of \eqref{eq:par-Fl} from $S_i$ to $S_j$ such that
the line $\tu(\cdot,0)$ connects $\eta(0)$ and $S_j^+$. Likewise,
$\left<S_i^-,S_j^-\right>$ counts the number of elements, with signs,
in the moduli space of broken trajectories made of a Floer trajectory
from $S_i$ to $S_j$ with $\tu(s,0)\to S_i^-$ as $s\to-\infty$ and an
anti-gradient trajectory of $g_j$ connecting
$\lim_{s\to\infty} \tu(s,0)$ to $S_j^-$. Furthermore, when
$\mu(S_j)=\mu(S_i)-2$, we set
$$
\p_2S_i^+=0 \textrm{ and }
\p_2 S_i^-=\sum_j\left<S_i^-,S_j^+\right>S_j^+,
$$
where $\left<S_i^-,S_j^+\right>$ is the number of Floer trajectories
from $S_i$ to $S_j$ with the line $\tu(\cdot,0)$ connecting $S_i^-$ to
$S_j^+$.

Let $\p=\p_1+\p_2$. (As in the Morse theoretic case, the differential
$\p$ should in general have an additional term $\p_0$ counting
anti-gradient trajectories of $g_i$ within each $S_i$, but due to our
choice of $g_i$ as perfect Morse functions this term is obviously
zero; cf.\ \cite[Prop.\ 5.2]{BO:Gysin}). It is useful to keep in mind
that, while the vector space $\CF_*(\tH)$ is essentially determined by
$\tH$, the differential $\p$ depends on the almost complex structure
$J$, the metric on $S^{2m+1}$, the metrics on $S_i$ and the Morse
functions $g_i$.  We have $\p^2=0$ and the homology $\HF_*(\tH)$ of
the complex $(\CF_*(\tH),\p)$ is the ordinary Floer homology of $\tH$,
canonically isomorphic to the standard Floer homology $\HF_*(H)$ when
$\tH$ is a sufficiently small perturbation of an autonomous
Hamiltonian $H$ on $V$; we refer the reader to \cite{BO:Duke} and
\cite{BO:Gysin} for the proofs of these facts. Note that $\HF_*(H)$ is
the filtered Floer homology with an action interval $I$ suppressed in
the notation and hence this space depends on~$H$.

Consider the vector space decomposition $\CF_*(\tH)=\Cc^+\oplus \Cc^-$
with $\Cc^\pm$ generated by $S_i^\pm$. This decomposition is preserved
by $\p_1$, i.e., $\p_1\colon \Cc^\pm\to\Cc^\pm$ while $\p_2(\Cc^+)=0$
and $\p_2\colon \Cc^-\to \Cc^+$. Hence, $(\Cc^+,\p_1)$ is a subcomplex
of $\CF_*(\tH)$ and the quotient complex $\CF_*(\tH)/\Cc^+$ is
isomorphic to $(\Cc^-,\p_1)$. By the definitions of $\p_1$ and $\p^G$,
the homology of $(\Cc^+,\p_1)$ or $(\Cc^-,\p_1)$ is $\HF_*^G(\tH)$, up
to the shift of grading by $+1$ in the former case. Then, as in the
Morse theoretic case, the exact sequence of complexes
$$
0\to(\Cc^+,\p_1)\to (\CF_*(\tH),\p)\to (\Cc^-,\p_1)\to 0
$$
gives rise to a long exact sequence in homology which is nothing else
but the Gysin sequence \eqref{eq:Gysin-F}. The connecting map, i.e.,
shift operator $D$ as defined in \cite{BO:Gysin}, is the map induced
by $\p_2\colon \Cc^-\to \Cc^+$ on the homology of these
complexes. This construction applies to any finite range of degrees
when $m$ is sufficiently large. To define $D$ for $\HF_*^G(\tH)$ for
all degrees $*$, it remains to take the limit as $m\to\infty$. The
construction extends by continuity to all autonomous Hamiltonians~$H$.

As an immediate consequence of this definition and
\eqref{eq:s-inv2}, we see that $D$ is action decreasing, although not
necessarily strictly, i.e.,
$$
\s_w(H)\geq \s_{D(w)}(H),
$$
without any assumptions on $H$ and for any coefficient ring.

\subsubsection{An alternative definition of the shift operator}
\label{sec:LS-Floer-D} 

The definition of the shift operator from the previous section is
parallel to the definition of the operator $\delta$ in Section
\ref{sec:LS-Morse-2}. In this section we give an alternative
definition of $D$ similar to the one in Section \ref{sec:LS-Morse-1}
via the intersection index and prove an analog of Lemma
\ref{lemma:D=delta} showing that the two definitions are
equivalent. We keep the notation and conventions from Section
\ref{sec:LS-Floer-def}. In particular, we suppress the interval $I$ in
the notation.

Let $\tH\colon S^1\times V\times S^{2m+1}\to \R$ be a transversely
non-degenerate $G$-invariant parametrized Hamiltonian. Fix a
$G$-invariant metric on $S^{2m+1}$ and a $(t,\zeta)$-dependent
$G$-invariant almost complex structure $J$ on $V$ such that all the
transversality requirements are met. Then the moduli space
$\CM(S_i,S_j)$ of solutions $\tu=(u,\lambda)$ of the parametrized
Floer equation \eqref{eq:par-Fl}, taken up to the natural $G$-action
and asymptotic to $S_i$ and $S_j$, has dimension $\mu(S_i)-\mu(S_j)$
and the evaluation map
$$
\ev \colon \CM(S_i,S_j)\to \CP^m, \quad \tu\mapsto \lambda(0)
$$ 
is well defined. (Here $\lambda(0)$ is understood as a point in
$\CP^m$ rather than in $S^{2m+1}$.) Recall also that every critical
set $S_i$ of $\CA_{\tH}$ projects to one point in $\CP^m$. Let
$\Sigma\subset \CP^m$ be a closed codimension-two co-oriented
submanifold, Poincar\'e dual to the negative first Chern class of the
Hopf bundle $\pi\colon S^{2m+1}\to \CP^m$. In addition, we require
$\Sigma$ not to meet any of the projections of $S_i$ to $\CP^m$ and be
transerve to all evaluation maps $\ev$ for moduli spaces of dimensions
up to three. It is not hard to show that such a submanifold $\Sigma$
exists. When $\mu(S_j)=\mu(S_i)-2$, let
$\left<S_i,S_j\right>_\Sigma\in\Z$ be the intersection index of $\ev$
with $\Sigma$. Similarly to \eqref{eq:D-Morse}, we set
$$
  D_\Sigma S_i=\sum_j\left<S_i,S_j\right>_\Sigma S_j.
$$
The same argument as in the standard Floer theory (see, e.g.,
\cite{LO} and also \cite{GG:hyperbolic,Li,Vi:Floer} and \cite[Rmk.\
12.3.3]{MS}) shows that $D_\Sigma$ commutes with $\p^G$ and the
resulting operator, which we temporarily denote by
$$
[D_\Sigma]\colon \HF_*^G(\tH)\to \HF_{*-2}^G(\tH),
$$ 
is independent of the auxiliary data and a particular choice of
$\Sigma$. Passing to the limit as $\tH\to H$ and $m\to\infty$, where
$H$ is an autonomous Hamiltonian on $V$, we obtain an operator
$[D_\Sigma]$ on the filtered homology $\HF_*^G(H)$.

\begin{Lemma}
\label{lemma:D=delta-F}
The two definitions of the shift operator in equivariant Floer
theory are equivalent: $D=[D_\Sigma]$.
\end{Lemma}

\begin{proof} We reason as in the proof of Lemma
  \ref{lemma:D=delta}. However, the logical structure of the argument
  is somewhat different now and first we need to specify the auxiliary
  data.  Let as above $\Sigma$ be a closed codimension-two co-oriented
  submanifold in $\CP^m$, Poincar\'e dual to the negative
  first Chern class of the Hopf bundle $\pi$. (For instance, we can
  take a generic projective hyperplane $\CP^{m-1}$ as $\Sigma$ at this
  stage.)  Furthermore, let $\sigma$ be a closed two-form ``Poincar\'e
  dual'' to $\Sigma$ and supported in a small tubular neighborood $U$
  of $\Sigma$. (Thus, in particular, $[\sigma]=-c_1(\pi)$.) Let us fix
  a Riemannian metric on $\CP^m$ and a connection on the Hopf bundle
  with curvature $-2\pi\sigma$. Then we equip $S^{2m+1}$ with the
  standard $G$-invariant lift of the metric from $\CP^m$ to $S^{2m+1}$
  using this connection.

  Examining the argument from \cite[Sect.\ 7]{BO:Trans}, one can see
  that the transversality conditions are satisfied for a generic
  choice of the almost complex structure $J$, provided that $\tH$ also
  meets some further generic requirements. Moreover, we can ensure
  that all evaluation maps $\ev$ for all moduli spaces are transverse
  to $\Sigma$ by replacing $\Sigma$ with its $C^\infty$-small
  perturbation without changing $U$ or $\sigma$.

  The connection on the Hopf bundle is flat outside $U$ and thus can be
  used to fix a $G$-equivariant trivialization
\begin{equation}
\label{eq:triv}
S^{2m+1}\setminus \tilde{U} 
\cong (\CP^m\setminus U)\times S^1,
\end{equation}
where $\tilde{U}$ is the inverse image of $U$. All critical sets $S_i$
lie in $S^{2m+1}\setminus \tilde{U}$ and for each of them the
projection onto $S^1$ is a $G$-equivariant
diffeomorphism. Furthermore, for each solution $\tu$ of the Floer
equation \eqref{eq:par-Fl} its $S^{2m+1}$-component $\lambda(s)$ is
the horizontal lift of its projection to $\CP^m$. (This follows
readily from the observation that the integral of $\tH_\zeta(t,x(t))$
over $S^1$ on the right hand side of \eqref{eq:par-Fl} is a
$G$-invariant function on $S^{2m+1}$.)

Assume now that $\mu(S_j)=\mu(S_i)-2$. Then the pull-back
$\ev^*\sigma$ has compact support in $\CM(S_i,S_j)$ and, to prove the
lemma for $\tH$ and a fixed $m$, it suffices to show that
\begin{equation}
\label{eq:S_i-to-S_j}
\left<S_i^-,S_j^+\right>=\int_{\CM(S_i,S_j)}\ev^*\sigma.
\end{equation}
For, similarly to \eqref{eq:intersect-M2}, the right hand side is
easily seen to be equal to $\left<S_i,S_j\right>_\Sigma$; cf.\ Remark
\ref{rmk:D=delta}. Let $\CM$ be a connected component of
$\CM(S_i,S_j)$ and $\widehat{\CM}=\CM/\R$ be the space of
un-parametrized trajectories.

For every element $\tu\in\CM$ let us fix its unique lift to a map
$\R\times S^1\to V\times S^{2m+1}$, still denoted by
$\tu=(u,\lambda)$, satisfying the Floer equation \eqref{eq:par-Fl} and
such that $\tu(s,0)\to S_i^-$ as $s\to-\infty$. Then
$\tu(\infty,0):=\lim_{s\to\infty} \tu(s,0)$ is a well-defined point in
$S_j$. Denote by $\lambda(\infty)$ the projection of this point to the
fiber $S^1$ via the identification \eqref{eq:triv}. Viewing
$\widehat{\CM}$ as a one-parameter space of un-parametrized
trajectories, we obtain a map $\widehat{\CM}\to S^1$ sending (the
equivalence class of) $\tu$ to $\lambda(\infty)$. Since $\lambda$ as a
map to $S^{2m+1}$ is a horizontal lift of its projection to $\CP^m$
and since the connection is flat outside $U$, this map is constant
outside the image of $\supp(\ev^*\sigma)$ in
$\widehat{\CM}$. Generically, the value of this map outside the image
is a point different from the projection of $S_j^+$ to $S^1$ by
\eqref{eq:triv}. The number of times the end-point $\lambda(0)$
traverses $S^1$ as $\tu$ varies through $\widehat{\CM}$ is equal to
the integral of $\ev^*\sigma$ over $\widehat{\CM}$. Since the map
$S_j\to S^1$ is one-to-one, this is also equal to the number of
revolutions $\tu(\infty,0)$ makes in $S_j$ which is obviously equal to
the contribution of $\widehat{\CM}$ to
$\left<S_i^-,S_j^+\right>$. Summing up over all connected components
of $\CM(S_i,S_j)$, we obtain \eqref{eq:S_i-to-S_j}. This proves the
lemma for $\tH$ and a fixed $m$.

Finally, to show that $D=[D_\Sigma]$ on the filtered equivariant Floer
homology of an autonomous Hamiltonian $H$ on $V$, it remains to pass
to the limit as $\tH\to H$ and $m\to\infty$.
\end{proof}

\subsection{Local homology}
\label{sec:local-MB}

In this section we examine two local variants of equivariant Floer
homology: the equivariant local Floer homology of an isolated orbit
and the equivariant Morse--Bott Floer homology.

\subsubsection{Equivariant local Floer homology}
\label{sec:local}
Just as in the non-equivariant case (cf., e.g., \cite{Gi:CC,GG:gap}),
the constructions from the previous sections can be easily adapted to
define the \emph{local equivariant Floer homology} and the shift
operator on it. Namely, let $x$ be an isolated one-periodic orbit of
an autonomous Hamiltonian $H$. (We do not need to impose any
assumptions on $V$; in fact, it suffices to have $H$ defined only on a
neighborhood of $x$.) Under a $C^2$-small perturbation $\tH$ of $H$,
as above, the orbit $x$ breaks down into a collection of transversely
non-degenerate critical $G$-orbits $S_i$ lying close to
$C=Gx\times_G S^{2m+1}$ where $Gx$ is the orbit of $x$ in $\Lambda$.
Let $\CF_*^G(x)$ be the (relatively) graded vector space or module
generated by $S_i$. In Section \ref{thm:LS-Floer}, we will show that
the connecting Floer trajectories $\tu$ between $S_i$ and $S_j$ stay
close to $C$. As a consequence, \eqref{eq:p^G} gives rise to a
differential on $\CF_*^G(x)$.  The resulting local equivariant Floer
homology of $x$, denoted by $\HF_*^G(x)$ or $\HF_*^G(x,H)$ when we
need to keep track of $H$, is independent of the perturbation $\tH$.
(This construction readily extends from one isolated orbit to a
compact, connected, isolated set of one-periodic orbits.)

\begin{Example}[Simple non-degenerate orbits]
\label{ex:local-nondeg-1}
Assume that $x$ is a non-degenerate (or rather maximally
non-degenerate in the obvious sense) non-constant orbit of an
autonomous Hamiltonian. Then, when $x$ is simple, $\HF_*^G(x)$ has
rank one when $*$ is equal to the Conley--Zehnder index $\mu(x)$ and
zero otherwise. We leave a detailed proof of this fact as an exercise
for the reader. Here we only mention that one can use a $C^2$-small
Morse function $f$ on the sphere $C$ to construct a perturbation
$\tH$; see \cite{BO:12} and Remark \ref{rmk:alt-MB}. Then the
resulting homology is isomorphic to $\HM_*(f)$ up to a shift of degree
by $\mu(x)$. This Morse homology is non-zero only in degrees $*=0$
and $*=2m-1$. Passing to the limit as $m\to\infty$, we see that the
homology has rank one in degree $*=\mu(x)$ and $\HF_*^G(x)=0$
otherwise.

\end{Example}

\begin{Example}[Iterated non-degenerate orbits]
\label{ex:local-nondeg-2}
When $x$ is a non-constant maximally non-degenerate iterated orbit,
the situation is somewhat more involved.  Recall that a simple orbit
$y$ is said to be \emph{odd} if $y$ has an odd number of Floquet
multipliers in $(-1,0)$ and \emph{even} if the number of such Floquet
multipliers is even. An iterated orbit $x=y^k$ is called \emph{bad} if
$y$ is odd and $k$ is even; otherwise, $x=y^k$ is \emph{good}. The
calculation from Example \ref{ex:local-nondeg-1} still holds over $\Q$
when $x$ is iterated and good: $\HF_*^G(x;\Q)=0$ unless $*=\mu(x)$
when the homology is $\Q$. When $x$ is bad, $\HF_*^G(x;\Q)=0$ in all
degrees. The reason is that in this case $\HF_*^G(x;\Q)$ is equal to
the homology of the lens space $C$ with coefficients in a non-trivial
local coefficient system with fiber $\Q$ and this homology is zero.
\end{Example}

Denote by $\hmu(x)$ the mean index of $x$. As in the non-equivariant
case, the equivariant local Floer homology is supported close to
$\hmu(x)$ with an error not exceeding $\dim V/2$; see, e.g.,
\cite{Gi:CC,GG:gap,McL}. More precisely, let $\mu_\pm(x)$ be the upper
and lower Conley--Zehnder indices of an orbit $x$; see Section
\ref{sec:index}. Then we have

\begin{Proposition}
\label{prop:local-support}
Let $x$ be an isolated non-constant one-periodic orbit of an
autonomous Hamiltonian. Then 
$$
\supp \HF_*^G(x)\subseteq [\mu_-(x),\,\mu_+(x)]
\subseteq [\hmu(x)-n+1,\,\hmu(x)+n-1]
$$ 
i.e., $\HF_*^G(x)$ can only be non-trivial for $*$ in this range of
degrees. Moreover, if at least one of the Floquet multipliers of $x$
is different from $1$, the end-points of the second interval can be
excluded.
\end{Proposition}

The proposition is an immediate consequence of the fact that, by
\eqref{eq:mu-del}, under a small non-degenerate perturbation $x$
splits into periodic orbits with Conley--Zehnder index in
$[\mu_-(x),\,\mu_+(x)]$.

In general, when $x$ is isolated and non-constant, one can expect
$\HF_*^G(x)$ to be isomorphic to the equivariant local Floer homology
$\HF_*^\Gamma(\varphi)$, where $\varphi$ is the Poincar\'e return map
of $x$ in the level containing $x$ and $\Gamma=\Z_k$ is the stabilizer
of $x$; cf. \cite{GHHM}.

The operator $D$ vanishes in the equivariant local Floer homology. To
be more precise, we have the following result established as a part of
the proof of Theorem \ref{thm:LS-Floer} in Section
\ref{sec:LS-Floer-pf}; see \eqref{eq:D=0-2} and its proof.

\begin{Proposition}
\label{prop:D-local}
Let $x$ be an isolated non-constant one-periodic orbit of an
autonomous Hamiltonian. Then
$$
D\equiv 0 \textrm{ in } \HF^{G}_*(x;\Q).
$$
\end{Proposition}

Proposition \ref{prop:D-local} fits in a larger pattern mentioned in
the introduction that under reasonable assumptions (co)homology
operations (such as $D$ or the pair-of-pants product) vanish in the
local homology of an isolated orbit; cf.\ \cite[Prop.\ 5.3]{GG:gap}
and~\cite{GoHi}.

\subsubsection{Equivariant Morse--Bott Floer homology}
\label{sec:MB}
We will now consider another case of localization of equivariant Floer
homology. Let as above $H$ be an autonomous Hamiltonian on $V$ and let
$P$ be a connected Morse--Bott non-degenerate manifold of one-periodic
orbits of $H$ in the sense of \cite{Poz}. Since $H$ is autonomous, $P$
carries a natural $G=S^1$-action by reparametrizations. Clearly, the
action functional $\CA_H$ is constant on $P$. Set $c:=\CA_H(P)$ and
assume that all one-periodic orbits of $H$ with action $c$ are
contained in $P$.

\begin{Proposition}
\label{prop:MB} 
Let $I$ be any interval such that $c$ is the only point of $\CS(H)$ in
$I$. Then, for any coefficient ring,
$$
\HF_*^{G,I}(H)\cong \H_*^G(P)
$$
up to a shift of degrees, and the isomorphism intertwines the operator
$D$ in the equivariant Floer homology and in $\H_*^G(P)$. In
particular, when the $G$-action on $P$ is locally free and all orbits
have the same stabilizer $\Z_k$, we have
$$
\HF_*^{G,I}(H;\Q)\cong \H_*(P/G;\Q),
$$ 
where now $D$ on the right comes from the Gysin exact sequence of the
$G/\Z_k$-principal bundle $P\to P/G$. (The same is true without the
stabilizer assumption, but then one needs to replace an ordinary
principal bundle by an orbi-bundle.)
\end{Proposition}

The proof of the proposition is standard and we omit it for the sake
of brevity. The non-equivariant case goes back to \cite{Poz}; see also
\cite{Bo:thesis,BO:Duke} for a different approach.

\subsection{Proof of the Lusternik--Schnirelmann inequality in Floer
  homology}
\label{sec:LS-Floer-pf} 
In this section we prove Theorem \ref{thm:LS-Floer} and hence Theorem
\ref{thm:LS-Floer-intro}.  Throughout the proof we keep the notation
and convention from Section \ref{sec:LS-Floer-def}. In particular, we
continue suppressing the interval $I$ in the notation.  We will use
the definition of the operator $D$ from Section \ref{sec:LS-Floer-D}
rather than the original definition from \cite{BO:Gysin}. The two
definitions are equivalent by Lemma~\ref{lemma:D=delta-F}.

Consider a Hamiltonian $H$ as in the statement of the theorem. In
other words, $H$ is autonomous and, when $V=\widehat{W}$, admissible
at infinity, and all orbits in $\PP(H)$ are non-constant and
isolated. Our goal is to show that for a non-degenerate small
perturbation $\tH$ of $H$ the operator $D$ on $\HF_*^G(\tH)$ decreases
the action by at least $\gap(H)-O\big(\|H-\tH\|_{C^0}\big)$. To be
more specific, we will prove that, in the notation from \eqref{eq:gap}
and \eqref{eq:s-inv2},
\begin{equation}
\label{eq:action-shift}
\s_w(\tH)\geq \s_{D(w)}(\tH) + \gap(H)-O\big(\|H-\tH\|_{C^0}\big)
\end{equation}
for every $w\neq 0$ in $\HF_*^G(\tH)=\HF_*^G(\tH)$. Then the theorem
will follow by passing to the limit as $\tH\to H$ and $m\to\infty$.

When we view $H$ as a degenerate parametrized Hamiltonian, the
critical sets of $\CA_H$ in $\Lambda\times_G S^{2m+1}$ have the form
$Gx\times_G S^{2m+1}$, where $x\in\PP(H)$ and $Gx$ is the orbit of $x$
in $\Lambda$. We denote these sets by $C_x$ or just $C$ when the role
of $x$ is inessential. Note that $C_x=C_y$ when $y\in Gx$, i.e., $x=y$
up to a reparametrization. As in the finite-dimensional case
considered in Section \ref{sec:LS-Morse}, $C$ is diffeomorphic to the
lens space $S^{2m+1}/\Gamma$ where $\Gamma\subset G$ is the stabilizer
of $x$, i.e., $\Gamma=\Z_k$ when $x$ is the $k$th iteration of a
$1/k$-periodic orbit.

Let $\tH\colon S^1\times V\times S^{2m+1}\to\R$ be a parametrized,
transversely non-degenerate $G$-invariant perturbation of $H$. (When
$V=\widehat{W}$, we can take $\tH\equiv H$ at infinity.)  We require
$\tH$ to be sufficiently $C^\infty$-close to $H$. Under this
perturbation, the critical sets $C=C_x$ of $H$ break down into finite
collections of one-dimensional families $S_{x,i}$ of critical points
of $\CA_{\tH}$. The action functional $\CA_{\tH}$ is constant on
$S_{x,i}$ and hence we set $\CA_{\tH}(S_{x_,i})$ to be its value at an
arbitrary point of $S_{x,i}$. Each $S_{x,i}$ is an isolated orbit of
$G$ lying close to $C$ in $\Lambda\times S^{2m+1}$.  In particular,
when $(y,\zeta)\in S_{x,i}$, the loop $y$ is $C^\infty$-close to $x$,
up to a parametrization, and $\CA_{\tH}(S_{x,i})$ is close to
$\CA_H(x)$. Furthermore, recall that the spectrum $\CS(H)$ (or, to be
more precise, $\CS(H)\cap I$) is finite due to our requirements on
$H$. For every $c\in\CS(H)$ pick a short interval $\CI_c$ centered at
$c$ so that the intervals $\CI_c$ and $\CI_{c'}$ do not overlap when
$c\neq c'$.  Thus, when $\tH$ is sufficiently close to $H$, the action
values $\CA_{\tH}(S_{x,i})$ are in $\CI_c$ where $c=\CA_H(x)$.

The complex $\CF^G_*(\tH)$ is generated by $S_{x,i}$, where $x$ runs
through all geometrically distinct one-periodic orbits of $H$, and
graded by $\CA_{\tH}(S_{x,i})$. Let $\CI$ be one of the intervals
$\CI_c$.  The filtered Floer homology $\HF^{G,\CI}_*(\tH)$ is the
homology of the complex $(\CF^{G,\CI}_*(\tH),\p^G)$ generated by the
critical points $S_{x,i}$ with action in $\CI$.  The operator
$D_\Sigma$ is defined on $\CF^{G,\CI}_*(\tH)$ and $D$ is defined on
$\HF^{G,\CI}_*(\tH)$.

It is not hard to see that to prove \eqref{eq:action-shift} it
suffices to show that
\begin{equation}
\label{eq:D=0}
D\equiv 0 \textrm{ on } \HF^{G,\CI}_*(\tH;\Q)=\HF^{G,\CI}_*(H;\Q).
\end{equation}

Before proving \eqref{eq:D=0}, we need to recall several standard
facts about the behavior of Floer trajectories, drawing from
\cite[Sect.\ 1.5]{Sa:notes}. For a solution $\tu=(u,\lambda)$ of the
parametrized Floer equation \eqref{eq:par-Fl}, consider the following
two energy integrals
$$
E(u)=\int_{-\infty}^{\infty}\|\p_s u(s)\|^2_{L^2}\,ds
\quad\textrm{and}\quad
E(\tu)=E(u)+\int_{-\infty}^{\infty}\|\dot{\lambda}(s)\|^2\,ds,
$$
where we set $u(s)=u(\cdot,s)$. (Here and below it is convenient to
interpret $\tu$ as literally a map to $V\times S^{2m+1}$, i.e.,
without taking a quotient by $G$; this should cause no confusion.)
Then, when $\tu$ is asymptotic to $S_{x,i}$ at $-\infty$ and to
$S_{y,j}$ at $\infty$, we have
\begin{equation}
\label{eq:energies}
E(u)\leq E(\tu)=\CA_{\tH}(S_{x,i})-\CA_{\tH}(S_{y,j}).
\end{equation}

Since $V$ is aspherical, $\|\p_s u\|$ is point-wise bounded from above
by a constant independent of $\tH$ as along as $\tH$ is $C^2$-close to
$H$. Moreover, when $E(u)$ is sufficiently small, we have a sharper
point-wise bound
\begin{equation}
\label{eq:L-infty}
\|\p_s u\|\leq O\big(E(u)^{1/4}\big)\leq 
O\Big(\big(\CA_{\tH}(S_{x,i})-\CA_{\tH}(S_{y,j})\big)^{1/4}\Big).
\end{equation}
This is a consequence of the fact that the energy density
$\varphi=\|\p_s u\|^2$ satisfies the differential inequality
$\Delta\varphi\geq -a$ for some $a\geq 0$ independent of $\tH$ as long
as $\tH$ is $C^2$-close to $H$. Indeed, first observe that
$$
\Delta\varphi\geq
-a'(1+\varphi^2)\big(1+\|\nabla_s\dot{\lambda}\|\big)
$$ 
with $a'\geq 0$. This inequality can be proved exactly in the same way
as its counterpart (without the term $\|\nabla_s\dot{\lambda}\|$) for
the ordinary Floer equation; see \cite[p.\ 137--138]{Sa:MT}. As we
have already seen, $\|\p_s u\|$ and hence $\varphi$ are a priori
bounded. Likewise, $\|\nabla_s\dot{\lambda}\|$ is also a priori
bounded as a consequence of \eqref{eq:par-Fl}, and therefore
$\Delta\varphi$ is a priori bounded from below. Then
\eqref{eq:L-infty} follows from this lower bound by a standard
argument; see, e.g., \cite[Sect.\ 5]{Sa:MT} or \cite[Sect.\
1.5]{Sa:notes}.

Furthermore, recall that for any map $z\colon S^1\to V$ and for some
$\eps>0$ independent of $z$, we have
\begin{equation}
\label{eq:L2}
\int_{S^1}\big\|\dot{z}(t)-X_H(z(t))\big\|^2\,dt>\eps,
\end{equation}
unless $z$ is sufficiently $C^0$-close to one of the periodic
solutions in $\PP(H)$; see \cite[Exercise 1.22]{Sa:notes}. Clearly,
\eqref{eq:L2} also holds with $\tH$ in place of $H$.

These facts have two consequences important for our proof.  First of
all, we claim that for any two geometrically distinct orbits $x$ and
$y$ such that there exists a solution $\tu$ asymptotic to $S_{x,i}$
and $S_{y,j}$, we have
\begin{equation}
\label{eq:treshold}
\CA_{\tH}(S_{x,i})-\CA_{\tH}(S_{y,j})\geq e>0,
\end{equation}
where the lower bound $e$ depends only on $H$. (We assume here that
the background auxiliary structure is sufficiently close to a fixed
almost complex structure on $V$ and a fixed metric on $S^{2m+1}$.)
Indeed, by \eqref{eq:L-infty} and \eqref{eq:L2},
$$
E(u)\geq e>0,
$$
where $e$ depends only on $H$ as long as $\tH$ is sufficiently close
to $H$ and on the auxiliary structure. This, combined with
\eqref{eq:energies}, implies \eqref{eq:treshold}.

Secondly, for every $\tu$ asymptotic to $S_{x,i}$ and $S_{x,j}$ we
have
\begin{equation}
\label{eq:C1}
d_{C^1}(u(s),x)\to 0 \textrm{ as }\|H-\tH\|_{C^2}\to 0
\end{equation}
uniformly in $s\in\R$ and $\tH$, where $d_{C^1}$ stands for the
$C^1$-distance in the loop space $\Lambda$. Indeed, it readily follows
from \eqref{eq:L-infty} and \eqref{eq:L2} that $u(s)$ is $C^0$-close
to $x$ when $\|H-\tH\|_{C^2}$ is small. Then, by the Floer equation
and again by \eqref{eq:L-infty}, $u(s)$ is also $C^1$-close to $x$.

We are now in a position to prove \eqref{eq:D=0} and thus finish the
proof of the theorem.

Pick $c\in \CS(H)$ and let as above $\CI=\CI_c$. We have a natural
decomposition of vector spaces
$$
\CF_*^{G,\CI}(\tH)=\bigoplus_x \CF_*^{G}(x),
$$
where the sum is taken over all geometrically distinct periodic orbits
$x$ of $H$ with $\CA_H(x)=c$ and $\CF_*^{G}(x)$ stands for the
subspace generated by $S_{x,i}$. Let us choose interval $\CI$ so short
that its length is less than $e$ from \eqref{eq:treshold}. Then, by
\eqref{eq:treshold}, this is also a decomposition of complexes, i.e.,
$\p^G\colon \CF_*^{G}(x)\to \CF_*^{G}(x)$, and moreover this
decomposition is also obviously preserved by $D$. (One can use here
either of the definitions of $D$.) The homology $\HF_*^{G}(x)$ of
$\CF_*^G(x)$ is the local Floer homology of $x$ discussed in Section
\ref{sec:local}, and the above argument fills in the technical details
of the definition of $\HF_*^{G}(x)$ omitted in that section.

Now we see that it is sufficient to prove that
\begin{equation}
\label{eq:D=0-2}
D\equiv 0 \textrm{ in } \HF^{G}_*(x).
\end{equation}

To this end, we will employ the second definition of $D$. Fix a
$C^1$-small tubular neighborhood $\CU$ of $C=C_x$ in
$\Lambda\times_G S^{2m+1}$ with $\rho\colon \CU\to C$ being the smooth
tubular neighborhood projection. More explicitly, one can show that
there exists a $G$-invariant $C^1$-small neighborhood $\widehat{\CU}$
of $Gx$ in $\Lambda$ and a smooth, in the obvious sense,
$G$-equivariant projection $\hat{\rho}\colon \widehat{\CU}\to Gx$.
Then $\CU=\widehat{\CU}\times _G S^{2m+1}$ and $\rho$ is induced by
$\hat{\rho}$ on the first factor.

Suppressing $x$ in the notation, let us write $S_i$ for $S_{x,i}$. By
\eqref{eq:C1}, a solution $\tu$ asymptotic to $S_i$ and $S_j$ is
entirely contained in $\CU$. Thus we have a well-defined evaluation
map
$$
\widetilde{\ev}\colon \CM(S_i,S_j)\to C,\quad
\widetilde{\ev}(\tu)=\rho\big((u(0),\lambda(0))\big).
$$
The composition of $\widetilde{\ev}$ with the natural projection
$C\to\CP^m$ is the evaluation map $\ev$ from Section
\ref{sec:LS-Floer-D}. Let us fix a submanifold $\Sigma\subset \CP^m$
as in Section \ref{sec:LS-Floer-D} and let $\tilde{\Sigma}$ be its
inverse image in $C$ under the submersion $C\to\CP^m$. Clearly,
$\tilde{\Sigma}$ is a closed, smooth, co-oriented submanifold of $C$
of codimension two, which is transverse to $\widetilde{\ev}$ if and
only if $\Sigma$ is transverse to $\ev$.  Therefore,
$\left<S_i,S_j\right>_\Sigma$ can also be evaluated as the
intersection index $\left<S_i,S_j\right>_{\tilde{\Sigma}}$ of
$\widetilde{\ev}$ with $\tilde{\Sigma}$. Hence, in the complex
$\CF_*^G(x)$, we have
$$
  D_\Sigma S_i=\sum_j\left<S_i,S_j\right>_{\tilde{\Sigma}} S_j.
$$
The cycle $\tilde{\Sigma}$ is homologous to zero in $C$ over $\Q$,
since $C$ is a lens space and thus $\H^2(C;\Q)=0$. Now the standard
argument showing that the ``intersection'' action of the ordinary
homology on the Morse or Floer homology is well-defined (see, e.g.,
\cite{LO}) also shows that the operator $D$ induced by $D_\Sigma$ on
the Floer homology $\HF_*^G(x)$ over $\Q$ is zero. This concludes the
proof of \eqref{eq:D=0-2} and of Theorem \ref{thm:LS-Floer} (and hence of
Theorem \ref{thm:LS-Floer-intro}).

\subsection{``Quotient'' construction of equivariant Floer homology}
\label{sec:quotient-Floer} 
Throughout this section we keep the notation and conventions from
Section \ref{sec:LS-Floer-def}. In particular, we suppress the
interval $I$ in the notation.

One shortcoming of the construction of equivariant Floer homology is
that even when $H$ is, say, autonomous and its orbits are maximally
non-degenerate, $\HF_*^G(H)$ is not defined as the homology of a
complex generated by one-periodic orbits of $H$. This creates some
inconvenience, admittedly minor, when the equivariant Floer or
symplectic homology is used to produce lower bounds on the number of
periodic orbits. A simple remedy for this is the following Floer
theoretic analog of Proposition~\ref{prop:quotient-Morse}.

\begin{Proposition}
\label{prop:quotient-Floer}
Let $H$ be an autonomous Hamiltonian such that all one-periodic
orbits of $H$ with action in $I$ are non-constant and maximally
non-degenerate. Then $\HF_*^G(H;\Q)$ is the homology of a certain
complex $\CF_*(H)^G$ generated over $\Q$ by the good one-periodic
orbits of $H$, graded by the Conley--Zehnder index, and filtered by
the action.
\end{Proposition}

\begin{proof} We will argue as in the proof of Proposition
  \ref{prop:quotient-Morse}. To this end, it is convenient to slightly
  modify the definition of $\HF_*^G(H)$ and use the standard
  Morse--Bott approach. Namely, denote by $x$ the periodic orbits of
  $H$ with action in $I$. Fixing $m$, we will treat $H$ as a
  $G$-invariant Hamiltonian $S^1\times V\times S^{2m+1}\to \R$. As in
  the proof of Theorem \ref{thm:LS-Floer}, every $x$ gives rise to a
  critical set $C_x=Gx\times_G \Lambda$ of $\CA_H$ in
  $\Lambda\times_G S^{2m+1}$ diffeomorphic to the lens space
  $S^{2m+1}/\Gamma$ where $\Gamma\subset G$ is the stabilizer of $x$,
  i.e., $\Gamma=\Z_k$ when $x$ is the $k$th iteration of a
  $1/k$-periodic orbit. (Furthermore, two orbits which differ by a
  parametrization give rise to the same critical set.) Moreover, since
  the orbits $x$ are maximally degenerate, the critical sets $C_x$ are
  now Morse--Bott non-degenerate.

  Next, rather than taking a transversely non-degenerate perturbation
  $\tH$ of $H$, we can use the Morse--Bott construction to account for
  the contributions of the critical manifolds $C_x$ along the lines
  of, e.g., \cite{BO:Duke}. Namely, let us equip each manifold $C_x$
  with a Morse function $g_x$ and a Riemannian metric. Denoting the
  critical points of $g_x$ by $S_{x,i}$, we obtain a complex generated
  by $S_{x,i}$, for all $x$ and all $i$, equipped with the
  differential counting broken Morse--Floer trajectories in a
  manner similar to the definition of the Floer differential in
  Section \ref{sec:shift-def1}, but with critical manifolds being lens
  spaces rather than circles. This complex is graded by the sum of the
  Conley--Zehnder index of $x$ and the Morse index of $S_{x,i}$ in
  $C_x$ and filtered by the action of $x$. A standard (but
  non-trivial) argument shows that the homology of this complex
  converges to $\HF_*^G(H)$ as $m\to\infty$; cf.\
  \cite{Bo:thesis,BO:Duke}. 

  Consider the spectral sequence associated with the action
  filtration. Its $E^1$-page is the direct sum of the Morse homology
  spaces of the functions $g_x$ with possibly twisted
  coefficients. For every closed orbit $x$, this is the homology over
  $\Q$ of $C_x$ with a trivial local coefficient system $\Q$ when $x$
  is good and non-trivially twisted local coefficients when $x$ is
  bad; see Example \ref{ex:local-nondeg-2}. Therefore, every good $x$
  contributes $\Q$ to the $E^1$-page in total degree $p+q=0$ and
  $p+q=2m+1$ and every bad $x$ contributes only in degree $p+q=2m+1$.
  Clearly, the spectral sequence converges in a finite number of steps
  bounded from above by the length of the action spectrum (and thus
  independent of $m$) and, for a suitable choice of auxiliary data,
  stabilizes as $m\to\infty$ in any finite $(p,q)$-region.

  Let us apply Lemma \ref{lemma:single_complex} with $r_0=1$ to this
  spectral sequence. As a result we obtain a complex $(E^1,\bp)$ whose
  homology is isomorphic to $E^\infty$. The $E^1$-term comes with a
  preferred set of generators, and hence the complex is canonically
  determined by the spectral sequence. Since the spectral sequence
  stabilizes as a function of $m$ in every finite $(p,q)$-range, the
  sequence of complexes $(E^1,\bp)$ converges as $m\to\infty$.
  Passing to the limit, we obtain a single complex
  $\CF_*(H)^G:=\lim E^1$ equipped with the limit differential, still
  denoted by $\bp$, such that its homology is
  $\lim E^\infty =\HF_*^G(H;\Q)$. (Here we use the fact that the
  direct limit and the homology commute.) The complex $\CF_*(H)^G$ is
  generated by the good orbits of $H$, graded by the Conley--Zehnder
  index and filtered by the action.
\end{proof}

\begin{Remark}
\label{rmk:diff}
As is clear from the proof, the differential $\bp$ can be described
explicitly. However, such a description does not seem to be
particularly useful; for the differential, as in all Floer-type
constructions, would be very difficult to calculate except when it is
obviously zero due to, say, lacuna in the complex.  It is also clear
that although the differential depends on the auxiliary data,
different choices of such data result in isomorphic
complexes. Furthermore, the differential is ``natural'' in the same
sense as the Floer or Morse differential: a monotone increasing
homotopy from $H_0$ to $H_1$ gives rise to a homomorphism of
complexes; cf.\ Remark \ref{rmk:diff-M}. Finally, this description of
equivariant Floer homology carries over word-for-word to the local
case, i.e., $\HF_*^G(x;\Q)$ is isomorphic to the homology of a certain
complex graded by the Conley--Zehnder index and generated by the good
orbits which $x$ breaks down into under a non-degenerate perturbation
of $H$.
\end{Remark}

\begin{Remark}
  With Proposition \ref{prop:quotient-Floer} in mind one can replace
  contact homology by equivariant symplectic homology without
  affecting the rest of the argument everywhere in the proofs from
  \cite{Gu:pr} and in some other instances.
\end{Remark}

\begin{Remark}
\label{rmk:alt-MB}
Our choice to work with the Morse--Bott complex in the proof of
Proposition \ref{prop:quotient-Floer} is mainly determined by
expository considerations. Instead, we could have worked with a
transversely non-degenerate parametrized perturbation of $H$ resulting
in exactly the same complex $(\CF_*(H)^G,\bp)$; cf.\ Remark
\ref{rmk:diff-M}. Furthermore, Proposition \ref{prop:quotient-Floer}
and its proof also enable one to replace a part of the proof of
Theorem \ref{thm:LS-Floer} by a purely Morse-theoretic argument as in
Section \ref{sec:LS-Morse} with some minor simplifications,
although conceptually the proof would remain the same. However, then
the definition of $D$ via the Floer trajectory count in Section
\ref{sec:shift-def1} requires an awkward from our perspective
two-level Morse--Bott construction.
\end{Remark}

\section{Shift operator: symplectic homology}
\label{sec:LS2}

\subsection{Shift operator in equivariant symplectic  homology}
\label{sec:LS-sympl} 
In this section we prove, essentially by passing to the limit, an
analog of Theorem \ref{thm:LS-Floer} for equivariant symplectic
homology. To state the result, let us briefly recall the relevant
definitions. Note that our conventions differ slightly from, e.g.,
\cite{BO:Gysin}, although the resulting definitions are equivalent to
the standard ones.

\subsubsection{Conventions and requirements} 
\label{sec:SH-conditions}
We will assume that $(W^{2n},\omega)$ is a compact exact symplectic
manifold with a contact type boundary $(M,\alpha)$, i.e., $M=\p W$ and
$d\alpha=\omega|_M$, and in addition the orientation
$\alpha \wedge (d\alpha)^{n-1}$ agrees with the boundary orientation
of $M$, i.e., the Liouville vector field along $M$ points outward.  In
other words, $(W,\omega)$ is a \emph{strong symplectic filling} of
$(M,\alpha)$. Furthermore, we will require that
$c_1(TW)|_{\pi_2(W)}=0$.  The symplectic completion $V=\widehat{W}$ is
the union
$$
\widehat{W}=W\cup_{M} \big(M\times [1,\infty)\big)
$$ 
with the symplectic form $\omega$ extended as $d(r\alpha)$ to the
cylindrical part. (Here $r$ is the coordinate on $[1,\infty)$.) 

We will concentrate on contractible in $W$ closed Reeb orbits of
$\alpha$ and denote by $\PP(\alpha)$ the collection of such orbits and
by $\CS(\alpha)$ the action or period spectrum of $(M,\alpha)$, i.e.,
$$
\CS(\alpha)=\Big\{\CA_\alpha(x):=\int_x\alpha\,\Big|\, x\in\PP(\alpha)\Big\}.
$$
Recall also that the Hamiltonian flow of $H\equiv r$ on the
cylindrical part coincides with the Reeb flow.

Let us consider autonomous Hamiltonians $H$ on $\widehat{W}$ meeting
the following requirements:
\begin{itemize}
\item[(i)] $H$ is $C^2$-small and negative on $W$;
\item[(ii)] $H=h(r)$, where $h\colon [1,\infty)\to \R$ is convex
  (i.e., $h''\geq 0$) on the cylindrical part $M\times [1,\infty)$;
\item[(iii)] $H=\kappa r-c$ outside a compact set (i.e., when $r\geq r_0$
  for some $r_0$), where $\kappa\not \in \CS(\alpha)$ is positive.
\end{itemize}
We call such Hamiltonians \emph{admissible}. (Note that for such
Hamiltonians we necessarily have $c>0$ in (iii).) When $H$ satisfies
only (iii), we call it \emph{admissible at infinity}. We emphasize
that we do not impose here any non-degeneracy conditions on $H$. When
$H$ is admissible at infinity, its Hamiltonian flow has no
one-periodic orbits in the region $r\geq r_0$ and the global Floer
homology of $H$ is defined and independent of $H$ as long as $\kappa$
is fixed.

Let $I=[a,\,b]\subset \R$ be a fixed interval with end-points outside
$\CS(\alpha)$.  This interval can be finite or semi-finite or infinite
and in the last two cases the infinite end-points are not included in
$I$.  For any two Hamiltonians $H_0\leq H_1$ admissible at infinity
and such that the end points of $I$ are outside $\CS(H_0)$ and
$\CS(H_1)$, we have a well-defined continuation map
$$
\HF_*^{G,I}(H_0)\to \HF_*^{G,I}(H_1)
$$
and we set
\begin{equation}
\label{eq:SH}
\SH_*^{G,I}(W):=\varinjlim_H \HF_*^{G,I}(H),
\end{equation}
where the limit is taken over all admissible Hamiltonians $H$ such
that the end-points of $I$ are not in $\CS(H)$. For instance, when
$I=[\delta,\infty)$ for a sufficiently small $\delta>0$, we obtain the
standard positive equivariant symplectic homology $\SH_*^{G,+}(W)$. In
fact, in the situation we are interested in we can always assume that
$a>0$.

Finally, note that by passing to the limit we obtain the shift operator
$$
D\colon \SH_*^{G,I}(W)\to \SH_{*-2}^{G,I}(W).
$$ 
Likewise, we have the spectral invariants $\s_w(\alpha)$ for
$w\in \SH_*^{G,I}(W)$, which can be defined in two ways. First of all,
as in \eqref{eq:s-inv}, we can set
\begin{equation}
\label{eq:spec-inv}
\s_w(\alpha)=\inf\big\{b'\in I\setminus \CS(\alpha)\,\big|\, w\in
\im\big(i^{b'}\big)\big\}
\end{equation}
where $i^{b'}$ is the natural map
$\SH_{*}^{G,I'}(W)\to \SH_{*}^{G,I}(W)$ for
$[a,\,b']=I'\subset I=[a,\,b]$. (By definition, $\s_0(\alpha)=a$.)
Alternatively, we can take a cofinal sequence of Hamiltonians $H_j$
and a sequence $w_j\in \HF_*^{G,I}(H_j)$ converging to $w$ in the
obvious sense, and set
$$
\s_w(\alpha)=\lim_{j\to\infty} \s_{w_j}(H_j).
$$
It is not hard to show that these definitions are equivalent; cf.\ the
proof of Theorem~\ref{thm:LS-SH}.

The shift map $D$ and the spectral invariants are ``functorial''. To
be more specific, consider a symplectic cobordism $Z$ with
$\p Z=M_0\cup M_1$ from $(M_0,\alpha_0)=\p W_0$, in obvious notation,
to $(M_1,\alpha_1)=\p W_1$, where $W_1=W_0\cup Z$, such that $W_1$ is
exact and $c_1(TW_1)|_{\pi_2(W_1)}=0$. Then $Z$
gives rise to a \emph{cobordism map}, also known as the transfer map,
$$
\Phi_Z\colon \SH_*^{G,I}(W_1)\to \SH_*^{G,I}(W_0),
$$ 
induced, before passing to the limit, by continuation maps in
equivariant Floer homology; see Section \ref{sec:LS-Floer-def}. This
map was originally introduced in \cite{Vi:GAFA} in a slightly different
setting and then studied in detail in \cite{Gut15}.

\begin{Proposition}
\label{prop:cobordism}
The shift operator $D$ and the cobordism map $\Phi_Z$ commute and
$\s_{\Phi_Z(w)}(\alpha_0)\leq \s_w(\alpha_1)$.
\end{Proposition}

\begin{Corollary}
The spectral invariant $\s_w(\alpha)$ is Lipschitz (with Lipschitz
constant equal to one) in $\alpha$.
\end{Corollary}
The proofs of the proposition and the corollary are absolutely
standard and we omit them.  Furthermore, as a consequence of
Proposition \ref{prop:quotient-Floer} we obtain, by passing to the
limit, the following result.

\begin{Proposition}
\label{prop:quotient-SH}
Assume that $I\subset (0,\infty)$ and that all contractible closed
Reeb orbits of $\alpha$ with action in $I$ are non-degenerate.  Then
$\SH_*^{G,I}(W;\Q)$ is the homology of a certain complex generated
over $\Q$ by the good closed Reeb orbits with action in $I$, graded by
the Conley--Zehnder index, and filtered by the action.
\end{Proposition}

We will give a detailed proof of Proposition \ref{prop:quotient-SH} in
Section \ref{sec:LS-SH-result}. It is not clear to us if the complex
from the proposition is necessarily isomorphic to the filtered contact
homology complex of $(M,\alpha)$ when (if) the contact homology is
defined; see \cite{BO:12}. However, in any event, the nature of the
differential is not really essential for our purposes. As we have
already pointed out in Remark \ref{rmk:diff}, even if the differential
is defined explicitly, it can rarely be calculated beyond some obvious
cases.  Furthermore, as in the case of equivariant Morse or Floer
homology (see Remark \ref{rmk:diff-M}), although the differential
depends on the auxiliary data, different choices of such data result
in isomorphic complexes. The differential is ``natural'' in the same
sense as the Floer or Morse differential, i.e., a symplectic cobordism
gives rise to a homomorphism of complexes.

\subsubsection{The Lusternik--Schnirelmann inequality in equivariant symplectic
  homology}
\label{sec:LS-SH-result}
With general definitions in place, we are ready to (re)state our main
result, Theorem \ref{thm:LS1-intro} from the introduction. Recall that
a closed Reeb orbit $x$ is called isolated if there exists a tubular
neighborhood $U$ of $x$ and an interval
$\CI=(\CA_\alpha(x)-\eps, \,\CA_\alpha(x)+\eps)$ such that no periodic
orbit with action in $\CI$ enters $U$.
 
\begin{Theorem}
\label{thm:LS-SH}
Assume that $I=[a,\,b]\subset (0,\,\infty)$ and that all closed Reeb
orbits of $\alpha$ with action in $I$ are isolated.  Then, for any
non-zero element $w\in \SH_{*}^{G,I}(W;\Q)$, we have
\begin{equation}
\label{eq:LS-SH}
\s_w(\alpha)>\s_{D(w)}(\alpha).
\end{equation}
\end{Theorem}

We emphasize again that the main new point of this theorem is that the
inequality is strict. The non-strict inequality holds without any
assumptions on the orbits and for any coefficient ring; cf.\ Remark
\ref{rmk:non-strict}. Furthermore, it is essential that in this
theorem, as in Theorem \ref{thm:LS-Floer}, we make no non-degeneracy
assumptions on $\alpha$.

\begin{proof}
  First, observe that it is sufficient to prove the theorem for an
  interval $I$ with a finite upper end-point $b$. (If the upper end of
  the interval is $\infty$, we can replace it by any
  $b> \s_w(\alpha)$.) Then $\CS(\alpha)\cap I$ is a finite set and
  hence $\gap(\alpha)$, which is by definition the infimum of positive
  action gaps in $I$, is strictly positive.

  Consider a cofinal sequence of admissible Hamiltonians $H_j$ with
  the following properties:
\begin{itemize}
\item $H_j=\kappa_jr -c_j$ on $M\times [r_j,\infty)$ where $r_j\to 1$
  and $\kappa_j\to\infty$,
\item $h_j\leq 0$ on $[1,r_j]$.
\end{itemize}
Every $x\in\PP(\alpha)$ occurs as a one-periodic orbit of $H_j$ for a
sufficiently large $j$ (depending on $x$) exactly once. This orbit,
denoted by $x_j$, lies on the level $r=\rho_j\in (1,\,r_j)$, where
$\rho_j=\rho_j(T)$ is uniquely determined by the condition $h'_j(\rho_j)=T$ with
$T:=\CA_\alpha(x)$. The Hamiltonian action of the resulting orbit
$x_j$ is
\begin{equation}
\label{eq:two-actions}
\CA_{H_j}(x_j)=\rho_j(T) T-h_j(\rho_j)=:f_j(T).
\end{equation}
Clearly, $\rho_j\to 1$ and $h_j(\rho_j)\to 0$ since the sequence $H_j$
is cofinal. Thus
$$
\CA_{H_j}(x_j)\to\CA_\alpha(x)\textrm{ as } j\to\infty.
$$
As a consequence, when $j$ is large enough, $\CS(H_j)\cap I$ is a
finite set converging to $\CS(\alpha)\cap I$ as
$j\to\infty$. 

Furthermore, we also have convergence of the minimal
action gaps:
\begin{equation}
\label{eq:gap_conv}
\gap(H_j)\to \gap(\alpha)>0,
\end{equation}
where in both cases we have intersected the action spectrum with
$I$. (To prove this, it is enough to guarantee that
$\gap(H_j)\not\to 0$.) To establish \eqref{eq:gap_conv}, observe first
that, by \eqref{eq:two-actions}, $\CS(H_j)=f_j\big(\CS(\alpha)\big)$,
where $f_j(T)$ for any $T\in I$ is given by \eqref{eq:two-actions}
with $\rho_j$ determined via $h'_j(\rho_j)=T$. A direct calculation
shows that $f'_j(T)=\rho_j(T)\to 1$ uniformly on $I$ as $j\to \infty$,
which implies \eqref{eq:gap_conv}.

Finally note that all one-periodic orbits of $H_j$ with action
in $I$ are non-constant since $a>0$. Now the theorem follows from Corollary
\ref{cor:LS-Floer}.
\end{proof}

\begin{proof}[Proof of Proposition \ref{prop:quotient-SH}] Throughout
  the proof we keep the notation from the proof of Theorem
  \ref{thm:LS-SH}.  Assume first that the interval $I$ is finite and
  its end points are outside $\CS(\alpha)$. Let $H_j$ be a cofinal,
  increasing sequence of admissible Hamiltonians as in the proof of
  Theorem \ref{thm:LS-SH}. Then it is not hard to see that when $j$ is
  large enough the generators $x_j$ of the complex $\CF_*(H_j)^G$ from
  Proposition \ref{prop:quotient-Floer} are naturally in one-to-one
  correspondence with the good closed Reeb orbits $x$ with action in
  $I$. Furthermore, again when $j$ is large enough, a monotone
  homotopy from $H_j$ to $H_{j+1}$ induces an isomorphism
  $\CF_*(H_j)^G\to \CF_*(H_{j+1})^G$.  (This isomorphism has the form
  $id+\Phi$, where $\Phi$ is strictly Hamiltonian action decreasing.)
  Furthermore, when the functions $h_j$ are concave (i.e.,
  $h''_j\leq 0$), the Hamiltonian action filtration and the contact
  action filtration are equivalent:
  $\CA_{H_j}(x_j)\geq \CA_{H_j}(x_j')$ if and only if
  $\CA_{\alpha}(x)\geq \CA_{\alpha}(x')$.

  Thus we have a well-defined complex
  $\CF_*(\alpha)^G:=\lim \CF_*(H_j)^G$, graded by the Conley--Zehnder
  index and filtered by the contact action, with homology equal to
  $\SH_*^{G,I}(W)$. In fact, this complex with its grading and
  filtration is isomorphic to $\CF_*(H_j)^G$ for a large $j$. For any,
  not necessarily finite, interval $I$, the complex $\CF_*(\alpha)^G$
  with the required properties is constructed by exhausting $I$ by
  finite intervals and applying the diagonal process. (Here we again
  use the fact that the homology functor and the direct limit functor
  commute.)
\end{proof}

The proof of Theorem \ref{thm:LS-SH} also lends itself readily for the
definition of the local equivariant symplectic homology. Namely, let
$x$ be an isolated closed Reeb orbit on $M$. Then the corresponding
orbit $x_j$ of $H_j$ is also isolated, although the size of the
isolating neighborhood goes to zero in the $r$-direction as
$j\to\infty$. It is not hard to see that for any fixed degree $*$ the
equivariant Floer homology $\HF_*^G(x_j, H_j)$ stabilizes as
$j\to\infty$.

\begin{Definition}
  The \emph{equivariant local symplectic homology} $\SH_*^G(x)$ of $x$
  is by definition the equivariant local Floer homology
  $\HF_*^G(x_j, H_j)$ where $j$ is large enough.
\end{Definition}

\begin{Example}[Non-degenerate orbits]
  Assume that $x$ is non-degenerate. Then by Examples
  \ref{ex:local-nondeg-1} and \ref{ex:local-nondeg-2}, $\SH_*^G(x;\Q)$
  is one-dimensional and concentrated in degree equal to the
  Conley--Zehnder $\mu(x)$ when $x$ is good and zero when
  $x$ is bad.
\end{Example}

By Proposition \ref{prop:local-support}, $\SH_*^G(x;\Q)$ is supported
in $[\mu_-(x),\,\mu_+(x)]$, which in turn is contained in
$[\hmu(x)-n+1,\,\hmu(x)+n-1]$. The operator $D$ obviously descends to
$\SH_*^G(x)$. However, by Proposition \ref{prop:D-local}, the
resulting operator is trivial.

\begin{Corollary}
\label{cor:D-local}
The shift operator in the local equivariant symplectic homology of an
isolated orbit is identically zero: $D \equiv 0$ in $\SH_*^G(x;\Q)$.
\end{Corollary}

The argument from \cite {BO:12} readily translates to the proof of the
fact that $\SH_*^G(x;\Q)$ is isomorphic, up to a shift of degree, to
the local contact homology $\HC_*(x)$ introduced in \cite{HM}.
Alternatively, Proposition \ref{prop:quotient-SH} carries over to the
local case and $\SH_*^G(x;\Q)$ is isomorphic to the homology of a
certain complex graded by the Conley--Zehnder index and generated by
the good orbits that $x$ splits into under a non-degenerate
perturbation (cf.\ Remark \ref{rmk:diff}).  When $x$ is simple,
$\HC_*(x)$ is isomorphic to the local Floer homology of the Poincar\'e
return map $\varphi$ in $M$; see \cite{HM,GHHM}. When $x=y^k$, with
$y$ simple, $\HC_*(x)$ is expected to be isomorphic to the equivariant
local Floer homology $\HF_*^{\Z_k}(\varphi)$; see \cite{GHHM}.

\begin{Remark}[Generalizations and variations, II]
\label{rmk:general-W}
The results of this section readily extend to non-contractible
periodic orbits, although the conditions on $W$ cannot be relaxed to
the same degree as in Remark \ref{rmk:general-V}. To be more specific,
one can focus on closed Reeb orbits in a fixed free homotopy class
$\ff$ of loops in $W$. Then, for any $\ff$, the condition that
$\omega$ is exact can be replaced by that $\omega$ is aspherical and,
for instance, $\pi_1(M)\to \pi_1(W)$ is a monomorphism. (The role of
this condition is to ensure that the contact action in $M$ is equal to
the symplectic area in $W$ giving rise to the filtration in symplectic
homology. In general, the symplectic area can differ from the contact
action as can be seen from the example of the pre-quantization disk
bundle over a surface of genus $g\geq 2$.)  Furthermore, when
$\ff\neq 1$ we need to assume $c_1(TW)$ to be atoroidal. Note also
that the condition $I\subset (0,\,\infty)$ can be dropped in
Proposition \ref{prop:quotient-SH} and Theorem \ref{thm:LS-SH} when
$\ff\neq 1$.

When $W$ is exact and $c_1(TW)=0$, the equivariant symplectic homology
of $W$ is defined for all free homotopy classes. This homology is
naturally filtered by the action and graded by the free homotopy
class. It is clear that the analogs of the results from this section
including Theorem \ref{thm:LS-SH} hold in this case; cf.\ Remark
\ref{rmk:general-V}.
\end{Remark}

\subsection{Examples and applications}
\label{sec:LS-SH-examples} 
In this section, having our main applications to dynamics in mind, we
consider some simplest cases where Theorem \ref{thm:LS-SH} can be
utilized to produce non-obvious results: the standard contact sphere
$S^{2n-1}$, the boundary of a displaceable Liouville domain, and the
boundary of a Liouville domain in $T^*S^n$ containing the zero
section, e.g., the standard unit cotangent bundle $ST^*S^n$.

\subsubsection{The standard contact $S^{2n-1}$ and displaceable
  Liouville domains}
\label{sec:displ}
Let $\alpha$ be a contact form on $M=S^{2n-1}$ supporting the standard
contact structure. Then, as is well known, $(M,d\alpha)$ can be
embedded as a hypersurface in $\R^{2n}$ bounding a star-shaped domain
$W$. We take $I=[\delta,\infty)$ and work with the standard positive
equivariant symplectic homology $\SH_*^{G,+}(W;\Q)$, where $G=S^1$.
The homology is concentrated and equal to $\Q$ in every second degree
starting with $n+1$:
$$ 
\SH_*^{G,+}(W;\Q) =
\begin{cases}
  \Q & \text{for $*=n+1,\, n+3,\ldots $,} \\
   0 & \text{otherwise};
\end{cases}
$$
and moreover the shift operator 
$$
D\colon \Q=\SH_*^{G,+}(W;\Q) \to \SH_{*-2}^{G,+}(W;\Q)=\Q
$$ 
is an isomorphism for $*=n+3,\, n+5,\ldots$; see, e.g.,
\cite{BO:Gysin} and references therein. Therefore, there
exists a sequence of non-zero homology classes $w_k\in
\SH_{n+2k-1}^{G,+}(W;\Q)$, $k\in \N$, such that $Dw_{k+1}=w_k$. 

In a similar vain but slightly more generally, we may assume that
$(M,\alpha)$ is a restricted contact type hypersurface in $\R^{2n}$
bounding a region $W$. (This is automatically the case when, e.g.,
$W$ is a simply connected Liouville domain in $\R^{2n}$.)  Then we
have $\lambda W\subset B_r\subset W \subset B_R$, for two balls in
$\R^{2n}$ and some $\lambda>0$. Thus the cobordism map
$\SH^{G,+}_*(B_R;\Q)\to \SH^{G,+}_*(W;\Q)$ is an isomorphism, and we
conclude that the image of $w_k$ in $\SH^{G,+}_*(W;\Q)$ is
non-zero. Hence, as above, we have a well-defined sequence of non-zero
elements, which we still denote by $w_k$, of degree $n+2k-1$ such that
$Dw_{k+1}=w_k$. Set $\s_k=\s_{w_k}$, where $\s_{w_k}$ is defined
by~\eqref{eq:spec-inv}.

As an immediate consequence of Theorem
\ref{thm:LS-SH}, we obtain

\begin{Corollary}
\label{cor:sphere} 
Let $(M,\alpha)$ be a restricted contact type hypersurface in
$\R^{2n}$. Then there exists a \emph{carrier map}
$$
\psi\colon \N\to \PP(\alpha),\quad k\mapsto y_k
$$
such that $\s_k(\alpha)=\CA_\alpha(y_k)$ and
\begin{equation}
\label{eq:action<}
\CA_{\alpha}(y_1)\leq \CA_{\alpha}(y_2)\leq \CA_{\alpha}(y_3)\leq
\cdots
\quad\textrm{and}\quad
\mu_-(y_k)\leq n+2k-1\leq \mu_+(y_k).
\end{equation}
In particular, 
$$
\big|\hmu(y_k)-(n+2k-1)\big|\leq n-1.
$$
Furthermore, assume that all orbits in $\PP(\alpha)$ are
isolated. Then $\psi$ is an injection, $\SH_*^G(y_k)\neq 0$ in degrees
$*=n+2k-1$, and
$$
\s_1(\alpha)<\s_2(\alpha)<\s_3(\alpha)<\cdots.
$$
or, in other words,
$$
\CA_{\alpha}(y_1)<\CA_{\alpha}(y_2)<\CA_{\alpha}(y_3)<\cdots.
$$
Finally, when $\alpha$ is non-degenerate, the orbits $y_k$ are good.
\end{Corollary}
This corollary is a minor generalization of Theorem
\ref{thm:LS2-intro} from the introduction.  Note that, in general, the
carrier map is not unique. However, it becomes unique when, for
instance, all closed characteristics on $M$ have different actions.

\begin{proof} Assume first that all orbits in $\PP(\alpha)$ are
  isolated. It is easy to show by arguing as in, e.g., \cite{CGG} that
  for every $k\in \N$ there exists an orbit $y_k$ such that
  $\CA_{\alpha}(y_k)=\s_k(\alpha)$ and $\SH_{n+2k-1}^G(y_k;\Q)\neq 0$.
  (If such an orbit is not unique, we just pick one of them.)  Then
  the map $k\mapsto y_k$ is an injection due to the fact that the
  inequalities in Corollary \ref{cor:sphere} are strict. The second
  inequality relating $\hmu(y_k)$ and $k$ holds by
  Proposition \ref{prop:local-support}. The general case follows by
  continuity, but, of course, the strict inequalities and hence the
  injectivity of $\psi$ are lost in the process.
\end{proof}

\begin{Example}[Ellipsoids]
\label{ex:ellipsoids}
Let $M$ be the ellipsoid 
$$
\sum_j\frac{|z_j|^2}{r_j^2}=1
$$
in $\C^n$ with the standard contact form $\alpha$. Let us combine $n$
sequences $\pi r_j^2 k$, $k\in\N$, into one monotone increasing
sequence $\s_1\leq \s_2\leq \cdots $. Then this is exactly the
sequence of spectral invariants $\s_k(\alpha)$. The Reeb orbits on $M$
are isolated if and only if $r_j^2=q r_{j'}^2$ with $q\in\Q$ only when
$j=j'$, i.e., the sequences $\pi r_j^2 k$ do not overlap or,
equivalently, the sequence $\s_k$ is strictly increasing.
\end{Example}

\begin{Remark} When $W$ is convex, the spectral invariants
  $\s_k(\alpha)$ are believed to be equal to the Ekeland--Hofer
  capacities of $W$, \cite{EH}. In this case, variants of the
  corollary are known in a form not relying on the machinery of Floer
  or symplectic homology; see \cite{Ek,Lo}.
\end{Remark}

The only feature of the domain $W$ essential for Corollary
\ref{cor:sphere} is that there exists an infinite chain of non-zero
classes $w_k$ such that $Dw_{k+1}=w_k$. Thus, whenever $W$ is a
Liouville domain with such a chain and $c_1(TW)|_{\pi_2(W)}=0$, the
corollary holds (for the trivial free homotopy class) with $n+2k-1$ in the
index bounds replaced by the degree of~$w_k$. When the free homotopy
class is non-trivial, it suffices to require in addition that
$c_1(TW)=0$; cf.\ Remark \ref{rmk:general-W}.


For instance, let $W$ be a subcritical Stein manifold with
$c_1(TW)=0$. Then the above results extend to $W$ word-for-word for
the trivial free homotopy class. Indeed, by \cite[Cor.\
1.3]{BO:Gysin}, there exists a sequence of non-zero homology classes
$w_k\in \SH_{n+2k-1}^{G,+}(W;\Q)$, $k\in \N$, such that $Dw_{k+1}=w_k$
and thus Corollary \ref{cor:sphere} holds in this case exactly as
stated. In fact, every non-zero element in $\H_d(W,\p W;\Q)$ gives
rise to such a sequence starting with degree $d+1-n$. However, it is
not clear to us how to make use of these multiple sequences.

More generally, Corollary \ref{cor:sphere} holds when $W$ is a
Liouville domain displaceable in $\widehat{W}$ or even when $W$ is a
displaceable Liouville subdomain in some other Liouville manifold,
provided that $c_1(TW)=0$. Indeed, as is well known, then the full
ordinary symplectic homology of $W$ vanishes; see, e.g.,
\cite{CFO,Ri}. As a consequence, we also have $\SH^G_*(W)=0$. Now,
exactly as for $\R^{2n}$, the Gysin exact sequence implies that there
exists a sequence of non-zero homology classes
$w_k\in \SH_{n+2k-1}^{G,+}(W;\Q)$, $k\in \N$, such that
$Dw_{k+1}=w_k$; cf.\ \cite{BO:Gysin}.

\subsubsection{Simple Reeb orbits on $ST^*S^n$} 
\label{sec:T^*S^n-1}
Our next objective is to extend some of the results from Section
\ref{sec:displ} to Liouville domains in $T^*S^n$ containing the zero
section $S^n$. For the sake of simplicity we will asume throughout
this section that $n\geq 3$.

Let $W$ be a such a domain with smooth boundary and $M=\p W$. In other
words, $M$ is a restricted contact type hypersurface $M\subset T^*S^n$
enclosing $S^n$.  For instance, $M$ can be the unit cotangent bundle
$ST^*S^n$ with respect to the round metric and $W$ is then the unit
disk bundle $W_1$, or we can take as $W$ any compact fiberwise
star-shaped domain with smooth boundary.

\begin{Proposition}
  \label{prop:D-STS} 
  We have
$$
\dim\SH_{k}^{G,+}(W;\Q)=
\begin{cases}
0 & \textrm{for $k\equiv n\pmod 2$,}\\
2 & \textrm{for $k=j(n-1)$ for all $j>1$ when $n$ is odd,}\\ 
2 & \textrm{for $k=j(n-1)$ for odd $j>1$ when $n$ is even,}\\
1 & \textrm{for all other $k\equiv n-1\pmod 2$.}
\end{cases}
$$
For every $j\in\N$, there exist $n$ non-zero elements
$w_i\in \SH_{2i+(2j-1)(n-1)}^{G,+}(W;\Q)$, $i=0,\,\ldots,\,n-1$, such
that $Dw_{i+1}=w_i$.
\end{Proposition}

\begin{Remark}
\label{rmk:prequant-1}
The first part of the proposition is standard and included only for
the sake of completeness; see, e.g., \cite{KvK}.  The second part of
Proposition \ref{prop:D-STS} holds in a more general setting than
considered here. Namely, with suitable modifications, it holds when
$M$ is a pre-quantization $S^1$-bundle over a closed symplectic
manifold $B^{2n}$ and $W$ is the disk bundle. Note also that in
general $\SH_{*}^{G,+}(W;\Q)$ may have a much longer sequence of
non-vanishing classes with $Dw_{i+1}=w_i$ than the sequences of length
$n$ coming from the proposition. For instance, $S^{2n-1}$ admits such
an infinite sequence; see Section \ref{sec:displ}. We do not know how
long a sequence for $ST^*S^n$ can actually be.
\end{Remark}

\begin{proof}
  First, we note that it is enough to prove the proposition for the
  unit disk bundle $W_1$. Indeed, it is easy to see from the
  inclusions $\lambda W\subset W_R\subset W \subset W_{R'}$, where
  $W_R$ and $W_{R'}$ are disk bundles and $\lambda>0$ is sufficiently
  small, that by Proposition \ref{prop:cobordism}
$$
\SH_{*}^{G,+}(W;\Q)\cong \SH_{*}^{G,+}(W_1;\Q).
$$

The homology $\SH_{*}^{G,+}(W_1;\Q)$ can be easily calculated using
the Morse--Bott techniques. Indeed, recall that $ST^*S^n$ is a
$G=S^1$-principal bundle whose base $B$ is the Grassmannian
$\Gr^+(2,n+1)$ of oriented two-planes in $\R^{n}$. The Reeb flow on
$ST^*S^n$ is the geodesic flow on $S^n$ with respect to the round
metric. Hence, the flow is Morse--Bott in the sense of
\cite{Bo:thesis,Es,Poz} and its ``critical sets'' $P_j$ are formed by
$j$-iterated geodesics for all $j\in\N$. The set $P_j$ is
equivariantly diffeomorphic to $ST^*S^n$ with the $G$-action obtained
by combining the standard action with the $j$-fold covering $G\to G$.
By Proposition \ref{prop:MB} and since $\H_*(B;\Q)$ vanishes in odd
degrees, $\SH_{*}^{G,+}(W_1;\Q)$ breaks down into the sum of infinite
number of terms $\H_*^G(P_j;\Q)$ up to a shift of degree. The
$G$-action on $P_j$ is locally free and $P_j/G=B$. Thus,
$\H_*^G(P_j;\Q)=\H_*(B;\Q)$.

In fact, $\H_*(B;\Q)$ has one generator $w'_i\neq 0$ in every even
degree $i=0,2,\ldots, 2n-2$ and also, when $n-1$ is even, one extra
generator in degree $n-1$ (the middle of the range); see, e.g.,
\cite{KvK} and references therein. In all other degrees the homology
is zero.

With our conventions, the shift for $P_j$ is equal to $-(n-1)+j\Delta$
where $\Delta$ is the mean index of a closed Reeb orbit in $P_1$ (a
simple geodesic) or, equivalently,
$\Delta=2\left<c_1(ST^*S^n\to B),u\right>$, where $u$ is a suitably
oriented generator of $\pi_2(B)\cong\Z$; cf.\ \cite[Ex.\ 8.2]{Es}. We
have $\Delta=2(n-1)$ and thus
\begin{equation}
\label{eq:coh-STS}
\SH_{*}^{G,+}(W_1;\Q)=\bigoplus_j \H_{*-(2j-1)(n-1)}(B;\Q),
\end{equation}
where every term in the sum contributes to the homology in degrees of
the same parity as $n-1$ in the interval
$$
\big[(2j-1)(n-1),\ldots,(2j+1)(n-1)\big]
$$
centered at $2(n-1)j$; see \cite[Sect.\ 5.6]{KvK}. 

The operator $D$ on $\H_*^G(P_j;\Q)=\H_*(B;\Q)$ is Poincar\'e dual (up
to a factor) to the multiplication by $c_1(ST^*S^n\to B)$, i.e., by
the cohomology class of the standard symplectic structure on
$\Gr^+(2,n+1)$. It follows that $Dw'_{i+1}=w'_i$ for a suitable choice
of the generators $w'_i\in\H_{2i}(B;\Q)$.  Now, fixing $j\in\N$, we
let $w_i$ be the image of $w'_i$ under the identification
\eqref{eq:coh-STS}. Then, by Proposition \ref{prop:MB},
$Dw_{i+1}=w_i\neq 0$.
\end{proof}

As an immediate consequence of Proposition \ref{prop:D-STS}, we obtain
the following analog of Corollary \ref{cor:sphere} generalizing
Theorem \ref{thm:STSn-intro} from the introduction.

\begin{Corollary}
\label{cor:STSn}
Let $(M,\alpha)$ be a restricted contact type hypersurface in
$T^*S^n$ enclosing the zero section and such that all periodic orbits
of the Reeb flow are isolated. Then, in the notation of Proposition
\ref{prop:D-STS}, for every $j\in \N$, there exist $n$ periodic orbits
$y_0,\ldots,y_{n-1}$ such that $\s_{w_i}(\alpha)=\CA_\alpha(y_i)$ and
$\SH_*^G(y_i;\Q)\neq 0$ in degrees $*=(2j-1)(n-1)+2i$. In particular,
\begin{equation}
\label{eq:STSn-action}
\CA_{\alpha}(y_0)< \CA_{\alpha}(y_1)<\cdots< \CA_{\alpha}(y_{n-1}),
\end{equation}
and
\begin{equation}
\label{eq:STSn-index}
\big|\hmu(y_i)-\big((2j-1)(n-1)+2i\big)\big|\leq n-1.
\end{equation}
\end{Corollary}

\begin{Remark}
Without the assumption that the closed Reeb orbits of $\alpha$ are
isolated, we still have $n$ orbits $y_i$ such that
$\s_{w_i}(\alpha)=\CA_\alpha(y_i)$ and \eqref{eq:STSn-index} is
satisfied, but now the action inequalities \eqref{eq:STSn-action}
are not necessarily strict:
$$
\CA_{\alpha}(y_0)\leq \CA_{\alpha}(y_1)\leq\cdots\leq \CA_{\alpha}(y_{n-1}).
$$
\end{Remark}

\begin{Remark}[Action carriers]
\label{rmk:carier}
Putting the constructions from this section and Section
\ref{sec:displ} in a more formal context, we could have introduced
action carriers for spectral invariants in the equivariant Floer or
symplectic homology similarly to the carriers in \cite{CGG}. Thus,
for instance, when $W$ is exact with $c_1(TW)|_{\pi_2(W)}=0$ and all
orbits in $\PP(\alpha)$ are isolated, we would have a map
$\SH_*^{G,+}(W;\Q)\setminus\{0\}\to \PP(\alpha)$ such that $w$ and
$Dw\neq 0$ are never mapped to the same orbit.
\end{Remark}

\section{Index theory}
\label{sec:index}
\subsection{Preliminaries: definitions and basic facts}
\label{sec:index-basics}
In this section we recall for the reader's convenience some basic
properties of the mean and Conley--Zehnder indices. We refer the
reader to, e.g., \cite{Lo} or \cite[Sect.\ 3]{SZ} for a more thorough
treatment; see also \cite{Ab,Gut} and, for a very quick introduction,
\cite[Sect.\ 2.4]{Sa:notes}. 

\subsubsection{Definitions}
\label{sec:index-def}
To every continuous path $\Phi\colon [0,\,1]\to\Sp(2m)$ beginning at
$\Phi(0)=I$, one can associate the \emph{mean index}
$\hmu(\Phi)\in \R$, a homotopy invariant of the path with fixed
end-points.  To give a formal definition, recall first that a map
$\hmu$ from a Lie group to $\R$ is said to be a quasimorphism if it
fails to be a homomorphism only up to a constant, i.e.,
$$
\big|\hmu(\Phi\Psi)-\hmu(\Phi)-\hmu(\Psi)\big|<\const,
$$ 
where the constant is independent of $\Phi$ and $\Psi$.  One can prove
that there is a unique quasimorphism
$\hmu\colon \widetilde{\Sp}(2m)\to\R$, on the universal covering
$\TSp(2m)$ of the symplectic group, which is continuous and
homogeneous (i.e., $\hmu(\Phi^k)=k\hmu(\Phi)$) and satisfies the
normalization condition:
$$
\hmu(\Phi_0)=2\quad \textrm{for}\quad \Phi_0(t)= \exp\big(2\pi \sqrt{-1}
t\big)\oplus I_{2m-2}
$$ 
with $t\in [0,\,1]$, in the self-explanatory notation; see
\cite{BG}. In particular, $\hmu$ restricts to an isomorphism
$\pi_1(\Sp(2m))\to 2\Z$.  The quasimorphism $\hmu$ is the mean
index. The continuity requirement holds automatically and is not
necessary for the characterization of $\hmu$, although this is not
immediately obvious. Furthermore, $\hmu$ is also automatically
conjugation invariant, as a consequence of the homogeneity.

The mean index $\hmu(\Phi)$ measures the total rotation angle of
certain unit eigenvalues of $\Phi(t)$ and can be explicitly defined as
follows. Following \cite{SZ}, for an elliptic transformation
$A\in\Sp(2)$, let us say that an eigenvalue
$\exp(\sqrt{-1}\theta)\in S^1$ of $A$ is of the first kind if
$0\leq\theta\leq \pi$ and $A$ is conjugate to the rotation in $\theta$
counterclockwise or if $-\pi<\theta<0$ and $A$ is conjugate to the
rotation in $\theta$ clockwise. (This rule unambiguously picks one of
the two eigenvalues of an elliptic tranformation $A\in \Sp(2)$.) We
set $\rho(A)$ to be equal its eigenvalue of the first kind when $A$ is
elliptic and $\rho(A)=\pm 1$ when $A$ is hyperbolic with the sign
determined by the sign of the eigenvalues of $A$. Then
$\rho\colon \Sp(2)\to S^1$ is a Lipschitz (but not $C^1$) function,
which is conjugation invariant and equal to $\det$ on $\U(1)$. A
matrix $A\in\Sp(2m)$ with distinct eigenvalues, can be written as the
direct sum of matrices $A_j\in\Sp(2)$ and a matrix with complex
eigenvalues not lying on the unit circle. We set $\rho(A)$ to be the
product of $\rho(A_j)\in S^1$.  Again, $\rho$ extends to a continuous
function $\rho\colon \Sp(2m)\to S^1$, which is conjugation invariant
(and hence $\rho(AB)=\rho(BA)$) and restricts to $\det$ on $\U(n)$;
see, e.g., \cite{SZ}. Finally, given a path
$\Phi\colon [0,\,1]\to \Sp(2m)$, there is a continuous function
$\theta(t)$ such that
$\rho(\Phi(t))=\exp\big(\sqrt{-1}\theta(t)\big)$, measuring the total
rotation of the eigenvalues of the first type, and we set
$$
\hmu(\Phi)=\frac{\theta(1)-\theta(0)}{\pi}.
$$ 
It is clear from the definition that $\hmu(\Phi_s)=\const$ for a
family of paths $\Phi_s$ as long as the eigenvalues of $\Phi_s(1)$
remain constant.

Assume now that the path $\Phi$ is \emph{non-degenerate}, i.e., by
definition, all eigenvalues of the end-point $A=\Phi(1)$ are different
from one. We denote the set of such matrices $A\in\Sp(2m)$ by
$\Sp^*(2m)$ and also denote the part of $\TSp(2m)$ lying over
$\Sp^*(2m)$ by $\TSpn(2m)$. It is not hard to see that $A$ can
be connected to a symplectic transformation with elliptic part equal
to $-I$ (if non-trivial) by a path $\Psi$ lying entirely in
$\Sp^*(2m)$. Concatenating this path with $\Phi$, we obtain a new path
$\Phi'$. By definition, the \emph{Conley--Zehnder index}
$\mu(\Phi)\in\Z$ of $\Phi$ is $\hmu(\Phi')$. One can show that
$\mu(\Phi)$ is well-defined, i.e., independent of $\Psi$. The function
$\mu\colon \TSpn(2m)\to\Z$ is locally constant, i.e., constant on
connected components of $\TSpn(2m)$. In other words,
$\mu(\Phi_s)=\const$ for a family of paths $\Phi_s$ as long as
$\Phi_s(1)\in\Sp^*(2m)$.

Furthermore, let us call $\Phi$ \emph{weakly non-degenerate} if at
least one eigenvalue of $\Phi(1)$ is different from one and
\emph{totally degenerate} otherwise. A path is \emph{strongly
  non-degenerate} if all its ``iterations'' $\Phi^k$ are
non-degenerate, i.e., none of the eigenvalues of $\Phi(1)$ is a root
of unity. The multiplicity of the generalized eigenvalue one of
$\Phi(1)$ is an even number, which we denote by $2\nu(\Phi)$
and call $\nu(\Phi)$ the \emph{nullity} of $\Phi$.

Following, e.g., \cite{Lo90,Lo97}, we define the \emph{upper and lower
  Conley--Zehnder indices} as
$$
\mu_+(\Phi):=\limsup_{\tPhi\to \Phi}\mu(\tPhi)
\quad\textrm{and}\quad
\mu_-(\Phi):=\liminf_{\tPhi\to \Phi}\mu(\tPhi),
$$
where in both cases the limit is taken over $\tPhi\in \TSpn(2m)$
converging to $\Phi\in \TSp(2m)$. In fact, $\mu_+(\Phi)$ is simply
$\max\mu(\tPhi)$, where $\tPhi\in \TSpn(2m)$ is sufficiently close to
$\Phi$ in $\TSp(2m)$; likewise, $\mu_-(\Phi)=\min \mu(\tPhi)$. (In
terms of actual paths, rather than their homotopy classes, $\tPhi$ can
be taken $C^r$-close to $\Phi$ for any $r\geq 0$; the resulting
definition of $\mu_\pm(\Phi)$ is independent of $r$ and equivalent to
the one above due to homotopy invariance of $\mu$.)  Clearly,
$\mu(\Phi)=\mu_\pm(\Phi)$ when $\Phi$ is non-degenerate. As readily
follows from the definition, the indices $\mu_\pm$ are the upper
semi-continuous and, respectively, lower semi-continous extensions of
$\mu$ from $\TSpn(2m)$ to $\TSp(2m)$. The upper and lower
Conley--Zehnder indices are quasimorphisms $\TSp(2m)\to\Z$. These
indices are of particular interest to us because they bound the
support of the local Floer homology of an isolated periodic orbit;
cf.\ Proposition \ref{prop:local-support} and \cite {GG:gap}.

\subsubsection{Basic properties}
\label{sec:index-prop}
Let us now list the properties of the Conley--Zehnder type indices,
which are essential for our purposes. Most of these properties readily
follow from the definitions and are well-known; see, e.g.,
\cite{Lo,SZ}. In what follows, all paths are required to begin at $I$
and are taken up to homotopy, i.e., as elements of
$\TSp(2m)$. Furthermore, we will tacitly assume the paths to be
parametrized by $[0,\,1]$ unless this is obviously not the case.

We start with three specific examples. For the path
$\Phi(t)=\exp\big(2\pi\sqrt{-1}\lambda t\big)$, $t\in [0,\,1]$, in
$\Sp(2)$ we have
\begin{equation}
\label{eq:sp2}
\hmu(\Phi)=2\lambda\textrm{ and } 
\mu(\Phi)=\sign(\lambda)\big(2\lfloor|\lambda|\rfloor +1\big)
\textrm{ when $\lambda\not\in\Z$. }
\end{equation}
Next, let $H$ be a non-degenerate quadratic form on $\R^{2m}$
with eigenvalues in the range $(-\pi,\,\pi)$. (Here, as is customary
in Hamiltonian dynamics, the eigenvalues of a quadratic form $H$ on a
symplectic vector space are by definition the eigenvalues of its
Hamiltonian vector field $X_H=J \nabla H$, where $J$ is the matrix of
the symplectic form; see, e.g., \cite{Ar}.) The path
$\Phi(t)=\exp(JHt)$, $t\in [0,\,1]$, is the linear autonomous
Hamiltonian flow generated by $H$. Then, with our conventions,
\begin{equation}
\label{eq:sgn}
\mu(\Phi)=\frac{1}{2}\sgn(H),
\end{equation}
where $\sgn(H)$ is the signature of $H$, i.e., the number of positive
squares minus the number of negative squares in the diagonal form of
$H$ with $\pm 1$ and $0$ on the diagonal. In addition, when $\Phi(1)$
is hyperbolic, we have
$$
\mu(\Phi)=\hmu(\Phi).
$$

Let us now list some ``additive and multiplicative properties''
which, combined with the examples above, allow one to calculate the
indices in many cases.  We start by observing that
$$
\mu(\Phi^{-1})=-\mu(\Phi)
$$
for any non-degenerate path $\Phi$ and hence, in general,
\begin{equation}
\label{eq:inv}
\mu_\pm(\Phi^{-1})=\mp\mu_\mp(\Phi).
\end{equation}
When $\varphi$ is a loop, we have
\begin{equation}
\label{eq:loop}
\mu_\pm(\varphi\Phi)=\hmu(\varphi)+\mu_\pm(\Phi).
\end{equation}
Finally, as readily follows from the definitions, $\hmu$ and $\mu$
are additive under direct sum. Namely, for $\Phi\in\TSp(2m)$ and
$\Psi\in\TSp(2m')$, we have
\begin{equation}
\label{eq:add}
\hmu(\Phi\oplus\Psi)=\hmu(\Phi)+\hmu(\Psi)
\quad\textrm{and}\quad
\mu(\Phi\oplus\Psi)=\mu(\Phi)+\mu(\Psi),
\end{equation}
where in the second identity we assumed that both paths are
non-degenerate. We will extend this additivity property to $\mu_\pm$
in Lemma \ref{lemma:additivity}.

The mean index and the upper and lower indices are related by the
inequalities
\begin{equation}
\label{eq:mu-del}
\hmu(\Phi)-m\leq\mu_-(\Phi)\leq \mu_+(\Phi)\leq \hmu(\Phi)+m,
\end{equation}
where $\Phi\in\TSp(2m)$. (Moreover, at least one of the inequalities
is strict when $\Phi$ is weakly non-degenerate.) As a consequence,
$$
\lim_{k\to\infty}\frac{\mu_\pm(\Phi^k)}{k}=\hmu(\Phi),
$$
 hence the name ``mean index'' for $\hmu$.

 The Conley--Zehnder index can also be evaluated as an intersection
 index with the discriminant $\Sigma=\{A\in\Sp(2m)\mid \det(A-I)=0\}$,
 leading to another extension of $\mu$ from $\TSpn(2m)$ to $\TSp(2m)$
 known as the \emph{Robbin--Salamon index}. This index, introduced in
 \cite{RS}, is a quasimorphism $\TSp(2m)\to\frac{1}{2}\Z$. Let us
 briefly recall its definition.  For a path $\Phi$ in $\Sp(2m)$,
 beginning at $I$ and generated by a quadratic time-dependent
 Hamiltonian $H_t$, i.e., $\dot{\Phi}(t)=JH_t\circ\Phi(t)$, let us
 call $\tau\in [0,\,1]$ a \emph{crossing} if $\Phi(\tau)\in\Sigma$. A
 crossing is non-degenerate if the quadratic form
 $Q_\tau:=H_\tau|_{V_\tau}$, where $V_\tau:=\ker(\Phi(\tau)-I)$, is
 non-degenerate. Geometrically, that means that $\dot{\Phi}(\tau)$ is
 not tangent to the stratum of $\Sigma$ through
 $\Phi(\tau)$. (Generically, all crossings $\tau$ are non-degenerate
 and the interior crossings $\tau\in (0,\,1)$ are also \emph{simple},
 i.e., $\dim V_\tau=1$.) Non-degenerate crossings are isolated. Now,
 for a path $\Phi$ with only non-degenerate crossings, we set
\begin{equation}
\label{eq:RS}
\MURS(\Phi)=\frac{1}{2}\sgn Q_0+\sum_{0<\tau<1}\sgn Q_\tau 
+ \frac{1}{2}\sgn Q_1,
\end{equation}
where the last term is dropped when $1$ is not a crossing. Then
$\MURS$ is homotopy invariant, $\MURS(\Phi)=\mu(\Phi)$ when
$\Phi\in\TSpn(2m)$ (see \cite{RS}), and $\mu_-\leq \MURS\leq \mu_+$.
(In fact, $\MURS=(\mu_++\mu_-)/2$, but we do not need this relation.)

\begin{Example}[Positive definite Hamiltonians]
\label{ex:pos-def}
Assume that $H_t>0$ on $\R^{2m}$, i.e., $H_t$ is positive definite,
for all $t\in [0,\,\infty)$, and let $\Phi_t$ be generated by
$H_t$. Then, as readily follows from \eqref{eq:RS},
$\MURS\big(\Phi|_{[0,\,T]}\big)$ and $\hmu\big(\Phi|_{[0,\,T]}\big)$
are increasing functions of $T>0$ and
$\mu_-\big(\Phi|_{[0,\,T]}\big)\geq m$ for all $T>0$.
\end{Example}

\subsubsection{Further details: totally degenerate paths, additivity of
  $\mu_\pm$,  and signature multiplicities $b_{0,\pm}$}
\label{sec:tot-deg}
In this subsection, we establish two additional properties of the
indices $\mu_\pm$, somewhat less standard than the facts from the
previous subsection, although still well-known to the experts. For the
sake of completeness we provide detailed proofs. We also
discuss some finer invariants of symplectic paths and their relations
with the indices $\mu_\pm$.

\begin{Lemma}
\label{lemma:iterate} 
Assume that $A=\Phi(1)$ is totally degenerate, i.e., all eigenvalues of
$A$ are equal to one. Then $\Phi$ is homotopic to the product
of a loop and a path $\Psi$ such that all eigenvalues of $\Psi(t)$ for
all $t$ are equal to one. Furthermore,
\begin{equation}
\label{eq:iterate}
\mu_\pm(\Phi^k)=(k-1)\hmu(\Phi)+\mu_\pm(\Phi) \quad\textrm{for all $k\in\N$}.
\end{equation}
\end{Lemma}

\begin{proof}
  Fix a small ball $\CB$ in the space of quadratic forms on $\R^{2m}$
  centered at $0$.  Then the exponential mapping
  $\exp\colon \CB\to \Sp(2m)$ is a diffeomorphism onto its
  image. Recall that, as is well-known, the matrix $A$ is
  symplectically conjugate to a matrix arbitrarily close to $I$; see,
  e.g., \cite[Lemma 5.1]{Gi:CC}. Thus, after applying a conjugation,
  we can assume that $A\in \exp(\CB)$.  In other words, there exists a
  small quadratic form $Q$ on $\R^{2m}$ with only zero eigenvalues
  such that
$$
A=\Psi(1),\quad\textrm{where}\quad 
\Psi(t)=\exp(JQt).
$$
As a consequence, $\Phi$ can be written as the concatenation, and
hence the product, of a loop $\varphi$ and a path $\Psi$ such that for
every $t$ all eigenvalues of $\Psi(t)$ are equal to one even if $A$ is
not in $\exp(\CB)$. This proves the first assertion of the lemma.

To prove the second assertion, we can assume that $A\in\exp(\CB)$ and
$Q$ is small. By \eqref{eq:loop} and since $\hmu(\Psi)=0$, we have
\begin{equation}
\label{eq:Phi-Psi}
\hmu(\varphi)=\hmu(\Phi) \quad\textrm{and}\quad
\mu_\pm(\Phi)=\hmu(\varphi)+\mu_\pm(\Psi). 
\end{equation}
Set $\nu_0(Q)=\dim\ker Q$ and  
$$
\sgn_+(Q):=\max \sgn(\tQ)=\sgn(Q)+\nu_0
$$
and
$$
\sgn_-(Q):=\min \sgn(\tQ)=\sgn(Q)-\nu_0,
$$
where $\max$ and $\min$ are taken over all small non-degenerate
perturbations $\tQ$ of $Q$.  The exponential mapping is a
diffeomorphism from a small neighborhood of $Q$ onto a small
neighborhood of $A$. Hence, by \eqref{eq:sgn}, we have
\begin{equation}
\label{eq:sgn2}
\mu_\pm(\Psi)=\frac{1}{2}\sgn_\pm(Q).
\end{equation}

Let us now apply this argument to the path $\Psi^k$ ending at $A^k$
for a fixed $k$. Again, without changing the index we can assume that
$A^k\in\exp(\CB)$ and $A\in\exp(\CB)$. We have $A^k=\exp(JH)$ and
$A=\exp(JQ)$ where $H$ and $kQ$ are in $\CB$. Since $\exp$ is
one-to-one on $\CB$, it readily follows that $H=kQ$. Thus
$$
\mu_\pm(\Psi^k)=\frac{1}{2}\sgn_\pm(kQ)=\frac{1}{2}\sgn_\pm(Q)=\mu_\pm(\Psi).
$$
This proves \eqref{eq:iterate} for $\Psi$ since $\hmu(\Psi)=0$. The
result for the original path $\Phi$ follows from~\eqref{eq:Phi-Psi}
and the fact that $\pi_1(\Sp(2m))$ lies in the center of $\TSp(2m)$.
\end{proof}

The last general property of the index essential for us is the direct
sum additivity for $\mu_\pm$ generalizing the second identity in
\eqref{eq:add}.

\begin{Lemma}[\cite{Lo97}]
\label{lemma:additivity}
The upper and lower Conley--Zehnder indices $\mu_\pm$ are additive
with respect to direct sum, i.e., for $\Phi\in\TSp(2m)$ and
$\Psi\in\TSp(2m')$, we have
\begin{equation}
\label{eq:additivity}
\mu_\pm(\Phi\oplus\Psi)=\mu_\pm(\Phi)+\mu_\pm(\Psi).
\end{equation}
\end{Lemma}
This result, which is a part of \cite[Thm.\ 1.4]{Lo97}, is also less
known than most of the properties listed in Section
\ref{sec:index-prop} and not obvious. Since it plays a
crucial role in our argument, we briefly outline the proof for
the sake of completeness.

\begin{Remark}
\label{rmk:subadditivity}
Note that the lemma is not a direct consequence of the definition of
the indices $\mu_\pm$. Namely, it only follows from the definitions
that $\mu_+$ is sup-additive and $\mu_-$ is sub-additive. However, by
\eqref{eq:inv}, it is enough to prove the lemma only for one of these
indices.
\end{Remark}

\begin{proof}
  The lemma holds for the direct sum of two non-degenerate paths by
  \eqref{eq:add}. From \eqref{eq:Phi-Psi} and \eqref{eq:sgn2} and the
  fact that $\sgn_\pm$ are clearly direct sum additive, we
  observe that the lemma holds for the direct sum of paths with totally
  degenerate end-points.

  Every path $\Phi\in\TSp(2m)$ can be decomposed, up to homotopy, as
  the direct sum of a non-degenerate path $\Phi_1\in\TSpn(2m_1)$ and a
  path $\Phi_0\in\TSp(2m_0)$ with totally degenerate $\Phi_0(1)$.
  Here $m=m_0+m_1$. Thus it is sufficient to show that
\begin{equation}
\label{eq:add-step}
\mu_\pm (\Phi)=\mu(\Phi_1)+\mu_\pm (\Phi_0).
\end{equation}
Let $\tPhi$ be a small non-degenerate perturbation of $\Phi$. The
eigenvalues of $\tPhi(1)$ can be broken down into two groups. The
first group is formed by the eigenvalues close to one; the sum $V_0$
of their generalized eigenspaces has dimension $2m_0$ and is close to
$\R^{2m_0}$.  The second group is formed by the eigenvalues close to
the eigenvalues of $\Phi(1)$ different from one; the sum $V_1$ of
their generalized eigenspaces has dimension $2m_1$ and is close to
$\R^{2m_1}$. It is not hard to see that $\tPhi$ can be deformed, in a
neighborhood of $\Phi\in\TSp(2m)$, to a path $\tPhi'$ such
that $\tPhi'(1)$ has the same eigenvalues as $\tPhi(1)$ and
$V_0=\R^{2m_0}$ and $V_1=\R^{2m_1}$, and moreover the eigenvalues of
the end-point map remain constant in the process of deformation. Thus
$\mu(\tPhi)$ also remains constant in the process of deformation, and
in particular $\mu(\tPhi')=\mu(\tPhi)$.

As an element of $\TSpn(2m)$, the path $\tPhi'$ decomposes into the
sum of a path $\tPhi'_0\in\TSpn(2m_0)$ close to $\Phi_0$ and a path
$\tPhi'_1\in\TSpn(2m_1)$ close to $\Phi_1$. (This follows from the
fact that the map from $\TSp(2m_0)\times \TSp(2m_1)$ to the part of
$\TSp(2m)$ lying above $\Sp(2m_0)\times \Sp(2m_1)$ is a covering map.)
By non-degeneracy, $\mu(\tPhi'_1)=\mu(\Phi_1)$. Hence
$$
\mu(\tPhi)=\mu(\tPhi')=\mu(\Phi_1)+\mu(\tPhi'_0)
$$
and, therefore, $\mu_+(\Phi)\geq \mu(\Phi_1)+\mu_+ (\Phi_0)$. (For
$\mu_-$, we have the opposite inequality.) Combining this with the
fact that $\mu_+$ is sup-additive (see Lemma \ref{lemma:additivity}),
we obtain \eqref{eq:add-step} for $\mu_+$. For $\mu_-$ the result
follows from a similar argument or from \eqref{eq:inv}. This concludes
the proof of the lemma.
\end{proof}

We finish this subsection by introducing certain invariants of $\Phi$,
which we call signature multiplicities and the absolute nullity. These
invariants play a central role in the index recurrence theorem
(Theorem \ref{thm:IRT2}).

Consider first a totally degenerate operator $A\in\Sp(2m)$. (In other
words, we require all eigenvalues of $A$ to be equal to one.) Then, as
we have seen in the proof of Lemma \ref{lemma:iterate}, $A=\exp(JQ)$
where all eigenvalues of $Q$ equal zero. The quadratic form
$Q$ can be symplectically decomposed into a sum of terms of four
types:
\begin{itemize}
\item the identically zero quadratic form on $\R^{2\nu_0}$,
\item the quadratic form $Q_0=p_1q_2+p_2q_3+\cdots+p_{d-1}q_d$ in
  Darboux coordinates on $\R^{2d}$, where $d\geq 1$ is odd,
\item the quadratic forms $Q_\pm=\pm(Q_0+p^2_d/2)$ on $\R^{2d}$ for
  any $d$.
\end{itemize}
(We find these normal forms, taken from \cite[Sect.\ 2.4]{AG}, more
convenient to work with than the original Williamson normal forms; see
\cite{Wi} and also, e.g., \cite[App.\ 6]{Ar}.) Clearly,
$\dim\ker Q_0=2$ and $\sgn Q_0=0$, and $\dim\ker Q_\pm=1$ and
$\sgn Q_\pm=\pm 1$.  Let $b_*(Q)$, where $*=0,\pm$, be the number of
the $Q_0$ and $Q_\pm$ terms in the decomposition. Let us also set
$b_*(A):=b_*(Q)$ and $\nu_0(A):=\nu_0(Q)$.  These are symplectic
invariants of $Q$ and $A$. Then, arguing as in the proof of Lemma
\ref{lemma:iterate}, we see that for a path $\Phi$ with $\Phi(1)=A$,
we have
\begin{equation}
\label{eq:bpm}
\mu_+(\Phi)=\hmu(\Phi)+b_0+b_+ +\nu_0
\quad\textrm{and}\quad
\mu_-(\Phi)=\hmu(\Phi)-b_0-b_- -\nu_0.
\end{equation}

These formulas readily extend to all paths. Namely, every
$\Phi\in\TSp(2m)$ can be written (non-uniquely) as a product of a loop
$\varphi$ and the direct sum $\Psi_0\oplus \Psi_1$ where
$\Psi_0\in\TSp(2m_0)$ is a totally degenerate (for all $t$) path
$\Psi_0(t)=\exp(JQt)$ and $\Psi_1\in\TSpn(2m_1)$. In particular,
$m_0=\nu(\Phi)$ and $m_0+m_1=m$. (Note that $\varphi$ can be absorbed
into $\Psi_1$ unless $m_1=0$.)

\begin{Definition}
\label{def:tg}
The \emph{signature multiplicities} and the \emph{absolute nullity} of
$\Phi$ are
$$
b_*(\Phi):=b_*(\Psi_0)\textrm{ for $*=0,\pm$ and }
\nu_0(\Phi):=\nu_0(\Psi_0).
$$
\end{Definition}
One can show that these are symplectic invariants of $\Phi$, and
$$
\mu_+(\Phi)=
\hmu(\varphi)+\mu(\Psi_1)+b_0+b_+ +\nu_0
$$
and
$$
\mu_-(\Phi)=\hmu(\varphi)+\mu(\Psi_1)-b_0-b_- -\nu_0.
$$

\subsection{Dynamical convexity}
\label{sec:index:DC}
The notion of dynamical convexity was originally introduced in
\cite{HWZ:convex} for Reeb flows on the standard contact $S^3$ and,
somewhat in passing, for higher-dimensional contact spheres. In
\cite{AM:elliptic,AM} the definition was extended to other contact
manifolds. Here we mainly focus on the linear algebra aspect of
dynamical convexity, which is more essential for our purposes. Thus
our entire approach is quite different from that in
\cite{AM:elliptic,AM}. It is convenient for us to adopt the following
definition.

\begin{Definition}
\label{def:DC} A path $\Phi\in \TSp(2m)$ is said to be \emph{dynamically
convex} (DC) if $\mu_-(\Phi)\geq m+2$.
\end{Definition}

\begin{Example}
  \label{ex:DC} Assume that $\Phi$ is generated by a positive definite
  Hamiltonian; see Example \ref{ex:pos-def}. Then $\Phi$ need not be
  dynamically convex. However, by \eqref{eq:RS}, $\Phi$ is dynamically
  convex if it has at least two simple interior crossings or at least
  one interior crossing with multiplicity two.
\end{Example}

\begin{Lemma}
\label{lemma:DC}
For any $\Phi \in \TSp(2m)$, we have
\begin{equation}
\label{eq:DC-it}
\mu_-(\Phi^{k+1})\geq \mu_-(\Phi^{k})+\big(\mu_-(\Phi)-m\big)
\end{equation}
for all $k\in\N$. In particular,
$$
\mu_-(\Phi^k)\geq (\mu_-(\Phi)-m)k+m.
$$
Assume furthermore that $\Phi$ is dynamically convex.  Then the
function $\mu_-(\Phi^k)$ of $k\in\N$ is strictly increasing,
$$
\mu_-(\Phi^k)\geq 2k+m,
$$
and $\hmu(\Phi)\geq \mu_-(\Phi)-m\geq 2$. Thus $\mu_-(\Phi^k)\geq m+2$
and all iterations $\Phi^k$ are also dynamically convex.
\end{Lemma}

The proof of the lemma is quite standard.  Therefore, we will just
briefly outline the argument; cf.\ \cite{Lo,SZ}.

\begin{proof}
  To prove the lemma, it is sufficient to establish
  \eqref{eq:DC-it}. The rest of the assertion follows from that
  $\mu_-(\Phi)\geq m+2$ by the definition of dynamical convexity.

  Let us first assume that $\Phi$ is strongly non-degenerate and its
  elliptic eigenvalues are distinct.  Any such path $\Phi$ is
  homotopic, i.e., equal as an element of $\TSp(2m)$, to a product of
  a loop $\varphi$ and a path which is the direct sum of a path $\Psi$
  with hyperbolic end-point $\Psi(1)$ and some number $q\leq m$ of
  exponential paths in $\Sp(2)$ of the form
  $\Phi_i(t)=\exp\big(2\pi\sqrt{-1}\lambda_i t\big)$ with
  $0<\lambda_i<1$. (Using dynamical convexity of $\Phi$ we can ensure
  that $\hmu(\varphi)\geq 2$ and $\hmu(\Psi)\geq 0$, but we do not
  need this fact.)  By \eqref{eq:sp2}, \eqref{eq:loop} and
  \eqref{eq:additivity}, 
$$
\mu(\Phi)=\hmu(\varphi)+\hmu(\Psi)+q,
$$
and therefore
$$
\Delta:=\hmu(\varphi)+\hmu(\Psi)\geq \mu(\Phi)-m.
$$
Finally, note that, by \eqref{eq:sp2}, the sequences $\mu(\Phi_i^k)$
are non-decreasing. Now we have
$$
\mu(\Phi^{k+1})-\mu(\Phi^{k})= \Delta+\sum_i\big(
\mu(\Phi^{k+1}_i)-\mu(\Phi^{k}_i)\big)\geq \Delta\geq \mu(\Phi)-m ,
$$
which proves \eqref{eq:DC-it} for $\Phi$.

It is not hard to see that with $\mu$ replaced by $\mu_-$ this
argument extends to the case where $\Phi$ is still non-degenerate but
its eigenvalues are not necessarily distinct and some iterations
$\Phi^k$ may be degenerate. (In essence, the reason is that the index
is determined by the behavior and type (with multiplicity) of the
eigenvalues of $\Phi(t)$ rather than the map itself.)

In the case where $\Phi(1)$ is totally degenerate, \eqref{eq:DC-it} is
an immediate consequence of Lemma \ref{lemma:iterate} together with
\eqref{eq:mu-del}. Finally, the general case follows from these two
cases by additivity (Lemma \ref{lemma:additivity}).
\end{proof}

Let now $(M^{2n-1},\xi)$ be a co-oriented contact manifold with
$c_1(\xi)|_{\pi_2(M)}=0$ and $\alpha$ be a contact form supporting
$\xi$. We denote by $\varphi_t$ its Reeb flow. For a contractible
closed Reeb orbit $x$, its \emph{linearized Poincar\'e return map}
$\Phi=d\varphi_t|_x\in \TSp(2m)$ with $m=n-1$ is defined in the
standard way. Namely, we fix a capping of $x$ (i.e., a map from $D^2$
with boundary $x$) and trivialization of $\xi$ along the capping. With
this trivialization, the linearized flow of $\varphi_t$ along $x$
becomes a path in $\Sp(2m)$ starting at $I$. This path is well-defined
as an element of $\TSp(2m)$ up to conjugation by a linear map from
$\Sp(2m)$ (no tilde!). By definition $\mu_\pm(x)$, $\hmu(x)$, etc. are
$\mu_\pm(\Phi)$, $\hmu(\Phi)$, etc. The condition that $c_1(\xi)$
vanishes on $\pi_2(M)$ guarantees that the indices are independent of
the capping. The resulting indices inherit all the properties of their
linear algebra counterparts from Section \ref{sec:index-basics}. For
instance, the mean index $\hmu$ is homogeneous under iterations:
$\hmu(x^k)=k\hmu(x)$.

In fact, a global trivialization of the complex determinant bundle
$\det\xi=\wedge_\C^{n-1}\xi$ is sufficient to define the indices; see,
e.g., \cite{Es}. Thus, when $c_1(\xi)=0$ in $\H^2(M;\Z)$, one can
instead use such a trivialization and, in this case, the orbits need not
be contractible. Again, the indices have all the expected properties
including the homogeneity of the mean index. (Note that the behavior
of the indices under iterations is tied up with the relation between
trivializations of $\xi$ along $x^k$. For arbitrary, unrelated
trivializations, the mean index would not be homogeneous when $x$ is
not contractible.)

A variant of this construction applies to a hypersurface $M$ in a
symplectic manifold $(W^{2n},\omega)$. In this case one can associate
the linearized Poincar\'e return map and the indices to an oriented
closed characteristic. This is done exactly as for contact manifolds
but now using the symplectic normal $TM/TM^\omega$ to the
characteristic foliation in place of $\xi$.

In either case, the notion of a dynamically convex Reeb orbit or
a closed characteristic is well defined.  

\begin{Definition}
  \label{def:DC2} 
  The Reeb flow on a $(2n-1)$-dimensional contact manifold is said to
  be \emph{dynamically convex} (DC) if every closed Reeb orbit (or
  equivalently every simple closed Reeb orbit) is dynamically convex,
  i.e., $\mu_-(x)\geq n+1$ for all Reeb orbits $x$.
\end{Definition}

Here, again, if needed one can fix a collection of free homotopy
classes of closed Reeb orbits, which is closed under iterations. The
definition extends to hypersurfaces in symplectic manifolds in an
obvious way. The reader should keep in mind that this definition
``makes sense'' only for a rather narrow class of contact manifolds
such as those strongly fillable by displaceable Liouville domains. For
instance, it readily follows from Proposition \ref{prop:D-STS} that
the standard contact $ST^*S^n$ admits no dynamically convex contact
forms in the sense of Definition \ref{def:DC2}. Furthermore, the
``right'' condition in this case is $\mu_-\geq n-1$; see Section
\ref{sec:T^*S^n-2}. (In fact, the authors are not aware of any
examples of dynamically convex Reeb flows on contact manifolds other
than the standard contact $S^{2n-1}$.) A much more general notion of
dynamical convexity for pre-quantization circle bundles (e.g., for
$ST^*S^n$) is introduced and studied in \cite {AM:elliptic,AM}.

\begin{Theorem}[\cite{HWZ:convex}]
\label{thm:SDC}
The Reeb flow on a strictly convex hypersurface in $\R^{2n}$ is
dynamically convex.
\end{Theorem}

\begin{Remark}
  There is a typo in a remark on pp.\ 222--223 in \cite{HWZ:convex}
  concerning the higher-dimensional case of the theorem. As in Lemma
  \ref{lemma:DC}, the index lower bound at the end of the remark
  should be $2k+n-1$, not $nk+1$.
\end{Remark}

For the sake of completeness we give a simple proof of Theorem
\ref{thm:SDC}.

\begin{proof} Let $H\colon \R^{2n}\to \R$ be a convex Hamiltonian,
  homogeneous of degree two and such that $M=\{H=1\}$. Then the
  restriction of the Hamiltonian flow $\varphi_H^t$ of $H$ to $M$ is
  the Reeb flow on $M$. In particular, there is a one-to-one
  correspondence between $\PP(\alpha)$ and the periodic
  orbits of $H$ on $M$. Let $x$ be one of such orbits. Without loss of
  generality we may assume that $x$ is one-periodic. The linearized
  flow $d\varphi_H^t|_x$ along $x$ is a path in $\Sp(2n)$ starting at
  $I$ and generated by the positive definite time-dependent
  Hamiltonian $d^2H|_{T_{x(t)}\R^{2n}}$. (Here we have identified
  $T_{x(t)}\R^{2n}$ with $\R^{2n}$ itself.) Hence, by Example
  \ref{ex:pos-def},
$$
\mu_-\big( d\varphi_H^t|_x\big)\geq n.
$$
Let us now fix a trivialization of the contact structure $\xi$ on $M$
along a capping of $x$. This trivialization can be extended to a
trivialization of $T\R^{2n}$ along the capping by adding to it the
frame $\{X_H,JX_H\}$. The path $d\varphi_H^t|_x$ in the new
trivialization decomposes into the direct sum $\Phi\oplus I_2$, where
$\Phi\in\TSp(2m)$ with $m=n-1$ is the linearized Poincar\'e return map
of $x$. The standard trivialization of $\R^{2n}$ and the new one are
homotopic along the capping and hence along $x$. By the additivity of
$\mu_-$ (Lemma \ref{lemma:additivity}), we have
$$
n\leq \mu_-(\Phi\oplus I_2)=\mu_-(\Phi)-1.
$$
In other words, $\mu_-(\Phi)\geq n+1=m+2$.
\end{proof}

\begin{Remark}
  Note that in this proof we could have as well used the
  Robin--Salamon index or the Conley--Zehnder index, after passing to
  a small perturbation to obtain a lower bound on $\mu_-$. A minor
  modification of the argument also shows that the sequence
  $\mu_-(\Phi^k)$ is increasing and that $\mu_-(\Phi^k)\geq
  2k+m$. However, as we have already seen in Lemma \ref{lemma:DC}, these
  facts are formal consequences of dynamical convexity and even a
  stronger result, \eqref{eq:DC-it}, holds.
\end{Remark}

\begin{Remark}[Other consequences of convexity] 
  In addition to dynamical convexity, convexity of a hypersurface
  $M\subset \R^{2n}$ imposes other restrictions on the behavior of the
  indices of closed Reeb orbits $x$, although the exact scope of these
  restrictions is unclear to us. In light of the proof of Theorem
  \ref{thm:SDC}, as a preliminary step one can examine the geometry of
  positive paths $\Phi$ in $\Sp(2n)$, i.e., paths generated by
  positive definite Hamiltonians. This question must have been
  extensively studied, but we could not pin-point exact references;
  see however \cite{Ek,Lo}. Here we would like to mention only some
  simple facts that go beyond Example \ref{ex:pos-def}:
\begin{itemize}
\item For every point $A\in \Sp(2n)$, there is a positive path, with
  possibly very high mean index, from $I$ to $A$.
\item Assume that $\Phi(1)$ is totally degenerate. Then
  $\hmu(\Phi)\geq 2(b_0+\nu_0)+b_-+b_+$.
\end{itemize}
The first observation is obvious when $A$ is close to $I$ and the
general case follows from the fact that the product of positive paths
is again positive. The second observation can be proved using
\eqref{eq:RS}. (It would be interesting and useful to understand how
close to being sharp this inequality is.)

Furthermore, it is not hard to see that the Reeb flow on a convex
hypersurface is index-positive in a very strong sense. Namely, the
mean index $\hmu$ of an orbit grows with the length $l$ of the orbit
and more specifically $\hmu\geq a l-b$ for some $a>0$ and $b$
independent of the orbit. In fact, once a suitable extra structure is
fixed, this is true for all, not necessarily closed, orbits; see,
e.g., \cite{Es}. (This lower bound is an immediate consequence of a
suitable version of the Sturm comparison theorem applied to the
linearized flow $\Phi$; see, e.g., \cite[Sect.\ 2.3]{GG:wm}.) Yet, to
the best of the authors' knowledge, there are no known examples of
dynamically convex hypersurfaces in $\R^{2n}$ which are not
symplectomorphic to genuinely convex hypersurfaces.
\end{Remark}

\section{Index recurrence}
\label{sec:IR}

\subsection{Index recurrence theorem and its consequences}
\label{sec:IRT}
Now we are in a positition to state and prove our main combinatorial
results concerning the behavior of the index under iterations. These
results have non-trivial overlap with the treatment of the question in
\cite{LZ} and \cite{DLW,Lo}, although our argument is self-contained
and its logical structure is quite different.

To set the stage for the general case, let us first state a simpler
version of the theorem which requires the paths to be strongly
non-degenerate.  This result is essentially contained, although in a
different form, in \cite{DLW,LZ} as a part of the ``common jump
theorem''.
 
\begin{Theorem}[Index recurrence theorem, the non-degenerate case]
\label{thm:IRT1}
Consider a finite collection of strongly non-degenerate elements
$\Phi_1,\ldots,\Phi_r$ in $\TSp(2m)$. Then for any $\eta>0$ and any
$\ell_0\in\N$, there exists an integer sequence $d_j\to\infty$ and $r$
integer sequences $k_{ij}$, $i=1,\ldots, r$, at least one of which
goes to infinity, such that for all $i$ and $j$, and all $\ell\in\Z$
in the range $1\leq |\ell|\leq \ell_0$, we have
\begin{itemize}
\item[\rm{(i)}] $\big|\hmu(\Phi^{k_{ij}}_i)-d_j\big|<\eta$, and
\item[\rm{(ii)}] $\mu(\Phi^{k_{ij}+\ell}_i)= d_j + \mu(\Phi^\ell_i)$.
\end{itemize}
Furthermore, when all mean indices $\Delta_i:=\hmu(\Phi_i)$ are
non-zero we can ensure that $k_{ij}\to\pm\infty$ as $j\to \infty$ for
all $i$, and that $k_{ij}\to\infty$ when, in addition, all $\Delta_i$
have the same sign.  Moreover, for any $N\in \N$ we can make all $d_j$
and $k_{ij}$ divisible by~$N$.
\end{Theorem}

To illuminate this result, let us first consider the case where $r=1$,
i.e., the case of one strongly non-degenerate path $\Phi$. Fix any
``interval'' $L=[-\ell_0,\ell_0]\cap\Z$ and denote by $\mathring{L}$
the punctured interval $L\setminus\{0\}$. Then the theorem asserts, in
particular, that up to a sequence of common shifts $d_j$, the restricted
pattern $\mu|_{\mathring{L}}$ repeats itself for infinitely many
shifted copies of $\mathring{L}$. (Hence, the name of the theorem.) In
other words, there exists an infinite sequence of shifts $k_j$ in the
argument direction and a sequence of shifts $d_j$ in the $\mu$
direction such that
$$
\mu|_{\mathring{L}+k_j}=d_j+\mu|_{\mathring{L}}.
$$
This is also true for $r$ functions, where we now have $r$ sequences
$k_{ij}$ of shifts in the argument direction but still only one shift
sequence $d_j$ in the $\mu$ direction.

Without non-degeneracy, the theorem still holds for the upper and
lower indices $\mu_\pm$ when the interval $L$ lies entirely in the
positive domain, i.e., for $L=\{1,\ldots,\ell_0\}$. However, for
negative values of $\ell$ the assertion is no longer literally true.
The result holds only up to a correction term of the form
$b_+(\Phi^{|\ell|})-b_-(\Phi^{|\ell|})$, where $b_\pm$ are the
signature multiplicities defined in Section \ref{sec:tot-deg}; see
Definition \ref{def:tg}.
 
\begin{Theorem}[Index recurrence theorem, the degenerate case]
\label{thm:IRT2}
Let $\Phi_1,\ldots,\Phi_r$ be a finite collection of elements in
$\TSp(2m)$. Then for any $\eta>0$ and any $\ell_0\in\N$, there exists
an integer sequence $d_j\to\infty$ and $r$ integer sequences $k_{ij}$,
$i=1,\ldots, r$, at least one of which goes to infinity, such that for
all $i$ and $j$, and all $\ell\in\N$ in the range
$1\leq \ell\leq \ell_0$, we have
\begin{itemize}
\item[\rm{(i)}] $\big|\hmu(\Phi^{k_{ij}}_i)-d_j\big|<\eta$,
\item[\rm{(ii)}] $\mu_\pm(\Phi^{k_{ij}+\ell}_i)= d_j + \mu_\pm(\Phi^\ell_i)$,
\item[\rm{(iii)}]
  $\mu_\pm(\Phi^{k_{ij}-\ell}_i)= d_j + \mu_\pm(\Phi^{-\ell}_i)\pm
  \big(b_+(\Phi^{\ell}_i)-b_-(\Phi^{\ell}_i)\big)$.
\end{itemize}
Furthermore, when all mean indices $\Delta_i:=\hmu(\Phi_i)$ are
non-zero we can ensure that $k_{ij}\to\pm\infty$ as $j\to \infty$ for
all $i$, and that $k_{ij}\to\infty$ when, in addition, all $\Delta_i$
have the same sign.  Moreover, for any $N\in \N$ we can make all $d_j$
and $k_{ij}$ divisible by~$N$.
\end{Theorem}

Clearly, this theorem reduces to Theorem \ref{thm:IRT1} when all paths
$\Phi_i$ are strongly non-degenerate, for in this case $b_\pm=0$. (In
fact, the proof of Theorem \ref{thm:IRT1} is contained in the proof of
Theorem \ref{thm:IRT2} as a ``subset''.)  Note also that due to
\eqref{eq:inv} we can replace the term $\mu_\pm(\Phi^{-\ell}_i)$ in
(iii) by $-\mu_\mp(\Phi^\ell_i)$. In particular, since the correction
term in (iii) obviously does not exceed $\nu(\Phi^{\ell}_i)$, we have
the following result.

\begin{Corollary}
\label{cor:jump}
In the setting of Theorem \ref{thm:IRT2},
$$
\mu_-(\Phi^{k_{ij}+\ell}_i)= d_j + \mu_-(\Phi^\ell_i),
$$
and
\begin{equation}
\label{eq:mu+}
\mu_+(\Phi^{k_{ij}-\ell}_i)= d_j -
  \mu_-(\Phi^{\ell}_i)+\big(b_+(\Phi^{\ell}_i)-b_-(\Phi^{\ell}_i)\big)
\leq d_j -
  \mu_-(\Phi^{\ell}_i)+\nu(\Phi^{\ell}_i).
\end{equation}
\end{Corollary}
Combining these inequalities with Lemma \ref{lemma:DC}, we obtain 

\begin{Corollary}
\label{cor:CGT}
Assume in the setting of Theorem \ref{thm:IRT2} that the paths
$\Phi_1,\ldots,\Phi_r$ are dynamically convex. Then $d_j \to\infty$
and $k_{ij}\to \infty$ as $j\to\infty$ for all $i$, and
\begin{itemize}
\item $\mu_-(\Phi^{k_{ij}+\ell}_i)\geq d_j + 2\ell+m$ for
  $1\leq \ell\leq \ell_0$,
\item
  $\mu_+(\Phi^{k_{ij}-\ell}_i)\leq d_j - m-2\ell
  +\nu(\Phi^\ell_i)\leq d_j-2\ell$ for $1\leq \ell\leq \ell_0$.
\end{itemize}
In particular, $\mu_-(\Phi^{k_{ij}+\ell}_i)\geq d_j+2+m$
and $\mu_+(\Phi^{k_{ij}-\ell}_i)\leq d_j-2$ for all $\ell\in \N$ (or 
$\mu_+(\Phi^{k_{ij}-\ell}_i)\leq d_j-2-m$ in the strongly
non-degenerate case).
\end{Corollary}

This corollary is a variant of the common jump theorem,
\cite{DLW,LZ}. In essence, the corollary asserts that there exists a
sequence of sufficiently long intervals $L_j\subset \N$, containing
$d_j$, such that the intervals
$[\mu_-(\Phi^{k}_i),\, \mu_+(\Phi^{k}_i)]$ can possibly overlap with
$L_j$ only for $k=k_{ij}$. In other words,
$[\mu_-(\Phi^{k}_i),\, \mu_+(\Phi^{k}_i)]\cap L_j=\emptyset$ when
$k\neq k_{ij}$. More specifically, we have $L_j=[d_j-1,\,d_j+m+1]$ in
general, and $L_j=[d_j-m-1,\,d_j+m+1]$ when all $\Phi_i$ are strongly
non-degenerate. Thus the length of $L_j$ is $m+2$ in the former case
and $2(m+1)$ in the latter.

We emphasize that none of these results give any new information about
the index of $\Phi^{k_{ij}}_i$. However, since the difference between
$d_j$ and $\hmu(\Phi^{k_{ij}}_i)$ does not exceed $\eta$, we can
conclude that $\mu_\pm(\Phi^{k_{ij}}_i)$ is in the range
$[d_j-m,\,d_j+m]$ once $\eta<1/2$.

\subsection{Proof of the index recurrence theorem}
\label{sec:IRT-pf}
We start the proof of Theorem \ref{thm:IRT2} by focusing on the case
of a single path $\Phi$, i.e., $r=1$, and then show how to modify the
argument for a finite collection of paths. Below, without loss of
generality, we can assume that all paths are parametrized
by $[0,\,1]$.

\subsubsection{The case of $r=1$.}
\label{sec:IRT-pf-r=1} 
Let $\Phi=\Phi_1\in\TSp(2m)$. Throughout the argument we suppress $i$
in the notation, i.e., we will write $k_j$ for $k_{1j}$, etc. To
establish the theorem in this setting, we will consider several
subcases depending on the end-map $\Phi(1)$. Then the general case
will be established by additivity. Fix $\eta>0$ and
$\ell_0\in\N$. Without loss of generality we can assume that
$\eta<1/2$.

\smallskip\noindent\emph{Subcase A: $\Phi(1)$ is hyperbolic.} Set
$d_k=\hmu(\Phi^k)$. (Here, and also in Subcases B and D below, it is
more convenient to index $d$ by $k$ rather than $j$.) Clearly, (i) is
automatically satisfied. Furthermore, $\Phi^k$ is non-degenerate for
all $k\in \N$ and $\hmu(\Phi^k)=\mu(\Phi^k)$. Hence, we have
$$
\mu(\Phi^{k+\ell})=\hmu(\Phi^{k})+\hmu(\Phi^{\ell})
=\mu(\Phi^{k})+\mu(\Phi^{\ell}).
$$
Thus (ii) and (iii) hold for all $k$, i.e., with $k_j=j$, and all
$\ell$. To ensure that $N\mid d_k$, it suffices to take $k$ divisible
by $N$.

\smallskip\noindent\emph{Subcase B: $\Phi(1)$ is totally degenerate.}
By Lemma \ref{lemma:iterate}, $\Phi$ as an element of $\TSp(2m)$
is the product of a loop $\varphi$ and a path $\Psi$ such that all
eigenvalues of $\Psi(t)$ for all $t\in [0,\,1]$ are equal to
one. Multiplication by $\varphi^k$ shifts $\mu_\pm(\Psi^k)$ by
$k\hmu(\varphi)$, i.e.,
$$
\mu_\pm(\Phi^k)=k\hmu(\varphi)+\mu_\pm(\Psi^k).
$$
Furthermore, by Lemma \ref{lemma:iterate} and \eqref{eq:inv}, for any
$k>0$ we have
$$
\mu_\pm  (\Psi^k)=\mu_\pm (\Psi)\quad\textrm{and}\quad
\mu_\pm(\Psi^{-1})=\mp\mu_\mp(\Psi).
$$

Set again $d_k=\hmu(\Phi^k)=\hmu(\varphi^k)$. Then, for any
$\ell\in\N$,
\begin{align*}
\mu_\pm(\Phi^{k+\ell}) &=d_k+ \hmu(\varphi^\ell)+\mu_\pm(\Psi^{k+\ell})\\
&=d_k+ \hmu(\varphi^\ell)+\mu_\pm(\Psi^{\ell})\\
&=d_k+ \mu_\pm(\Phi^\ell).
\end{align*}
This proves (ii) for all $k\in\N$ and all $\ell\in\N$.

In the notation of Definition \ref{def:tg} and again by Lemma
\ref{lemma:iterate},
$$
\mu_+(\Psi^\ell)=b_0+b_+ +\nu_0 =-\mu_-(\Psi)+b_+-b_-
$$
where for brevity we set $b_*:=b_*(\Psi)=b_*(\Psi^\ell)$.

Likewise, when $0<\ell<k$, again by using \eqref{eq:inv} and Lemma
\ref{lemma:iterate}, we see that
\begin{align*}
\mu_+(\Phi^{k-\ell}) &= d_k- \hmu(\varphi^\ell)+\mu_+(\Psi^{k-\ell})\\
                    &=d_k- \hmu(\varphi^\ell)+\mu_+(\Psi^\ell)\\
&=d_k- \hmu(\varphi^\ell)-\mu_-(\Psi^\ell)+b_+ - b_-\\
&=d_k- \mu_-(\Phi^{\ell})+b_+ -b_-\\
&=d_k+ \mu_+(\Phi^{-\ell})+b_+ -b_-.
\end{align*}
This proves (iii) for $\mu_+$. The case of $\mu_-$ is handled
similarly.  In other words, (iii) holds for all $k\in \N$ and
all $\ell$ such that $0<|\ell|<k$.

\smallskip\noindent\emph{Subcase C: $\Phi(1)$ is elliptic and strongly
  non-degenerate.}  This is the first case in the argument which does
not hold for all $k\in\N$ and all with $0<|\ell|<k$, and we need to
take a subsequence $k_j$ and limit the range of $\ell$. Let
$\exp\big(\pm 2\pi\sqrt{-1}\lambda_q\big)$, $q=1,\ldots,m$, be the
eigenvalues of $\Phi(1)\in\Sp(2m)$, where $|\lambda_q|<1$, and at the
moment the choice of the sign of $\lambda_q$ can be arbitrary. Since
$\Phi(1)$ is strongly non-degenerate, all $\lambda_q$ are
irrational. Set
\begin{equation}
\label{eq:eps1}
\eps_0=\min_{0<\ell\leq\ell_0}\min_q\|\lambda_q\ell\|>0,
\end{equation}
where $\|\cdot\|$ stands for the distance to the nearest integer. Let
$\eps>0$ be so small that
$$
\eps\leq\eps_0\quad\textrm{and}\quad m\eps<\eta.
$$
It is easy to see that there exists a sequence $k_j\to\infty$ such
that for all $q$ we have
\begin{equation}
\label{eq:eps2}
\|\lambda_q k_j\|<\eps\leq \eps_0.
\end{equation}
Indeed, consider the semi-orbit
$\Gamma=\{k\vec{\lambda}\mid k\in \N\}\subset \T^{m}$ where
$\vec{\lambda}\in\T^{m}$ is the collection of eigenvalues of
$\Phi(1)$. As is well known, the closure of $\Gamma$ is a subgroup of
$\T^{m}$. Hence, $\Gamma$ contains points arbitrarily close to the
unit in $\T^{m}$ and, in particular, there exist infinitely many
points $k_j\vec{\lambda}$ in the $2\pi\eps$-neighborhood of the
unit. Then, by passing to a subsequence we can ensure that
$k_j\to\pm\infty$ and that, in fact, $k_j\to\infty$ by changing if
necessary the sign of all $k_j$. 

Let $d_j$ be the nearest integer to $\hmu(\Phi^{k_j})$. Then
$$
\big|d_j-\hmu(\Phi^{k_j})\big|\leq m\eps<\eta,
$$
and hence (i) is satisfied. This also shows that $d_j$ is
unambiguously defined. Clearly, for any $N\in\N$ we can also make all
$k_j$ and $d_j$ divisible by $N$. (To see this, it suffices to replace
the semi-orbit $\Gamma$ by $\{kN\vec{\lambda}\mid k\in \N\}$.)

To prove (ii) and (iii), observe first that a small perturbation of
$\Phi$ does not effect individual terms in these inequalities for
fixed $k_j$ and $\ell$. Thus, by altering $\Phi$ slightly, we can
ensure that all eigenvalues $\lambda_q$ are distinct. Then we can
write $\Phi$, up to homotopy, as the product of a loop $\varphi$ and
the direct sum of paths $\Psi_q=\exp(2\pi\sqrt{-1}\lambda_q t)\in\TSp(2)$ for a
suitable choice of signs of $\lambda_q$; see, e.g., \cite[Sect.\
3]{SZ}. The loop $\varphi$ contributes $k\hmu(\varphi)$ to
$\mu(\Phi^k)$ and hence we only need to prove (ii) and (iii) when
$\varphi=I$.

Then, for any $k$,
$$
\mu(\Phi^k)=\sum_q\mu(\Psi_q^k).
$$
Next, observe that by \eqref{eq:eps1} and \eqref{eq:eps2} we have
$$
d_j=\sum_q [\hmu(\Psi_q^k)],
$$
where $[\,\cdot\,]$ denotes the nearest integer. Thus it suffices to
prove (ii) and (iii) for each path $\Psi_q$ individually when we set
$d_j=[\hmu(\Psi_q^k)]$.  However, with \eqref{eq:eps1} and
\eqref{eq:eps2} in mind, (ii) and (iii) for $\Psi_q$ easily follows
from, e.g., \eqref{eq:sp2}.

\smallskip\noindent\emph{Subcase D: $\Phi(1)$ is non-degenerate, but
  $\Phi(1)^N=I$ for some $N\in\N$.} This subcase is a combination of
Subcases B and C.

Let us first assume that all eigenvalues of $\Phi(1)$ are equal to
each other, up to complex conjugation, and thus equal to
$\exp\big(\pm 2\pi\sqrt{-1}\lambda\big)$ where $\lambda$ is a root of
unity of degree $N$. We claim that (i)--(iii) hold for all $k$
divisible by $N$ and all $\ell$ with $d_k=\hmu(\Phi^k)$.

There are two cases to consider depending on whether $\ell$ is
divisible by $N$ or not.

Assume first that $N \centernot\mid \ell$. Then $\Phi^{k+\ell}$ is
non-degenerate. All eigenvalues of $\Phi^k(1)$ are equal to one since
$N \mid k$, and we can connect $\Phi^k(1)$ to $I$ by a path
$\Lambda(s)$, starting at $\Phi^k(1)$ at $s=0$ and ending at $I$ at
$s=1$, such that all eigenvalues of $\Lambda(s)$ are equal to one for
all $s$; cf.\ the proof of Lemma \ref{lemma:iterate}. Consider now the
following deformation $Z_s$ of the path $\Phi^{k+\ell}$. Namely,
$Z(s)$ is the concatenation of two paths. The first one, ending at
$\Lambda(s)$, is itself the concatenation of $\Phi^k$ and the path
$\Lambda(\tau)$ for $\tau \in [0,\,s]$. The second one, starting at
$\Lambda(s)$, is the path $\Phi^\ell\Lambda(s)$. The end-point
$\Phi^\ell(1)\Lambda(s)$ of the path $Z_s$ is non-degenerate for all
$s$ and hence $\mu(Z_s)$ remains constant. By construction,
$Z_0=\Phi^{k+1}$ and $Z_1$ is the concatenation of a loop with the
same mean index $d_k$ as $\Phi^k$ and the path $\Phi^\ell$. We
conclude that
$$
\mu(\Phi^{k+\ell})=d_k+\mu(\Phi^\ell).
$$
This implies (ii). Clearly, this argument works not only for positive
$\ell$ but for any $\ell\in\Z$ not divisible by $N$. Recalling that
$\mu(\Phi^{-1})=-\mu(\Phi)$, we obtain (iii).

The remaining case is when $N\mid \ell$. The end-point of the path
$\Phi^N$ is totally degenerate. Hence, we can apply the argument from
Subcase B to $\Phi^N$ in place of $\Phi$ with $k$ and $\ell$ replaced
by $k'=k/N$ and $\ell'=\ell/N$. Then we have
$$
\mu_\pm(\Phi^{k+ \ell})=d_k+\mu_\pm(\Phi^\ell) 
$$
and 
$$
\mu_\pm(\Phi^{k- \ell})=d_k+\mu_\pm(\Phi^{-\ell}) +
\big(b_+(\Phi^\ell)-b_-(\Phi^\ell)\big),
$$
which proves (ii) and (iii) in this case.

In general, we can decompose $\Phi$, up to homotopy, into a direct sum
of paths $\Phi_q$, where $\Phi_q(1)$ has only one eigenvalue, up to
complex conjugation, and this eigenvalue is a root of unity of degree
$N_q$. Applying the above argument to each $\Phi_q$ individually, when
$k$ is divisible by $N=\lcm\{N_q\}$ or any other $N$ with $\Phi^N=I$,
we see that (ii) and (iii) hold for $\Phi$ for all $k$ divisible by
$N$ and all $\ell$.

\smallskip\noindent\emph{Putting Subcases A--D together.} Let us
decompose $\Phi$ into the direct sum of four paths
$\Phi_A,\ldots,\Phi_D$ with each path as in one of Subcases A--D. Let
$N$ be such that $\Phi_D^N=I$. We can chose $k_j$ divisible by $N$ so
that (i), (ii) and (iii) are satisfied for $\Phi_C$ with
$d_j=[\hmu(\Phi^{k_j})]$. Furthermore, in all cases but Subcase C, we
have $\hmu(\Phi^k_{A,B,D})=d_k$. Thus it is clear that (i) holds for
$\Phi$ for the sequence $d_j=d_{k_j}$ which is the sum of such
sequences for all four subcases. Likewise, since (ii) and (iii) hold
in Subcases A, B, and C for all $k$ divisible by $N$, we conclude that
(ii) and (iii) are satisfied for $\Phi$ for the sequence $k_j$. In
addition, we can make $k_j$ divisible by any other integer.

\subsubsection{The general case: $r\geq 1$.}
\label{sec:CGT-pf-gen} 
Let $\Phi_1,\dotsc,\Phi_r$ be a finite collection of elements in
$\TSp(2m)$. If we apply the argument from Section \ref{sec:IRT-pf-r=1}
to each $\Phi_i$ individually, we obtain $r$ integer sequences
$k_{ij}$ and $r$ integer sequences $d_{ij}$ such that (i)--(iii)
hold. Thus our goal is to show that $k_{ij}$ can be chosen so that
(i)--(iii) hold for the same sequence $d_j$.

Denote by $\exp\big(\pm 2\pi\sqrt{-1}\lambda_{iq}\big)$ the elliptic
eigenvalues of $\Phi_i$ with irrational $\lambda_{iq}$ and set
$\Delta_i=\hmu(\Phi_i)$. (The choice of the sign of $\lambda_{iq}$ is
immaterial at the moment.)  Given $\eps>0$, consider the system of
inequalities
\begin{equation}
\label{eq:k}
\begin{aligned}
  \|k_i\lambda_{iq}\|& <\eps \quad\textrm{ for all $i$ and $q$,}\\
  |k_1\Delta_1-k_{i}\Delta_{i}| &< \frac{1}{8} \quad\textrm{ for
    $i=2,\dotsc,r$,}
\end{aligned}
\end{equation}
where we treat the integer vector $\vec{k}=(k_1,\dotsc, k_r)\in\Z^r$
as a variable. Introducing additional integer variables $c_{iq}$, we
can rewrite the first group of inequalities in the form
$$
|k_i\lambda_{iq}-c_{iq}|<\eps.
$$
With this in mind, system \eqref{eq:k} has one fewer equation than the
number of variables. By Minkowski's theorem (see, e.g., \cite{Ca}),
there exists an infinite sequence of distinct solutions
$\vec{k}_j= (k_{1j},\dotsc, k_{rj})$ of \eqref{eq:k}.

Now, by passing to a subsequence and changing if necessary the signs
of $k_{ij}$, we can ensure that at least one of the sequences $k_{ij}$
goes to $\infty$ as $j\to\infty$. Then the second group of
inequalities implies that $k_{ij}\to\pm\infty$ for all $i$ when all
mean indices $\Delta_i\neq 0$ and also that $k_{ij}\to\infty$ when all
$\Delta_i$ have the same sign.

Moreover, we can make all $k_{ij}$ divisible by any
fixed integer $N$.  In particular, let $N$ be the least common
multiple of the degrees of roots of unity among the eigenvalues of all
$\Phi_i(1)$. We take the sequences $k_{ij}$ divisible by $N$ and by
any other integer as required in the statement of the theorem.

Finally, fix $\ell_0$ and $\eta>0$ which we assume to be sufficiently
small (e.g., $\eta<1/4$).  Similarly to Subcase C, set
$$
\eps_0=\min_{0<\ell\leq\ell_0}\min_{i,q}\|\lambda_{iq}\ell\|>0,
$$
and let $\eps>0$ be so small that again
$$
\eps\leq\eps_0\quad\textrm{and}\quad m\eps<\eta.
$$

By the second series of inequalities in \eqref{eq:k}, we have
$$
|k_i\Delta_i-k_{i'}\Delta_{i'}|<\frac{1}{4} \quad \textrm{for all $i$
  and $i'$,}
$$
and $\|k_{ij}\Delta_i\|<\eta$ by the first group of inequalities. Thus
$k_{ij}\Delta_i$ is $\eta$-close, for all $i$, to the same integer
$$
d_j=[k_{ij}\Delta_i].
$$
In other words, (i) is satisfied for this choice of $d_j$. Note that
for every $i$, decomposing $\Phi_i$ according to the Subcases A--D, we
have in the obvious notation
$$
d_j=[\hmu(\Phi_{i,C}^{k_{ij}})]+\hmu(\Phi_{i,A}^{k_{ij}})
+\hmu(\Phi_{i,B}^{k_{ij}})+\hmu(\Phi_{i,D}^{k_{ij}}).
$$
Furthermore, for every $i$, condition \eqref{eq:eps2} is met for
$\lambda_{iq}$ and, since all $k_{ij}$ are divisible by $N$, it
readily follows as in Section \ref{sec:IRT-pf-r=1} that (ii) and (iii)
hold for all $i$. This concludes the proof of Theorem
\ref{thm:IRT2}. \qed

\section{Multiplicity results and other applications}
\label{sec:pf-main}
In this section we combine the results from Lusternik--Schnirelmann
theory for the shift operator and the index theory to establish our
multiplicity results for simple closed Reeb orbits. These are Theorem
\ref{thm:mult-intro}, which is a direct consequence of Theorems
\ref{thm:mult2} and \ref{thm:RR}, and also Theorem
\ref{thm:mult4-intro} proved in Section~\ref{sec:T^*S^n-2}.  As
mentioned in the introduction, we focus on the standard contact
$S^{2n-1}$ or, more generally, the boundary of a displaceable
Liouville domain and $ST^*S^n$.

\subsection{Hypersurfaces in $\R^{2n}$ and displaceable Liouville domains}
We start with the simplest and arguably the most interesting situation
where $(M^{2n-1},\alpha)$ is a closed, restricted contact type,
dynamically convex hypersurface in $\R^{2n}$. For instance, $M$ can be
the boundary of a star-shaped domain, provided that the Reeb flow is
dynamically convex. However, even when $M$ is convex, some of our
results are new.

\subsubsection{Multiplicity results for hypersurfaces in $\R^{2n}$}
Let us call a simple orbit \emph{reoccurring} if its iterations occur
infinitely many times in the image of a carrier injection $\psi$ from
Corollary \ref{cor:sphere}. (This notion depends on the choice of
$\psi$.) One can show that a generic Reeb flow has no reoccurring
closed orbits, but the flows with only finitely many simple orbits
necessarily do. Our main multiplicity result is the following theorem.

\begin{Theorem}
\label{thm:mult2}
Let $(M^{2n-1},\alpha)$ be a closed contact type,
dynamically convex hypersurface in $\R^{2n}$ bounding a simply
connected Liouville domain. Then $M$ carries at least $r$ simple
closed characteristics $x_1,\dotsc, x_r$, where
$r=\lceil n/2\rceil +1$ in general and $r=n$ when $\alpha$ is
non-degenerate. Moreover, assume that $\PP(\alpha)$ is finite.  Then
the orbits $x_i$ can be chosen to be reoccurring and, if in addition
$\alpha$ is non-degenerate, so that all $x_i$ are even and
$\mu(x_i)\equiv n+1 \pmod 2$.
\end{Theorem}

\begin{Remark}
  \label{rmk:elliptic} 
  Similarly to the case of convex hypersurfaces considered in
  \cite{LZ}, at least one of the orbits $x_i$ is elliptic (two, in the
  non-degenerate case) when $M$ carries only finitely many simple
  periodic orbits. This can be easily seen from the proof of the
  theorem. (See also \cite{AM} for some relevant results.)
\end{Remark}

\begin {proof}
  Without loss of generality we may assume that the Reeb flow of
  $\alpha$ has only finitely many simple closed orbits, which, as in
  the theorem, we denote by $x_1,\ldots,x_{r}$. The set of closed Reeb
  orbits $\PP(\alpha)$ comprises all iterations $x_i^k$, $k\in\N$, of
  the orbits $x_i$.  Let us first show that $r= \lfloor n/2\rfloor+1$
  in general and $r= n$ when all closed orbits are non-degenerate.

Consider the map 
$$
\psi\colon \CI=\{n+1,\,n+3,\,n+5,\,\ldots\}\to \PP(\alpha),\quad
d\mapsto y_d
$$
from Corollary \ref{cor:sphere}, where we relabeled the domain of
$\psi$ by the index.  (In other words, this map is obtained by
composing the map in the corollary with the bijection
$d\mapsto (d+1-n)/2$ from $\CI$ to $\N$.) Thus the orbits which were
denoted in the corollary by $y_1,y_2,\ldots$ are now
$y_{n+1}, y_{n+3},\ldots $. We have
$$
\mu_-(y_d)\leq d \leq \mu_+(y_d).
$$

Let $\Phi_i\in \TSp(2m)$, $m=n-1$, be the linearized Poincar\'e return
map along $x_i$ (see Section \ref{sec:index:DC}); without loss of
generality we can assume that the paths $\Phi_i$ are parametrized by
$[0,\,1]$.  Fixing a small parameter $\eta>0$ and a sufficiently large
$\ell_0\in\N$, let us apply Corollary \ref{cor:CGT} to the paths
$\Phi_i$, where we require the iterations $k_{ij}$ to be even and
divisible by the degrees of the roots of unity among the eigenvalues
of $\Phi_i(1)$. Then, for all $\ell\in\N$, we have

\begin{itemize}
\item $\mu_-(x^{k_{ij}+\ell}_i)\geq d_j + n+1$, and
\item $\mu_+(x^{k_{ij}-\ell}_i)\leq d_j - 2$ and
  $\mu_+(x^{k_{ij}-\ell}_i)\leq d_j -n-1$ when all orbits $x_i$ are
  strongly non-degenerate.
\end{itemize}

Denote by $L$ the index interval $[d_j-n,\,d_j+n]\cap\CI$ in the
non-degenerate case and set $L=[d_j-1,\,d_j+n]\cap\CI$ in
general. Then for any $d\in L$ the orbit $y_d$ must have the form
$x_i^{k_{ij}}$, and therefore at most one iteration of $x_i$ can occur
as $y_d$ with $d\in L$. It follows that $r$, the number of simple
orbits $x_i$, is greater than or equal to the number of points in $L$,
i.e., $r\geq \#(L)$. In the non-degenerate case $\#(L)=n$ and in
general $\#(L)= \lfloor n/2\rfloor +1$. Furthermore, in the
non-degenerate case the orbits $x_i$ must be even since the iterations
$k_{ij}$ are even; see Example \ref{ex:local-nondeg-2}. We then
necessarily have $\mu(x_i)\equiv n+1 \pmod 2$.

Our next goal is to improve in the degenerate case this lower bound by
one when $n$ is odd which we will assume from now on. Then
$d_{\max}=d_j-2$ is the largest point in $\CI$ before the interval
$L$.  It readily follows from Corollary \ref{cor:CGT} that
$y_{d_{\max}}$ can only have the form $x_i^{k_{ij}}$ or
$x_i^{k_{ij}-1}$. In the former case, $x_i^{k_{ij}}$ does not
contribute to the interval $\CI$ and hence $r\geq \#(L)+1$.

We claim that in the latter case, i.e., when
$y_{d_{\max}}=x_i^{k_{ij}-1}$ and hence $\ell=1$, we have infinitely
many simple periodic orbits. To prove this, first note that

$$
\mu_+( x_i^{k_{ij}-1})\leq d_j-n-1+\nu(x_i)\leq d_{\max}=d_j-2
$$
since $\nu(x_i)\leq n-1$, and
$$
\nu(x_i)=n-1\textrm{ and }
\mu_+( x_i^{k_{ij}-1}) = d_j-2.
$$
As a consequence, $x_i$ is totally degenerate and
$d_j=\hmu (x_i^{k_{ij}})$. It follows that $x_i^{k_{ij}-1}$ is the
so-called symplectically degenerate maximum (SDM); see
\cite{GHHM}. Indeed, by \eqref{eq:bpm} and, since $M$ is dynamically
convex, we must have $\mu_-(x_i)=\hmu(x_i)=n+1$ and $b_+(x_i)=0$ and
$b_-(x_i)=b_0(x_i)=\nu_0(x_i)=0$. Furthermore, since by construction
$k_{ij}$ is divisible by the degrees of the roots of unity among the
eigenvalues of $x_i$, the iteration $k_{ij}-1$ is relatively prime
with these degrees and hence $k_{ij}-1$ is an admissible
iteration. Therefore, $x_i$ is also an SDM; see \cite[Prop.\
3]{GHHM}. (Thus the local Floer homology of $x_i$ is concentrated in
degree $\mu_+(x_i)=2n$ which is the upper end point of its support.)
Finally, as is shown in \cite{GHHM}, the presence of a simple SDM
orbit implies that the Reeb flow of $\alpha$ has infinitely many
simple periodic orbits.

To summarize, we have $r\geq \lfloor n/2\rfloor +1$ when $n$ is even
and $r\geq \lfloor n/2\rfloor +2$ when $n$ is odd. This is equivalent
to that $r\geq \lceil n/2\rceil +1=: r$. This completes the proof of
the first part of the theorem.

Next, observe that Corollary \ref{cor:CGT} provides an infinite
sequence of the intervals $L$, and hence, when $M$ carries only finitely
many simple periodic orbits, there exist $r$ reoccurring simple orbits
$x_1,\ldots,x_r$, with $r=\lceil n/2\rceil +1$ in general and $r=n$
when $\alpha$ is non-degenerate.
\end{proof}

\begin{Remark}
  When $n$ is even, the point $d_j-2$ is not in $\CI$. The largest
  point in $\CI$ before the interval $L$ is $d_{\max}=d_j-3$. Consider
  the orbit $y_{d_{\max}}$.  There are now two possible cases. One is
  that $\nu(y_{d_{\max}})=n-1$ and $\hmu(y_{d_{\max}})=n+1$. Then,
  exactly as in the proof, $y_{d_{\max}}$ is the $k_{ij}$th iteration
  of a simple SDM orbit $x_i$ and $(M,\alpha)$ carries infinitely many
  periodic orbits. However, there is a second possibility. This is
  that $\nu(y_{d_{\max}})=n-2=b_+(y_{d_{\max}})$ with $b_-=b_0=0$ and
  $\hmu(y_{d_{\max}})$ is either $n$ or $n+1$. (Thus the linearized
  Poincar\'e return map along $y_{d_{\max}}$ is the direct sum of a
  totally degenerate map in dimension $2(n-2)$ and an elliptic or
  negative hyperbolic map in dimension $2$.) It is not clear to us how
  to rule out such an orbit $y_{d_{\max}}$.
\end{Remark}

\subsubsection{Resonance relations} In a variety of settings the
actions and/or indices of closed Reeb orbits satisfy certain resonance
relations (in fact, more than one type), which have applications in
Reeb dynamics; see, e.g., \cite{EH:RR, GGo, GG:generic, GiKe, Gu:pr,
  LL, LLW, Ra:geod, Vi:RR}. This is also the case for the orbits
$x_1,\ldots,x_r$ from Theorem~\ref{thm:mult2}. Namely, set
$$
\hs(x)=\frac{\CA_\alpha(x)}{\hmu(x)},
$$
where $x\in\PP(\alpha)$. (If $\hmu(x)=0$, we set $\hs(x)=\infty$.)
This ratio, originally considered in \cite{Gu:pr}, is a contact analog
of the augmented action from \cite{CGG, GG:gaps}. It is clear that
$\hs(x)=\hs(x^k)$ for all $k$.

\begin{Theorem}[Resonance Relations]
\label{thm:RR}
Let $(M^{2n-1},\alpha)$ be a closed contact type hypersurface in
$\R^{2n}$ bounding a simply connected Liouville domain.  Assume that
the set $\{\hs(x)\}$, where $x$ ranges over all reoccurring orbits, is
discrete. Then, for any two reoccurring closed Reeb orbits $x$ and
$y$, we have $\hs(x)=\hs(y)$, i.e.,
$$
\frac{\CA_\alpha(x)}{\hmu(x)}=\frac{\CA_\alpha(y)}{\hmu(y)}.
$$
\end{Theorem}

\begin{Remark}
  In general, the carrier map $\psi$ from Corollary \ref{cor:sphere} is
  not unique. Theorem \ref{thm:RR} holds for any choice of $\psi$. One
  can show that the requirement on the reoccurring augmented action
  spectrum $\{\hs(x)\}$ is satisfied if, for instance, the ordinary
  action spectrum $\CS(\alpha)$ is discrete and $c_1(\xi)=0$. It is
  also met for quasi-finite hypersurfaces introduced below, e.g., when
  $M$ carries only finitely many simple periodic orbits.
\end{Remark}

Together, Theorems \ref{thm:mult2} and \ref{thm:RR} imply Theorem
\ref{thm:mult-intro}.  Without additional assumption on $(M,\alpha)$,
Theorem \ref{thm:RR} gives little information. For, hypothetically, it
is possible that the image of a carrier injection $\psi$ consists
entirely of the iterations of a single simple orbit. Furthermore, it
is also possible that there are no reoccurring orbits. (In fact, this
should be true $C^\infty$-generically.) In this vein, one consequence
of Theorem \ref{thm:RR} is the $C^\infty$-generic existence of infinitely
many simple periodic orbits on a restricted contact type hypersurface in
$\R^{2n}$; see, e.g., \cite{GG:generic} for an applicable
argument. The key example meeting the requirements of the theorem is
the standard contact sphere $S^{2n-1}$. In this case, however, the
$C^\infty$-generic existence of infinitely many periodic orbits is
well-known; see \cite{Vi:RR}.

Another application of Theorem \ref{thm:RR} concerns the behavior of
the ``normalized'' spectral invariants $\s_d(\alpha)/d$ in the
notation from Section \ref{sec:displ}.  Let us call a restricted
contact type hypersurface $M\subset \R^{2n}$ \emph{quasi-finite} if
there exists a finite collection of simple periodic orbits $x_i$ such
that the iterations of $x_i$ occur in the image of $\psi$ infinitely
many times and cover all but a finite part of the image. In other
words, all but a finite part of $\psi(\CI)$ lies in the union of the
sets $\{ x_i^k\mid k\in\N\}$ and each of the sets has infinite
intersection with the image. This is an extremely non-generic
condition. However, it is satisfied, for instance, when $M$ carries
only finitely many simple closed characteristics.  For a quasi-finite
$M$, let us set
$\hat{\s}(\alpha)=\CA_\alpha(x_i)/\hmu(x_i)=\hs(x_i)$. By Theorem
\ref{thm:RR}, this ratio is independent of $x_i$.

\begin{Corollary}
\label{cor:limit}
Assume that $(M,\alpha)$ is a quasi-finite contact type hypersurface
in $\R^{2n}$ bounding a simply connected Liouville domain. Then
$$
\lim_{d\to\infty}\frac{\s_d(\alpha)}{d}=\hat{\s}(\alpha).
$$
\end{Corollary}

\begin{proof}
  Consider an infinite sequence $x_i^{k_{ij}}=\psi(d_j)$ in the image
  of $\psi$. Then
$$
\frac{\s_{d_j}(\alpha)}{d_j}=
\CA_\alpha(x_i) \frac{ k_{ij}}{d_j}.
$$
By \eqref{eq:mu-del}, $k_{ij}/d_j\to \hmu(x_i)^{-1}$ and hence
$\s_{d_j}(\alpha)/d_j\to\hat{\s}(\alpha)$. The sequence
$\s_d(\alpha)/d$ is a finite ``union'', in the obvious sense, of the
sequences $\s_{d_j}(\alpha)/d_j$ converging to the same limit
$\hat{\s}(\alpha)$, and the result follows. (Note that this argument
breaks down if we omit the requirement that the collection $\{x_i\}$
is finite.)
\end{proof}

\begin{Example}[Ellipsoids]
  Let $M$ be the ellipsoid $\sum_i |z_i|^2/r_i^2=1$ in $\C^n=\R^{2n}$
  as in Example \ref{ex:ellipsoids}. Let us assume that the closed
  orbits are isolated and denote by $x_i$ the simple periodic orbit
  lying on the $z_i$-axis. As is shown in \cite[Example 1.2]{Ba},
$$
\hmu(x_i)=2r_i^2\sum_j r_j^{-2},
$$
and obviously $\CA_\alpha(x_i)=\pi r_i^2$. 
Therefore,
$$
\hat{\s}(\alpha)=\frac{\pi}{2\sum_j r_j^{-2}}.
$$ 
Curiously, it is not immediately obvious how to directly prove that
the sequence $\s_d(\alpha)/d$, explicitly written down in Example
\ref{ex:ellipsoids}, converges to $\hat{\s}(\alpha)$.
\end{Example}

\begin{Remark}
  It is clear that the sequence $\s_d(\alpha)/d$ lies in the interval
  $[\pi r^2,\,\pi R^2]$, where $r$ is the radius of a ball enclosed by
  $M$ and $R$ is the radius of a sphere enclosing $M$. It would be
  interesting to understand if or when this sequence converges and
  what the limit is in general.  When $M$ is convex, this question
  appears to be related to some of the results from \cite{EH:RR}. Also
  note that, by Corollary \ref{cor:limit},
  $\hat{\s}(W):=\hat{\s}(\alpha)$ is a monotone function of $W$ with
  respect to inclusions and hence a capacity, as long as $\p W$ is
  quasi-finite.
\end{Remark}

\begin {proof}[Proof of Theorem \ref{thm:RR}] It is sufficient to
  prove the theorem for simple periodic orbits.  We will focus
  exclusively on simple orbits $z$ with $\hmu(z)>0$, and hence with
  $\hs(z)>0$. Let us say that $d\in\CI$ is represented by an orbit $z$
  if $\psi(d)$ is an iteration of $z$. (Here, as in the proof of
  Theorem \ref{thm:mult2}, we prefer to index the domain of $\psi$ by
  $\CI=\{n+1,\,n+3,\, \ldots\}$.)  Whenever $\psi(d)=z^k$, we have
\begin{equation}
\label{eq:delta}
\CA_\alpha(z^k)\approx \hs(z) d
\end{equation}
up to an error not exceeding the constant $C(z)=(n-1)\hs(z)$ which is
independent of $d\in\CI$.  Indeed, since
$|d-\hmu(z^k)|=|d-k\hmu(z)|\leq n-1$ by Corollary \ref{cor:sphere}, we
have
$$
\hs(z)d = \CA_\alpha(z)\frac{d}{\hmu(z)}
                    =k\CA_\alpha(z)+e\hs(z),
$$
where $|e|\leq n-1$, and \eqref{eq:delta} follows.

To establish the theorem, it suffices to show that for any two simple
reoccurring orbits $x$ and $y$ we necessarily have $\hs(x)\geq \hs(y)$
and thus, by symmetry, $\hs(x)=\hs(y)$. We prove this by
contradiction.

First, let $x$ and $y$ be two simple, not necessarily reoccurring,
orbits with $\hs(x)<\hs(y)$. Then there exists a constant $d(x,y)$
such that $x$ cannot represent $d+2$ when $d\geq d(x,y)$ is
represented by $y$. Indeed, let $d(x,y)$ be the first integer in $\CI$
for which
$$
\hs(x)\big((d(x,y)+2)+(n-1)\big)< \hs(y)\big(d(x,y)-(n-1)\big).
$$
Then $d+2$ cannot be represented by $x$. For, if it were, the action
$\CA_\alpha(\psi(d+2))$ on the resulting iteration $x^k$ would be, by
\eqref{eq:delta}, strictly smaller than $\CA_\alpha(\psi(d))$. This is
impossible by Corollary \ref{cor:sphere}. Note that $d(x,y)$ is
completely determined by $\hs(x)/\hs(y)$, and $d(x,y)$ is a decreasing
function of this ratio.

Consider all simple orbits $x'$ with 
$$
\hs(x)\leq \hs(x')< \hs(y).
$$
By the conditions of the theorem, the set $\{\hs(x')\}$ is finite.
Assuming now that $y$ is reoccurring, we can find $d\geq \max d(x',y)$
represented by $y$. Then $d+2$ cannot be represented by any of the
orbits $x'$ including $x$. We denote by $y_1$ a simple orbit
representing $d+2$.

By construction, $\hs(y)\leq \hs(y_1)$; for otherwise $y_1$ would be
one of the orbits $x'$. Hence $d(x',y_1)\leq d(x',y)\leq d+2$.
Therefore, any of the orbits $x$ or $x'$ cannot represent $d+4$, which
is then represented by some simple orbit $y_2$ with
$\hs(y)\leq \hs(y_2)$. (But not necessarily $\hs(y_1)\leq \hs(y_2)$.)
Thus $d(x',y_2)\leq d(x',y)\leq d+4$, and $d+6$ is represented by some
orbit $y_3$ different from $x$ and $x'$, and so on. Arguing
inductively, we conclude that when $y$ is reoccurring $x$ cannot
represent any $d\geq d(x,y)$ and hence cannot be reoccurring.
\end{proof}

\subsubsection{Generalizations, refinements and failures}
\label{sec:generalizations}
The proof of Theorem \ref{thm:mult2} readily lends itself to several
generalizations and refinements which we will now discuss.

The first of these results, generalizing \cite[Thm.\ 1.4]{GK},
concerns the situation where the dynamical convexity lower bound
$\mu_-(x)\geq n+1$ is replaced by $\mu_-(x)\geq q+1$ for some
$q\geq 0$. Here, again, we have a lower bound $r$, depending on $q$,
on the number of simple periodic orbits. Some examples where the lower
bound $\mu_-\geq n-1$ arises naturally are considered in \cite{AM}.

\begin{Theorem}
\label{thm:mult3}
Let $(M^{2n-1},\alpha)$ be a closed contact type
hypersurface in $\R^{2n}$ bounding a simply connected Liouville
domain.  Assume that $\mu_-(y)\geq q$ with $0<q\leq n$ for all, not
necessarily simple, closed characteristics $y$ on $M$.  Then $M$
carries at least $r$ simple closed characteristics, where
$$
r=
\begin{cases}
q-\lceil n/2 \rceil &\textrm{ when $n$ and $q$ are odd,}\\
q+1-\lceil n/2 \rceil & \textrm{otherwise.}
\end{cases}
$$
(If the right hand side is negative or zero the result is void.) When
$\alpha$ is non-degenerate, we can take
$$
r=
\begin{cases}
q +1 & \textrm{ when $n$ and $q$ have the same parity,}\\
q & \textrm{ when $n$ and $q$ have opposite parity.}
\end{cases}
$$
\end{Theorem}

\begin{Remark} 
\label{rmk:q1}
The main new point of the theorem is the lower bound on $r$ in the
degenerate case and then it is sufficient to only assume that
$\mu_-(x)\geq q$ for the orbits with $\hmu(x)>0$. In the
non-degenerate case, a stronger result has recently been established
under different but conceptually less restrictive assumptions; see
\cite{DLLW}.  When $q=n-2$ and $M$ is non-degenerate we obtain
\cite[Thm.~1.4]{GK}.  Similarly to Theorem \ref{thm:mult2}, all $r$
orbits in the general case and $r-1$ orbits in the non-degenerate case
can be chosen reoccurring when $M$ carries only a finite number of
simple periodic orbits. Furthermore, in the non-degenerate case, $r-1$
orbits can be chosen even and with Conley--Zehnder index of the same
parity as $n+1$. The remaining orbit is either odd or has
Conley--Zehnder index of the same parity as $n$.
\end{Remark}

\begin{proof}
  The general case of the theorem is derived from Theorem
  \ref{thm:IRT2} exactly in the same way as the general case of
  Theorem \ref{thm:mult2} and we omit the argument; see also the proof
  of Theorem \ref{thm:mult5}. Note, however, that in the present
  setting we cannot conclude that the orbit $y_{d_{\max}}$ is an SDM
  and thus strengthen the result.

  In the non-degenerate case, reasoning as in the proof of Theorem
  \ref{thm:IRT2}, we can take $L$ to be the open interval
  $(d_j-q,\,d_j+q)$. Then $\#(L\cap\CI)=r-1$. Thus, for $q<n$, we have
  found $r-1$ simple even orbits $x_i$ such that $\mu(x_i)$ has the
  same parity as $n+1$.
  
  To find an extra orbit, we borrow an argument from the proof of
  \cite[Thm.\ 1.4]{GK}. Observe first that without loss of generality
  we can assume that the Reeb flow has an orbit, not necessarily
  simple, of index $q<n+1$. Let us denote this orbit by $y^k$, where
  $y$ is simple. We can further assume that $y^k$ is good; for
  otherwise $y$ is odd and hence different from the orbits $x_i$.  As
  a consequence, the parity of $\mu(y)$ is the same as $q$. It follows
  that if $q$ has parity different from $n+1$, the orbit $y$ is
  different from any of the orbits $x_i$.

  The remaining case is that of $q$, and hence $\mu(y)$, having the
  same parity as $n+1$. It is easy to see that, since
  $\SH_{q}^{G,+}(W)=0$, there must be a good orbit with
  Conley--Zehnder index $q+1$. Let us denote this orbit $z^l$, where
  $z$ is simple. Now the parity of $\mu(z)$ is the same as that of
  $q+1$, and therefore different from $n+1$. Thus $z$ is different
  from any of the orbits $x_i$.
 \end{proof}

\begin{Remark}
\label{rmk:displ-mult}
Ultimately, on the side of spectral invariants, the proofs of Theorems
\ref{thm:mult2}, \ref{thm:RR} and \ref{thm:mult3} and Corollary
\ref{cor:limit} depend only on Corollary \ref{cor:sphere}.  As was
mentioned in Section \ref{sec:displ}, these corollaries carry over
word-for-word to any simply connected Liouville domain $W$
displaceable in $\widehat{W}$ or even to $W$ displaceable in some
other Liouville manifold, provided that $c_1(TW)=0$. Hence, for such
$W$, the theorems and Corollary \ref{cor:limit} also hold as stated.
\end{Remark}

We finish this section by briefly touching upon some related results.
First, recall that by the Ekeland--Lasry theorem, \cite{EL}, a convex
hypersurface $M^{2n-1}\subset \R^{2n}$ enclosing a sphere
$S^{2n-1}_{R}$ and enclosed by the sphere $S^{2n-1}_{R'}$ with
$R'=\sqrt{2}R$ carries at least $r=n$ closed characteristics.  (Here
we say that $M$ encloses $M'$ when $M'$ lies in the open domain
bounded by $M$.)  In fact, there is a similar lower bound for any $R'$
with $r=\lceil n/\kappa\rceil$, where $\kappa\in\N$ is the smallest
positive integer such that $R'<\sqrt{\kappa+1}R$, \cite{AmMa}. It
would be interesting to cast the Ekeland--Lasry theorem into the
symplectic-topological framework (see, e.g., \cite{AGH,Ke:EL}) and, for
instance, extend it to the dynamically convex hypersurfaces. What
follows is an obvious observation along these lines, cf.\
\cite{Gut15}. 

\begin{Corollary}
\label{prop:EL}
Let $(M^{2n-1},\alpha)\subset \R^{2n}$ be a contact type hypersurface
bounding a simply connected Liouville domain, enclosing a sphere
$S^{2n-1}_{R}$ and enclosed by the sphere $S^{2n-1}_{R'}$ with
$R'=\sqrt{2}R$. Assume that
\begin{equation}
\label{eq:CE}
\min \CS(\alpha)\geq \pi R^2.
\end{equation}
Then $M$ carries at least $r=n$ closed characteristics.
\end{Corollary}

The corollary is an immediate consequence of Corollary
\ref{cor:sphere}, and a similar argument can also be used in the more
general setting from \cite{AmMa}. The proof of Corollary \ref{prop:EL}
ultimately relies on Theorem \ref{thm:LS-SH} in the same way as the
proof of the Ekeland--Lasry theorem utilized Lusternik--Schnirelmann
theory for convex Hamiltonians, \cite{Ek}.  Not surprisingly,
Corollary \ref{prop:EL} readily implies the Ekeland--Lasry theorem;
for \eqref{eq:CE} is satisfied for convex hypersurfaces by the
Croke--Weinstein theorem, \cite{CW}. In fact, \eqref{eq:CE} for convex
hypersurfaces is the assertion of the Croke--Weinstein theorem. Its
proof, while non-trivial, is self-contained and does not use Morse or
Lusternik--Schnirelmann theory. This lower bound fails easily for
star-shaped hypersurfaces meeting any prescribed pinching condition,
but not dynamically convex; cf.\ \cite[Sect.\ 3.5]{HZ}.  Note also
that in the corollary one of the two ``enclosures'' need not be
strict. For instance, it is sufficient to assume that $S_R^{2n-1}$
lies in the closed domain bounded by $M$ while $M$ is in the open
domain bounded by $S_{R'}^{2n-1}$, and the other way around.

Finally, recall that, as is proved in \cite{LLZ} (see also \cite{Lo}),
a convex hypersurface in $\R^{2n}$, which is symmetric with respect to
the involution $x\mapsto -x$, necessarily carries at least $n$ closed
characteristics. One could expect a similar lower bound to also hold
for dynamically convex symmetric hypersurfaces.  However, it is not
clear how prove such a generalization in the present framework. This
is somewhat surprising because the results from \cite{DDE}, and more
generally from \cite{Arnaud} for $\Z_k$-symmetry, on the existence of
elliptic orbits on symmetric convex hypersurfaces in $\R^{2n}$, which
seem to be closely related to \cite{LLZ}, can be generalized to
symmetric dynamically convex hypersurfaces and to other classes of
Reeb flows on pre-quantization contact manifolds; see
\cite{AM:elliptic}.
  
\subsection{Reeb flows on $ST^*S^n$}
\label{sec:T^*S^n-2}
Our next goal is to extend some of the results from Section
\ref{sec:displ} to Liouville domains $W$ in $T^*S^n$ containing the
zero section $S^n$. Let $W$ be such a domain with smooth boundary
$M=\p W$. Thus $M$ is a restricted contact type hypersurface, equipped
with a contact form $\alpha$, in $T^*S^n$ enclosing $S^n$. For
instance, $M$ can be the unit cotangent bundle $ST^*S^n$ or, more
generally, the boundary of any compact fiberwise star-shaped domain.
Throughout this section we will assume that $n\geq 2$. When $n=2$ and
the hypersurface $M$ in $T^*S^2$ is three-dimensional, much more
general results are available; see \cite{CGH,GGo} and also \cite{BL}
for the case of a Finsler metric on $S^2$.

In this setting the right analog of the dynamical convexity condition
is the requirement that $\mu_-(y)\geq n-1$ for all
$y\in\PP(\alpha)$. When $W$ is the unit disk bundle of a Finsler
metric, this requirement is satisfied, for instance, if the metric
meets certain curvature pinching conditions; see, e.g.,
\cite{Ra:pinching,Wa12} and also \cite{AM:elliptic,DLW} for further
references and \cite{HP} for the case of $n=2$. Similarly to Theorem
\ref{thm:mult2}, we have the following result.

\begin{Theorem}
\label{thm:mult4}
Let $(M^{2n-1},\alpha)$ be a closed contact type hypersurface in
$T^*S^n$ enclosing the zero section and bounding a simply connected
Liouville domain. Assume that $\mu_-(y)\geq n-1$ for every closed
characteristic $y$ on $M$.  Then $M$ carries at least $r$ simple
closed characteristics, where $r=\lfloor n/2\rfloor-1$.  When $\alpha$
is non-degenerate, we can take $r=n$ if $n$ is even and $r=n+1$ if $n$
is odd.
\end{Theorem}

This is Theorem \ref{thm:mult4-intro} from the introduction. The
non-degenerate case of the theorem is not new and included only for
the sake of completeness. A much more general result is proved in
\cite{AM}.  However, the proof below is self-contained and relatively
simple.  We also emphasize that in this case the lower bound is sharp
as the Katok--Ziller examples show; see \cite{Ka,Zi}.

\begin{proof}
  The general case of the theorem is established similarly to that of
  Theorem \ref{thm:mult2}. Assume that $(M,\alpha)$ carries only a
  finite number of closed characteristics and denote them by
  $x_1,\ldots,x_r$. For $j\in\N$, consider the range of indices
  $$
\CI_j=[-(n-1),\,n-1]+d_j=[d_j-(n-1),\,d_j+(n-1)]
$$ 
centered at $d_j=2j(n-1)$. By Corollary \ref{cor:STSn}, for every
$j\in\N$ and $d\in\CI_j$ there exists
$y_d=x_{i}^{k_{ij}}\in \PP(\alpha)$ with $\SH_{d}^G(y_{d})\neq 0$ and
$\s_d(\alpha)=\CA_\alpha(y_d)$, where $i$ depends on $j$ and
$d$. (Here it is more convenient again to relabel the orbits by
$\CI_j$.)  Hence, in particular, $|\hmu(y_{ij})-d|\leq n-1$. By
Theorem \ref{thm:IRT2}, when $\mu_-(y)\geq n-1$ for all
$y\in\PP(\alpha)$, the number of simple closed characteristics is
bounded from below by $\#\big[\big(d_j,\, d_j+n-1\big)\cap\CI\big]$,
i.e., by the number $r$ of the integers of the same parity as $n-1$ in
the open interval $(d_j,\, d_j+n-1)$, where $d_j$ can be taken of the
form $2j(n-1)$ for some sequence $j\to\infty$. Thus
$r=\lfloor n/2\rfloor-1$.

Let us now turn to the non-degenerate case. We use Proposition
\ref{prop:quotient-SH}. Assume that $M$ carries only finitely many
periodic orbits and denote these orbits by $x_i$. Let as above
$d_j=2j(n-1)$ and
$$
L=\big(d_j-(n-1),\, d_j+(n-1)\big).
$$
By Theorem \ref{thm:IRT1}, for some sequence $j\to\infty$ there exist
$k_{ij}\in 2\N$ such that $\mu(x_i^{k_{ij}})$ is in the closed
interval $\bar{L}=[d_j-(n-1),\, d_j+(n-1)]$ while
$\mu(x_i^{k_{ij}\pm \ell})$ for all $\ell\geq 1$ is outside
$L$. Hence, the number of even simple periodic orbits with
Conley--Zehnder index of the same parity as $n-1$ is bounded from
below by
\begin{equation}
\label{eq:r0}
r_0=\sum_i\dim \SH_i^{G,+}(W; \Q),
\end{equation}
where $i\equiv n-1\pmod 2$ is in $L$. The homology $\SH_i^{G,+}(W;\Q)$
is well known and calculated, for instance, in Proposition
\ref{prop:D-STS}. For the relevant range of degrees, the homology is
one-dimensional in every degree $i$ of the same parity as $n-1$,
except for $i=2j(n-1)$ if $n-1$ is even where the homology is
two-dimensional. When $i\equiv n\pmod 2$, the homology is zero. Now it
is easy to see that $r_0=r-2$.

Thus, to complete the proof, it suffices to find two more simple
periodic orbits.

Consider good periodic orbits of index $d_\pm=d_j\pm (n-1)$. There are
at least two such orbits for both $d_-$ and $d_+$ because
$\dim \SH_{d_\pm}^{G,+}(W; \Q)=2$. These orbits are either of the form
$x_i^{k_{ij}}$ (the first type) or $x_i^{k_{ij}\pm\ell}$ for some
$\ell\geq 1$ (the second type). If an orbit is of the first type, the
simple orbit $x_i$ does not contribute to \eqref{eq:r0}. Thus we can
assume that for at least in one of the degrees $d_\pm$ all orbits have
second type, for otherwise the proof is finished. (The same orbit
$x_i^{k_{ij}}$ cannot contribute simultaneously to $d_-$ and $d_+$.)
Observe that, by Theorem \ref{thm:IRT1}, $\mu(x_i^{k_{ij}-\ell})=d_-$
if and only if $\mu(x_i^{k_{ij}+\ell})=d_+$ and if and only if
$\mu(x_i^\ell)=n-1$. In fact, by Lemma \ref{lemma:DC}, $\mu(x_i^\ell)$
is a non-decreasing function of $\ell$, and hence
$\mu(x_i^{\ell'})=n-1$ for all $1\leq\ell'\leq \ell$. As a
consequence, there are at least two good periodic orbits of index
$n-1$.

Next, note that if $M$ carries at least three, not necessarily simple,
good periodic orbits of index $n-1$, it must also carry at least two
good periodic orbits of index $n$ since
$\dim\SH_{n-1}^{G,+}(W; \Q)=1$.  Under the conditions of the theorem,
every good periodic orbit of index $n$ is necessarily simple. Thus, we
have two extra orbits, not accounted for in \eqref{eq:r0} when there
are three or more orbits of index $n-1$.

If $M$ carries only one closed orbit of index $n-1$, say $x_1$, then
the only orbits of the second type are $x_1^{k_{1j}\pm 1}$.
Therefore, for both degrees $d_-$ and $d_+$ there must be at least one
orbit of the first type. This gives us two extra simple periodic
orbits.

Focusing on the remaining case where there are exactly two good closed
orbits (not necessarily simple) of index $n-1$, we can in addition
assume that there is exactly one good closed orbit, say $x_2$, of
index $n$. (For otherwise there are extra two simple periodic orbits.)
As has been pointed out above, the orbit $x_2$ is simple. Furthermore,
we can also assume that there are no orbits of the first type because
such an orbit together with $x_2$ would give us the required two
simple orbits.

As a consequence, we have exactly two good periodic orbits of index
$d_-$ and $d_+$. It is not hard to see that if there is a good orbit
$y$ with $\mu(y)\in L$ and such that $\mu(y)\equiv n\pmod 2$, there
must be an extra simple orbit $x_i$ such that $x_i^{k_{ij}}$
``cancels'' $y$. This is the case, for instance, when $x_2$ is
even. Indeed, then we can take $y=x_2^{k_{2j}}$.

When $x_2$ is odd, $x_2^{k_{2j}}$ is bad and the above argument does
not apply. To recover an extra simple closed orbit, observe first that
by our assumptions we have no good periodic orbits of index $d_-+1$,
two good periodic orbits of index $d_-$ and one good closed orbit
$x_2^{k_{2j}-1}$ of index $d_--1$. Since
$\dim\SH_{d_--2}^{G,+}(W; \Q)=1$, there are at least two good closed
orbits of index $d_--2$. Therefore, by Theorem \ref{thm:IRT1}, we have
two good periodic orbits of index $n-1$, one simple closed orbit $x_2$
of index $n$, and at least two good periodic orbits of index
$n+1$. Since $\dim\SH_{n+1}^{G,+}(W; \Q)=1$, there must then be at
least one good closed orbit $y$ of index $n+2$. Arguing as in the
proof of Lemma \ref{lemma:DC}, it is not hard to show that
$\mu(x_2^\ell)\geq n+3$ for $\ell>1$ since $\mu(x_2)=n$ and $x_2$ is
odd.  Hence $y\neq x_2^\ell$ for all $\ell\in\N$. Therefore, there
exists an extra simple closed orbit with index of the same parity as
$n$. This completes the proof of the non-degenerate case of the
theorem.
\end{proof}

\begin{Remark}
  One might expect that as in Theorem \ref{thm:mult2} the lower bound
  $r=\lfloor n/2\rfloor-1$ could be improved by one when $n-1$ is even
  by showing that either the spectral invariant corresponding to the
  lower limit $d_j$ of the range of degrees is carried by one of the
  orbits $x_{i}^{k_{ij}}$ or there exists an SDM orbit and hence
  infinitely many simple periodic orbits, \cite{GHHM}. However, in the
  latter case the SDM orbit need not be simple and the result from
  \cite{GHHM} does not apply.
\end{Remark}

\begin{Theorem}
\label{thm:mult5}
Let $(M^{2n-1},\alpha)$ be a closed contact type hypersurface in
$T^*S^n$ enclosing the zero section and bounding a simply connected
Liouville domain. Assume that $\mu_-(y)\geq q$, where $0<q<n-1$, for
all, not necessarily simple closed Reeb orbits $y$ on $M$.  Then $M$
carries at least $r$ simple closed orbits,
where 
$$
r=
\begin{cases}
q-\lceil n/2 \rceil &\textrm{ when $n$ and $q$ are odd,}\\
q+1-\lceil n/2 \rceil & \textrm{otherwise.}
\end{cases}
$$
(If the right hand side is negative or zero the result is void.) When
$\alpha$ is non-degenerate, we can take
$$
r=
\begin{cases}
  q +1 & \textrm{when $n$ is odd or when $n$ and $q$ are both even,}\\
  q & \textrm{when $n$ is even and $q$ is odd.}
\end{cases}
$$
\end{Theorem}

\begin{Remark}
\label{rmk:q2}
In the degenerate case it is sufficient to assume only that
$\mu_-(x)\geq q$ when $\hmu(x)>0$ and when $\alpha$ is non-degenerate
such an assumption yields the lower bound $r-1$ rather than
$r$. Furthermore, it follows from the proof that in the non-degenerate
case $r-1$ orbits can be chosen even and with Conley--Zehnder index of
the same parity as $n-1$, just as in Theorem \ref{thm:mult3}. The
remaining orbit is either odd or has Conley--Zehnder index of the same
parity as $n$.  For bumpy Finsler metrics on $S^n$ stronger lower
bounds on the number of simple prime closed geodesics are established
under less restrictive, at least on the conceptual level, assumptions
in \cite{DLW} where the index conditions are also related to certain
curvature bounds; see also \cite{BTZ:JDG83,Wa} and references therein.
\end{Remark}

\begin{proof}
  By Corollary \ref{cor:STSn} and Theorem \ref{thm:IRT2}, we can take
  as $r$ the number of the integers of the same parity as $n-1$ in the
  open interval
$$
L=\big(2j(n-1)-q+(n-1),\, 2j(n-1)+q\big).
$$ 
Here $j$ is an unknown positive integer, which can be arbitrarily
large. However, $r$ is independent of $j$. This proves the general
case of the theorem.

Next, assume that $\alpha$ is non-degenerate. As in the proof of
Theorem \ref{thm:mult4}, we rely on Proposition
\ref{prop:quotient-SH}. Due to Theorem \ref{thm:IRT1}, the number of
simple periodic orbits is bounded from below by $r_0$ given by
\eqref{eq:r0}, where $i$ ranges over integers of the same parity as
$n-1$ in the open interval
$$
L=\big(2j(n-1)-q,\, 2j(n-1)+q\big) 
$$
and $W$ is the domain bounded by $M$ in $T^*S^n$. It is easy to see
that $r_0=q$ when $n$ is odd or when $n$ and $q$ are both even, and
$r_0=q-1$ otherwise.

Finally, arguing exactly as in the proof of Theorem \ref{thm:mult3},
one can show that $(M,\alpha)$ carries an extra simple periodic orbit
and hence $r=r_0+1$.  This proves the lower bound in the
non-degenerate case.
\end{proof}

\begin{Remark}
  In the results from this section, the condition that the domain $W$
  is simply connected can be relaxed or perhaps eliminated.  For
  instance, it would be sufficient to assume that $\p W$ has
  restricted contact type and $\pi_1(W)$ is torsion free.
  \end{Remark}


\begin{thebibliography}{CKRTZ}

\bibitem[Ab]{Ab}
A. Abbondandolo, 
\emph{Morse Theory for Hamiltonian Systems}, Research Notes in
Mathematics 425, Chapman \& Hall/CRC, Boca Raton, FL, 2001.

\bibitem[AS]{AS}
A. Abbondandolo, M. Schwarz, 
Floer homology of cotangent bundles and the loop product,
\emph{Geom.\ Topol.}, \textbf{14} (2010), 
1569--1722. 

\bibitem[AM14]{AM:elliptic} 
M. Abreu, L. Macarini,
Dynamical convexity and elliptic periodic orbits for Reeb flows,
Preprint arXiv:1411.2543.

\bibitem[AM15]{AM} 
M. Abreu, L. Macarini, 
Multiplicity of periodic orbits for dynamically convex contact forms,
Preprint arXiv:1509.08441.

\bibitem[AGH]{AGH}
P. Albers, J. Gutt, D. Hein, Periodic Reeb orbits on prequantization
bundles, Preprint arXiv:1612.02205.

\bibitem[AmMa]{AmMa} 
A. Ambrosetti, G. Mancini,  
On a theorem by Ekeland and Lasry concerning the number of periodic
Hamiltonian trajectories, \emph{J. Differential Equations},
\textbf{43} (1982), 
249--256.

\bibitem[Ar]{Ar} 
V.I. Arnold, 
\emph{Mathematical Methods of Classical Mechanics}, Springer-Verlag,
1989.

\bibitem[AG]{AG} 
V.I. Arnold, A.B. Givental, 
Symplectic geometry, in \emph{Dynamical systems, IV}, Encyclopaedia
Math.\ Sci., \textbf{4}, 1--138, Springer, Berlin, 2001.

\bibitem[Arn]{Arnaud} 
M.-C. Arnaud, 
Existence d'orbites p\'eriodiques compl\'etement elliptiques des
hamiltoniens convexes pr\'esentant certaines sym\'etries,
\emph{C. R. Acad.\ Sci.\ Paris S\'er.\ I Math.}, \textbf{328} (1999),
1035--1038, 

\bibitem[BTZ]{BTZ:JDG83}
W. Ballmann, G. Thorbergsson, W. Ziller, Existence of closed geodesics
on positively curved manifolds, \emph{J. Differential Geom.},
\textbf{18} (1983), 
221--252.

\bibitem[BL]{BL}
V. Bangert, Y. Long, 
The existence of two closed geodesics on every Finsler 2-sphere,
\emph{Math.\ Ann.}, \textbf{346} (2010), 335--366. 


\bibitem[BG]{BG} 
J. Barge, E. Ghys, 
Cocycles d'Euler et de Maslov, \emph{Math. Ann.}, \textbf{294} (1992),
235--265.

\bibitem[Ba]{Ba}
M. Bator\'eo, 
On the rigidity of the coisotropic Maslov index on certain rational
symplectic manifolds, \emph{Geom.\ Dedicata}, \textbf{165} (2013),
135--156.

\bibitem[Bo]{Bo:thesis} 
F. Bourgeois, 
\emph{A Morse--Bott Approach to Contact Homology}, Ph.D. dissertation,
Stanford University, Stanford, Calif., 2002.

\bibitem[BO09a]{BO:Duke} 
F. Bourgeois, A. Oancea, 
Symplectic homology, autonomous Hamiltonians, and Morse-Bott moduli
spaces, \emph{Duke Math.\ J.}, \textbf{146} (2009), 
71--174.

\bibitem[BO09b]{BO:Exact} 
F. Bourgeois, A. Oancea,
An exact sequence for contact- and symplectic homology, \emph{Invent.\
  Math.}, \textbf{175} (2009), 
611--680.

\bibitem[BO10]{BO:Trans} F. Bourgeois, A. Oancea,
Fredholm theory and transversality for the parametrized and for the
$S^1$-invariant symplectic action, \emph{J. Eur.\ Math.\ Soc.\ (JEMS)},
12 (2010), 
1181--1229.

\bibitem[BO12]{BO:12} 
F. Bourgeois, A. Oancea,
$S^1$-equivariant symplectic homology and linearized contact homology,
Preprint arXiv:1212.3731.


\bibitem[BO13a]{BO:index} 
F. Bourgeois, A. Oancea, 
The index of Floer moduli problems for parametrized action
functionals, \emph{Geom.\ Dedicata}, \textbf{165} (2013), 5--24.


\bibitem[BO13b]{BO:Gysin} 
F. Bourgeois, A. Oancea, 
The Gysin exact sequence for $S^1$-equivariant symplectic homology,
\emph{J. Topol.\ Anal.}, \textbf{5} (2013), 
361--407.

\bibitem[Ca]{Ca} 
J.W.S. Cassels, 
\emph{An Introduction to Diophantine Approximation}, Cambridge
Tracts in Mathematics and Mathematical Physics, No.\ 45. Cambridge
University Press, New York, 1957.

\bibitem[CGG]{CGG} 
M. Chance, V.L. Ginzburg, B.Z. G\"urel,
Action-index relations for perfect Hamiltonian diffeomorphisms,
\emph{J. Sympl.\ Geom.}, \textbf{11} (2013), 449--474.

\bibitem[CFO]{CFO} K. Cieliebak, U. Frauenfelder, A. Oancea,
  Rabinowitz Floer homology and symplectic homology, \emph{Ann.\ Sci.\
    \'Ec.\ Norm.\ Sup\'er. (4)}, \textbf{43} (2010),
  957--1015. 

\bibitem[CH]{CH} V. Colin, K. Honda, Reeb vector fields and open book
  decompositions, \emph{J. Eur.\ Math.\ Soc.\ (JEMS)}, \textbf{15} (2013), 
    443--507.

\bibitem[CGH]{CGH}
D. Cristofaro-Gardiner, M. Hutchings,
From one Reeb orbit to two, Preprint arXiv:1202.4839.

\bibitem[CW]{CW}
C.B. Croke, A. Weinstein, 
Closed curves on convex hypersurfaces and periods of nonlinear oscillations,
\emph{Invent.\ Math.}, \textbf{64} (1981), 199--202.   

\bibitem[DDE]{DDE}
G. Dell'Antonio, B. D'Onofrio, I. Ekeland, 
Periodic solutions of elliptic type for strongly nonlinear Hamiltonian
systems, in \emph{The Floer Memorial Volume}, 327--333, 
Progr.\ Math., 133, Birkh\"auser, Basel, 1995. 

\bibitem[DLW]{DLW}
H. Duan, Y. Long, W. Wang,
The enhanced common index jump theorem for symplectic paths and
non-hyperbolic closed geodesics on Finsler manifolds, Preprint
arXiv:1510.02872.

\bibitem[DL$^2$W]{DLLW}
H. Duan, H. Liu, Y. Long, W. Wang,
Non-hyperbolic closed characteristics on non-degenerate star-shaped
hypersurfaces in $\R^{2n}$, Preprint arXiv:1510.08648.

\bibitem[Ek]{Ek}
I. Ekeland, 
\emph{Convexity Methods in Hamiltonian Mechanics},  Springer-Verlag,
Berlin, 1990. 

\bibitem[EH87]{EH:RR}
I. Ekeland, H. Hofer,  
Convex Hamiltonian energy surfaces and their periodic trajectories,
\emph{Comm.\ Math.\ Phys.}, \textbf{113} (1987), 419--469.

\bibitem[EH89]{EH}
I. Ekeland, H. Hofer, 
Symplectic topology and Hamiltonian dynamics, 
\emph{Math.\ Z.}, \textbf{200} (1989), 
355--378. 

\bibitem[EL]{EL} I. Ekeland, J.-M. Lasry, On the number of periodic
trajectories for a Hamiltonian flow on a convex energy surface,
\emph{Ann.\ of Math. (2)}, \textbf{112} (1980), 
283--319.

\bibitem[Es]{Es}
J. Espina, 
On the mean Euler characteristic of contact structures, \emph{Int.\
J.\ Math.}, \textbf{25} (2014), doi: 10.1142/S0129167X14500463.

\bibitem[FR]{FR}
E. Fadell, P. Rabinowitz, P.: Generalized cohomological theories for Lie group
actions with an application to bifurcation questions for Hamiltonian systems,
\emph{Inv.\ Math.}, \textbf{45} (1978) 139--174.

\bibitem[Fl89a]{Fl1}
A. Floer, Cuplength estimates on Lagrangian intersections, 
\emph{Comm.\ Pure Appl.\ Math.}, \textbf{42} (1989), 335--356.

\bibitem[Fl89b]{Fl2}
A. Floer,
Symplectic fixed points and holomorphic spheres, \emph{Comm.\ Math.\
Phys.}, \textbf{120} (1989), 575--611.


\bibitem[Gi07]{Gi:coiso}
V.L. Ginzburg,
Coisotropic intersections, \emph{Duke Math. J.}, \textbf{140} (2007),
111--163.


\bibitem[Gi10]{Gi:CC}
V.L. Ginzburg,
The Conley conjecture, \emph{Ann.\ of Math.} (2), \textbf{172} (2010),
1127--1180.

\bibitem[Gi14]{Gi:sl}
V.L. Ginzburg, My contact homology shopping list, 
Preprint arXiv:1412.7999.

\bibitem[GGo]{GGo} V.L. Ginzburg, Y. G\"oren, Iterated index and the
  mean Euler characteristic, \emph{J. Topol.\ Anal.}, \textbf{7}
  (2015), 453--481.


\bibitem[GG09a]{GG:gaps}
V.L. Ginzburg, B.Z. G\"urel,
Action and index spectra and periodic orbits in Hamiltonian dynamics,
\emph{Geom.\ Topol.}, \textbf{13} (2009), 2745--2805.


\bibitem[GG09b]{GG:generic}
V.L. Ginzburg, B.Z. G\"urel,
On the generic existence of periodic orbits in Hamiltonian dynamics,
\emph{J. Mod.\ Dyn.}, \textbf{4}
(2009), 595--610. 

\bibitem[GG09c]{GG:wm} 
V.L. Ginzburg, B.Z. G\"urel, 
Periodic orbits of twisted geodesic flows and the Weinstein--Moser
theorem, \emph{Comment.\ Math.\ Helv.}, \textbf{84} (2009), 865--907.


\bibitem[GG10]{GG:gap} 
V.L. Ginzburg, B.Z. G\"urel, 
Local Floer homology and the action gap, \emph{J. Sympl.\ Geom.},
\textbf{8} (2010), 323--357.


\bibitem[GG14]{GG:hyperbolic}
V.L. Ginzburg, B.Z. G\"urel, 
Hyperbolic fixed points and periodic orbits of Hamiltonian
diffeomorphisms, \emph{Duke Math. J.}, \textbf{163} (2014), 565--590.

\bibitem[GG15]{GG:CC-survey} V.L. Ginzburg, B.Z. G\"urel,
The Conley conjecture and beyond, \emph{Arnold Math.\ J.}, \textbf{1}
(2015), 299--337. 


\bibitem[GGM]{GGM} 
V.L. Ginzburg, B.Z. G\"urel, L. Macarini,
On the Conley conjecture for Reeb flows, 
\emph{Internat.\ J. Math.},
 \textbf{26} (2015), 1550047 (22 pages); doi: 10.1142/S0129167X15500470.

\bibitem[GH$^2$M]{GHHM} 
V.L. Ginzburg, D. Hein, U.L. Hryniewicz, L. Macarini, 
Closed Reeb orbits on the sphere and symplectically degenerate maxima,
\emph{Acta Math.\ Vietnam.}, \textbf{38} (2013), 55--78


\bibitem[GK]{GiKe}
V.L. Ginzburg, E. Kerman, 
Homological resonances for Hamiltonian diffeomorphisms and Reeb flows,
\emph{Int.\ Math.\ Res.\ Not.\ IMRN}, \textbf{2010}, 53--68. 

\bibitem[GoHi]{GoHi} M. Goresky, N. Hingston, Loop products and closed
  geodesics, \emph{Duke Math.\ J.}, \textbf{150} (2009), 117--209.

\bibitem[Gr]{Gr} 
M.A. Grayson, 
Shortening embedded curves, \emph{Ann.\ of Math.} (2), \textbf{129}
(1989), 71--111.


\bibitem[GM]{GM} 
D. Gromoll, W. Meyer, 
Periodic geodesics on compact Riemannian manifolds,
\emph{J. Differential Geom.}, \textbf{3} (1969), 493--510.

\bibitem[GGK]{GGK}
V. Guillemin, V. Ginzburg, Y. Karshon, 
\emph{Moment Maps, Cobordisms, and Hamiltonian Group Actions},
Mathematical Surveys and Monographs, \textbf{98}.  American
Mathematical Society, Providence, RI, 2002.

\bibitem[G\"u13]{Gu:nc} 
B.Z. G\"urel, 
On non-contractible periodic orbits of Hamiltonian diffeomorphisms,
\emph{Bull.\ Lond.\ Math.\ Soc.}, \textbf{45} (2013), 1227--1234.


\bibitem[G\"u15]{Gu:pr} 
B.Z. G\"urel,
Perfect Reeb flows and action-index relations, \emph{Geom.\ Dedicata},
\textbf{174} (2015), 105--120.  

\bibitem[Gut14]{Gut}
J. Gutt, Generalized Conley--Zehnder index, \emph{Ann.\ Fac.\ Sci.\
  Toulouse Math.}, \textbf{23} (2014), 
907--932. 

\bibitem[Gut15]{Gut15}
J. Gutt, The positive equivariant symplectic homology as an invariant
for some contact manifolds, Preprint arXiv:1503.01443. 

\bibitem[GuKa]{GK}
J. Gutt, J. Kang,
On the minimal number of periodic orbits on some hypersurfaces in
$\mathbb{R}^{2n}$, Preprint  arXiv:1508.00166.

\bibitem[HP]{HP} A. Harris, G. Paternain, Dynamically convex Finsler
  metrics and $J$-holomorphic embedding of asymptotic cylinders,
  \emph{Ann.\ Global Anal.\ Geom.}, \textbf{34} (2008),
  115--134. 

\bibitem[Ho]{Ho}
H. Hofer, 
Lusternik--Schnirelmann theory for Lagrangian intersections,
\emph{Ann.\ Inst.\ H. Poincar\'e Anal.\ Non Lin\'eaire}, \textbf{5}
(1988), 465--499.


\bibitem[HS]{HS}
H. Hofer, D. Salamon, 
Floer homology and Novikov rings, in \emph{The Floer Memorial Volume},
483--524, Progr.\ Math., 133, Birkh\"auser, Basel, 1995.


\bibitem[HWZ]{HWZ:convex} 
H. Hofer, K. Wysocki, E. Zehnder, 
The dynamics on three-dimensional strictly convex energy surfaces,
\emph{Ann.\ of Math.} (2), \textbf{148} (1998), 197--289.


\bibitem[HZ]{HZ}
H. Hofer, E. Zehnder,
\emph{Symplectic Invariants and Hamiltonian Dynamics}, Birk\"auser
Verlag, Basel, 1994.

\bibitem[How12]{How} W. Howard,
Action selectors and the fixed point set of a Hamiltonian
diffeomorphism, Preprint arXiv:1211.0580.


\bibitem[How13]{How:thesis} W. Howard, \emph{The Monster Tower and
    Action Selectors}, Ph.D. Thesis -- University of California, Santa
  Cruz, 2013, 101 pp. 

\bibitem[HM]{HM} 
U. Hryniewicz, L. Macarini, 
Local contact homology and applications, \emph{J.\ Topol.\ Anal.},
\textbf{7} (2015), 167--238. 


\bibitem[Ka]{Ka} 
A.B. Katok, 
Ergodic perturbations of degenerate integrable Hamiltonian systems,
\emph{Izv.\ Akad.\ Nauk SSSR Ser.\ Mat.}, \textbf{37} (1973),
539--576.

\bibitem[Ke]{Ke:EL}
E. Kerman,
Rigid constellations of closed Reeb orbits, Preprint arXiv:1612.01073.

\bibitem[KvK]{KvK}
M. Kwon, O. van Koert,
Brieskorn manifolds in contact topology, Preprint arXiv:1310.0343.

\bibitem[LO]{LO} 
H.V. L\^{e}, K. Ono, 
Cup-length estimates for symplectic fixed points, in \emph{Contact and
  Symplectic Geometry (Cambridge, 1994)}, 268--295, Publ.\ Newton
Inst., 8, Cambridge Univ.\ Press, Cambridge, 1996.

\bibitem[Li]{Li}
W. Li, 
A module structure on the symplectic Floer cohomology,
\emph{Comm.\ Math.\ Phys.}, \textbf{211} (2000), 
137--151. 

\bibitem[LL]{LL} 
H. Liu, Y. Long, 
The existence of two closed characteristics on every compact
star-shaped hypersurface in $\R^4$, Preprint arXiv:1308.3904.

\bibitem[LLZ]{LLZ}
C.-G. Liu, Y. Long, C. Zhu, 
Multiplicity of closed characteristics on symmetric convex
hypersurfaces in $\R^{2n}$, \emph{Math.\ Ann.}, \textbf{323} (2002),
201--215. 
 
  \bibitem[Lo90]{Lo90} 
Y. Long, 
Maslov-type index, degenerate critical points, and asymptotically
linear Hamiltonian systems, \emph{Sci.\ China Ser.\ A}, \textbf{33}
(1990), 
1409--1419.

\bibitem[Lo97]{Lo97}
Y. Long,
A Maslov-type index theory for symplectic paths, 
\emph{Topol. Methods Nonlinear Anal.}, \textbf{10} (1997), 
47--78. 

\bibitem[Lo02]{Lo}
Y. Long,  
\emph{Index Theory for Symplectic Paths with Applications},
Birkh\"auser Verlag, Basel, 2002.


\bibitem[LLW]{LLW} 
Y. Long, H. Lui, W. Wang, 
Resonance identities for closed characteristics on compact star-shaped
hypersurfaces in $\R^{2n}$, \emph{J.\ Funct.\ Anal.}, \textbf{266}
(2014), 5598--5638. 


\bibitem[LZ]{LZ} 
Y. Long, C. Zhu, 
Closed characteristics on compact convex hypersurfaces in $\R^{2n}$,
\emph{Ann.\ of Math.} (2), \textbf{155} (2002), 317--368.

\bibitem[LS]{LS}
L. Lusternik, L. Schnirelmann, \emph{M\'ethods Topologiques dans les
  Probl\`emes Variationnels}, Hermann, Paris, 1937.


\bibitem[MS]{MS}
D. McDuff, D. Salamon,
\emph{J-holomorphic Curves and Symplectic Topology}, Colloquium
publications, vol.\ 52, AMS, Providence, RI, 2004.


\bibitem[McL]{McL} 
M. McLean, 
Local Floer homology and infinitely many simple Reeb orbits,
\emph{Algebr.\ Geom.\ Topol.}, \textbf{12} (2012), 1901--1923.


\bibitem[Po]{Poz} 
M. Po\'zniak,
Floer homology, Novikov rings and clean intersections, in
\emph{Northern California Symplectic Geometry Seminar}, 119--181,
Amer.\ Math.\ Soc.\ Transl.\ Ser.\ 2, \textbf{196}, AMS, Providence,
RI, 1999.


\bibitem[Ra94]{Ra:geod}
H.B. Rademacher,
On a generic property of geodesic flows, \emph{Math.\ Ann.},
\textbf{298} (1994), 101--116.

\bibitem[Ra04]{Ra:pinching} H.B. Rademacher, A sphere theorem for
  non-reversible Finsler metric, \emph{Math.\ Ann.}, \textbf{328}
  (2004), 373--387.

\bibitem[Ri]{Ri}
A. Ritter, 
Topological quantum field theory structure on symplectic cohomology,
\emph{J. Topol.}, \textbf{6} (2013), 391--489.  

\bibitem[RS]{RS}
J. Robbin, D. Salamon, The Maslov index for paths, \emph{Topology},
\textbf{32} (1993), 
827--844. 

\bibitem[Sa90]{Sa:MT} D.A. Salamon, Morse theory, the Conley index and
  Floer homology, \emph{Bull.\ London Math.\ Soc.}, \textbf{22} (1990),
113--140. 

\bibitem[Sa99]{Sa:notes}
D.A. Salamon,
Lectures on Floer homology, in \emph{Symplectic Geometry and
Topology}, IAS/Park City Math.\ Ser., vol.\ 7, Amer.\ Math.\ Soc.,
Providence, RI, 1999, 143--229.

\bibitem[SZ]{SZ}
D. Salamon, E. Zehnder,
Morse theory for periodic solutions of Hamiltonian systems and the
Maslov index, \emph{Comm.\ Pure Appl.\ Math.}, \textbf{45} (1992),
1303--1360.

\bibitem[Sc]{Sc}
M. Schwarz,
A quantum cup-length estimate for symplectic fixed points,
\emph{Invent.\ Math.}, \textbf{133} (1998),  353--397. 

\bibitem[Vi89]{Vi:RR}
C. Viterbo, 
Equivariant Morse theory for starshaped Hamiltonian systems,
\emph{Trans.\ Amer.\ Math.\ Soc.}, \textbf{311} (1989), 621--655. 

\bibitem[Vi95]{Vi:Floer}
C. Viterbo, 
The cup-product on the Thom--Smale--Witten complex, and Floer
cohomology, in \emph{The Floer Memorial Volume}, 609--625, 
Progr.\ Math., 133, Birkh\"auser, Basel, 1995. 

\bibitem[Vi97]{Vi97} C. Viterbo, Some remarks on Massey products, tied
  cohomology classes, and the Lusternik--Schnirelmann category,
  \emph{Duke Math.\ J.}, \textbf{86} (1997), 
  547--564.


\bibitem[Vi99]{Vi:GAFA} 
C. Viterbo, 
Functors and computations in Floer cohomology, I, \emph{Geom.\ Funct.\
Anal.}, \textbf{9} (1999), 985--1033.

\bibitem[Wa12]{Wa12} 
W. Wang, 
On a conjecture of Anosov, \emph{Adv.\ Math.}, \textbf{230} (2012),
1597--1617.  

\bibitem[Wa13]{Wa}
W. Wang, 
Closed characteristics on compact convex hypersurfaces in $\R^8$,
Preprint arXiv:1305.4680.

\bibitem[Wa16]{Wa16}
W. Wang,
Existence of closed characteristics on compact convex hypersurfaces in
$\R^{2n}$, \emph{Calc.\ Var.\ Partial Differential Equations},
\textbf{55} (2016), 1--25.



\bibitem[WHL]{WHL}
W. Wang, X. Hu, Y. Long, 
Resonance identity, stability, and multiplicity of closed
characteristics on compact convex hypersurfaces, \emph{Duke Math.\
  J.}, \textbf{139} (2007), 411--462. 


\bibitem[Wi]{Wi}
J. Williamson,
 On the algebraic problem concerning the normal forms of linear
dynamical systems, \emph{Amer.\ J. Math.}, \textbf{58} (1936), 
141--163.

\bibitem[Us]{Us}
M. Usher, Spectral numbers in Floer theories, 
\emph{Compos.\ Math.}, \textbf{144} (2008), 1581--1592.  

\bibitem[Zi]{Zi} 
W. Ziller, 
Geometry of the Katok examples, \emph{Ergodic Theory Dynam.\ Systems},
\textbf{3} (1983), 135--157.


\end{thebibliography}
\end{document}